\documentclass[a4paper,12pt]{article}
\pdfoutput=1
\usepackage{amsmath,amssymb, amsthm}
\usepackage[english]{babel}
\usepackage[a4paper,margin=2cm,footskip=0.8cm]{geometry}
\usepackage{cite}
\usepackage[latin1]{inputenc}
\usepackage{enumerate}
\usepackage{latexsym}
\usepackage{mathtools}
\usepackage{mathrsfs}
\usepackage{textcomp}
\usepackage{stmaryrd}
\usepackage[pdftex]{graphicx}
\usepackage[font=small]{caption}
\usepackage{subcaption}
\usepackage[overload]{empheq}
\usepackage{mdframed}
\usepackage{chngcntr}
\usepackage[usenames,dvipsnames,svgnames,table]{xcolor}
\usepackage{tikz}
\usepackage{pgfplots,pgfplotstable}
\usepackage{epstopdf}
\usepackage{bm}
\usepackage{titlesec}
\usepackage[titletoc,toc,title]{appendix}
\pgfplotsset{compat=newest,width=0.328\textwidth,height=0.55\textwidth,grid style={dotted},ticklabel style = {font=\tiny},x label style={font=\footnotesize, at={(axis description cs:0.5,-0.035)},anchor=north},y label style={font=\footnotesize, at={(axis description cs:-0.158,.5)},anchor=south}, title style={font=\small, at={(axis description cs:0.5,1.06)},anchor=north}, legend pos=south west, legend cell align=left,legend style = {fill=none,inner xsep=0.8pt,inner ysep=0.4pt,font=\tiny,at={(axis description cs:0,0)}}}

\usepackage[colorlinks=false]{hyperref}

\newcommand{\settable}[2]{
\pgfplotstableread{#2}{\datatable}
\pgfplotstablecreatecol[copy column from table={#1}{0}] {data} {\datatable}
}
\usetikzlibrary{backgrounds,shapes,arrows}

\newtheorem{theorem}{Theorem}
\newtheorem{lemma}[theorem]{Lemma}
\newtheorem{proposition}[theorem]{Proposition}
\newtheorem{corollary}[theorem]{Corollary}
\newtheorem{assumption}[theorem]{Assumption}

\newtheorem{remark}[theorem]{Remark}

\counterwithin{theorem}{section}
\counterwithin{equation}{section}
\counterwithin{figure}{section}
\renewcommand{\Re}{\operatorname{Re}}             % real part
\providecommand{\diam}{\operatorname{diam}}

\providecommand{\Bb}{{\boldsymbol{b}}}

\providecommand{\Bd}{{\boldsymbol{d}}}

\providecommand{\Bn}{{\boldsymbol{n}}}

\providecommand{\Bu}{{\boldsymbol{u}}}
\providecommand{\Bv}{{\boldsymbol{v}}}
\providecommand{\Bw}{{\boldsymbol{w}}}
\providecommand{\Bx}{{\boldsymbol{x}}}
\providecommand{\By}{{\boldsymbol{y}}}
\providecommand{\Bz}{{\boldsymbol{z}}}

\newcommand{\VF}{{\mathbf{F}}}

\newcommand{\VH}{{\mathbf{H}}}

\newcommand{\VL}{{\mathbf{L}}}

\newcommand{\Bnu}        {\boldsymbol{\nu}}

\providecommand{\bbC}{\mathbb{C}}

\providecommand{\bbE}{\mathbb{E}}

\providecommand{\bbI}{\mathbb{I}}

\providecommand{\bbN}{\mathbb{N}}

\providecommand{\bbR}{\mathbb{R}}

\newcommand{\pspace}{\mathcal{P}}
\newcommand{\psubspace}{\mathcal{P}^{\Gamma}}
\newcommand{\Var}{\text{Var}_{\mu_J}}
\newcommand{\dd}{\,\text{d}}
\newcommand{\din}{D_{in}}
\newcommand{\hdin}{\hat{D}_{in}}
\newcommand{\dout}{D_{out}}
\newcommand{\doutR}{D_{out,R_{out}}}
\newcommand{\hdout}{\hat{D}_{out}}
\newcommand{\hdoutR}{\hat{D}_{out,R_{out}}}
\newcommand{\dom}{D_{R_{out}}}
\newcommand{\hexp}{\beta}
\newcommand{\jac}{\text{D}}
\newcommand{\jacsl}{\emph{D}}

\newcommand{\refgrad}{\hat{\nabla}}
\newcommand{\uref}{\hat{u}}
\newcommand{\aref}{\hat{\alpha}}
\newcommand{\bref}{\hat{b}}
\newcommand{\cref}{\hat{\kappa}^2}
\newcommand{\fref}{\hat{f}}
\newcommand{\partialref}{\hat{\partial}}
\newcommand{\xref}{\hat{\Bx}}
\newcommand{\nref}{\hat{\Bn}}
\newcommand{\vref}{\hat{v}}
\newcommand{\gammaref}{\hat{\Gamma}}

\newcommand{\evmin}{\Lambda_{min}}
\newcommand{\evmax}{\Lambda_{max}}
\newcommand{\xreffix}{\hat{\Bx}_{\hat{\Gamma}}}
\newcommand{\greffix}{\hat{g}}
\newcommand{\freffix}{\hat{F}}
\newcommand{\omatrix}{M}
\newcommand{\diagmatrix}{\Lambda}

\newcommand{\xif}{\tilde{\Bx}_y}
\newcommand{\xifgrad}{\tilde{\nabla}_y}
\newcommand{\xifgradtg}{\tilde{\nabla}_{y,\tilde{\Gamma}}}
\newcommand{\partialxif}{\tilde{\partial}_y}
\newcommand{\partialxifi}[1]{\tilde{\partial}_{y,#1}}
\newcommand{\xiffix}{\tilde{\Bx}_{\tilde{\Gamma}}}
\newcommand{\uxif}{\tilde{u}_y}
\newcommand{\wxif}{\tilde{w}_y}
\newcommand{\vxif}{\tilde{v}_y}
\newcommand{\psixif}{\tilde{\psi}_y}
\newcommand{\nxif}{\tilde{\Bn}_y}
\newcommand{\fxif}{\tilde{F}}
\newcommand{\fxifs}{\tilde{f}}
\newcommand{\axif}{\tilde{\alpha}_y}

\newcommand{\gxif}{\tilde{g}}
\newcommand{\gammaxif}{\tilde{\Gamma}}
\newcommand{\Ckpwxi}[2]{C^{#1,\hexp}_{\tilde{pw}}({#2})}

\newcommand{\sholder}[3][\hexp]{\left|{#2}\right|_{{#1};{#3}}}
\newcommand{\sholderk}[4][\hexp]{\left|{#2}\right|_{{#4},{#1};{#3}}}
\newcommand{\holder}[3][\hexp]{\left\|{#2}\right\|_{{#1};{#3}}}
\newcommand{\holderk}[4][\hexp]{\left\|{#2}\right\|_{{#4},{#1};{#3}}}
\newcommand{\scont}[3][0]{\left|{#2}\right|_{{#1};{#3}}}
\newcommand{\cont}[3][0]{\left\|{#2}\right\|_{{#1};{#3}}}

\newcommand{\sholderpw}[3][\hexp]{\left|{#2^{\pm}}\right|_{{#1};{#3}^{\pm}}}
\newcommand{\sholderkpw}[4][\hexp]{\left|{#2^{\pm}}\right|_{{#4},{#1};{#3}^{\pm}}}
\newcommand{\holderpw}[3][\hexp]{\left\|{#2^{\pm}}\right\|_{{#1};{#3}^{\pm}}}
\newcommand{\holderkpw}[4][\hexp]{\left\|{#2^{\pm}}\right\|_{{#4},{#1};{#3}^{\pm}}}
\newcommand{\scontpw}[3][0]{\left|{#2^{\pm}}\right|_{{#1};{#3}^{\pm}}}
\newcommand{\contpw}[3][0]{\left\|{#2^{\pm}}\right\|_{{#1};{#3}^{\pm}}}

\newcommand{\holderkpwb}[3]{\left\|{#1^{\pm}}\right\|_{{#3},{\hexp};B^{\pm}_{#2}(\xiffix)}}
\newcommand{\contpwb}[3][0]{\left\|{#2^{\pm}}\right\|_{{#1};B^{\pm}_{#3}(\xiffix)}}
\newcommand{\sholderpwb}[2]{\left|{#1^{\pm}}\right|_{{\hexp};B^{\pm}_{#2}(\xiffix)}}
\newcommand{\sholderkpwb}[3]{\left|{#1^{\pm}}\right|_{{#3},\hexp;B^{\pm}_{#2}(\xiffix)}}

\newcommand{\holderpwbr}[2]{\left\|{#1^{\pm}}\right\|_{{\hexp};B^{\pm}_{#2}}}
\newcommand{\holderkpwbr}[3]{\left\|{#1^{\pm}}\right\|_{{#3},{\hexp};B^{\pm}_{#2}}}
\newcommand{\contpwbr}[3][0]{\left\|{#2^{\pm}}\right\|_{{#1};B^{\pm}_{#3}}}
\newcommand{\sholderpwbr}[2]{\left|{#1^{\pm}}\right|_{{\hexp};B^{\pm}_{#2}}}
\newcommand{\sholderkpwbr}[3]{\left|{#1^{\pm}}\right|_{{#3},\hexp;B^{\pm}_{#2}}}
\newcommand{\scontpwbr}[3][0]{\left|{#2^{\pm}}\right|_{{#1};B^{\pm}_{#3}}}

\newcommand{\dz}{\dd\Bz}

\newtheoremstyle{dotless}{}{}{\itshape}{}{\bfseries}{}{ }{}
\theoremstyle{dotless}
\newtheorem*{notation}{Notation}

\title {Multilevel Monte Carlo on a high-dimensional parameter space for transmission problems with geometric uncertainties}
\author{
    Laura~Scarabosio\footnote{\textsc{Technical University of Munich, Germany, Department of Mathematics, Institute for Numerical Mathematics,} Boltzmannstrasse 3, 85748 Garching b. M\"unchen. \textsl{Email}: scarabos\symbol{64}ma.tum.de}}

\begin{document}

\maketitle

\begin{abstract}
 In the framework of uncertainty quantification, we consider a quantity of interest which depends non-smoothly on the high-dimensional parameter representing the uncertainty. We show that, in this situation, the multilevel Monte Carlo algorithm is a valid option to compute moments of the quantity of interest (here we focus on the expectation), as it allows to bypass the precise location of discontinuities in the parameter space. We illustrate how such lack of smoothness occurs for the point evaluation of the solution to a (Helmholtz) transmission problem with uncertain interface, if the point can be crossed by the interface for some realizations. For this case, we provide a space regularity analysis for the solution, in order to state converge results in the $L^{\infty}$-norm for the finite element discretization. The latter are then used to determine the optimal distribution of samples among the Monte Carlo levels. Particular emphasis is given on the robustness of our estimates with respect to the dimension of the parameter space.
\end{abstract}

\small
\textbf{Keywords:} multilevel Monte Carlo, shape uncertainty, interface problem, $L^{\infty}$-estimates, uncertainty quantification.

\normalsize
\section{Introduction}
In many engineering applications, the behavior of a physical system depends on a parameter vector $\By$ belonging to a high-dimensional parameter space $\pspace_J\subseteq \bbR^J$ with $J\in\bbN$ large. The vector $\By\in\pspace_J$ may represent, for instance, random variations in material or geometrical properties of the physical system. Equipping $\pspace_J$ with a $\sigma$-algebra $\mathcal{A}_J$ and a probability measure $\mu_J$, we obtain the probability space $(\pspace_J,\mathcal{A}_J,\mu_J)$.  In such cases, it is of interest to compute statistics, with respect to the parameter, of a quantity   $q(\By;u(\By))$ (quantity of interest, Q.o.I. for short) depending on the solution $u$ to a partial differential equation (PDE):
\begin{equation}
\begin{split}
 \text{Find } u \text{ s.t.}: \quad
&  u(\By)\in \mathcal{X} \quad\text{for every } \By\in\pspace_J,\\
& \mathcal{D}(\By;u(\By))=0\quad \text{for every } \By\in\pspace_J.
\end{split}\label{eq:diffeq}
\end{equation}
In \eqref{eq:diffeq}, $\mathcal{X}$ denotes a separable Banach space. For every $\By\in\pspace_J$ and every $J\in\bbN$, $q(\By;\cdot):\mathcal{X}\rightarrow\mathcal{Y}$, that is, every realization of $q$ belongs to a separable Hilbert space $\mathcal{Y}$. For instance, if $q$ is the solution $u$ itself, then $\mathcal{Y}=\mathcal{X}$, if $q$ is some linear output functional, then $\mathcal{Y}=\bbR$ or $\mathcal{Y}=\bbC$. 

Introducing the quantity $Q: \pspace_J\rightarrow \mathcal{Y}$ such that $Q(\By):=q(\By;u(\By))$ for every $\By\in\pspace_J$, the present work focuses on the case when $Q$ is \emph{non-smooth}, with respect to the high-dimensional parameter $\By$, across a submanifold $\psubspace_J\subset \pspace_J$ which is not easy to track. Here, by non-smooth we mean `not analytic in $\By$', and in our treatment we allow $Q$ to have jumps across $\psubspace_J$. More precisely, in this paper \eqref{eq:diffeq} is a (acoustic) transmission problem, where the shape of the scatterer is subject to random variations modeled by the high-dimensional parameter $\By$, and the Q.o.I. is the point evaluation of the solution in locations that, depending on the realization, may be either inside or outside the scatterer. We focus on the computation of the mean
\begin{equation}\label{eq:meanintro}
\bbE_{\mu_J}[Q]:=\int_{\pspace_J}Q(\By)\dd\mu_J(\By),
\end{equation}
and aim at numerical methods which are robust with respect to the dimension $J$ of the parameter space, that is, whose convergence rates \emph{do not deteriorate for large $J$}, possibly tending to infinity.

\smallskip
%\textbf{Related work in shape uncertainty quantification.} 
\textbf{Related work.} We first review the literature on the computation of moments of a Q.o.I., and then, in view of our application to a transmission problem with random interface, the literature in shape uncertainty quantification.

If the randomness in the system consists of deviations from a deterministic quantity that are small enough, it is possible to apply a perturbation approach \cite{DW}, and approximate moments of the Q.o.I. exploiting its Taylor expansion centered at the deterministic quantity. Otherwise, we have to compute \eqref{eq:meanintro} directly (or analogous expression for higher order moments), which means employing quadrature formulas on the parameter space. In this work we focus on this second option. Quadrature rules on a (high-dimensional) parameter space can be classified in two main cathegories: Monte Carlo-like rules and deterministic rules. As with quadrature rules for functions of one real variable, there is a compromise between speed of convergence with respect to the number of function evaluations and smoothness required on the integrand. The Monte Carlo approach to compute \eqref{eq:meanintro}, consisting of random sampling \cite{Caf}, converges almost surely to the exact mean provided the Q.o.I. is Lebesgue integrable with respect to the parameter. This is ensured by the strong law of large numbers \cite[Sect. 2]{Caf}. If the Q.o.I. has also finite variance, then the Monte Carlo quadrature converges with rate $M^{-1/2}$, where $M$ is the number of samples \cite[Thm. 2.1]{Caf}. The high computational effort due to the slow convergence rate can be reduced using the multilevel Monte Carlo (MLMC) method \cite{He98,HeS,Giles08,Giles15} or other variance reduction techniques \cite{Glas}. To converge, MLMC requires square integrability of the Q.o.I., and details are provided in Section \ref{sect:mlmcintro} of this paper. Deterministic quadrature rules comprise quasi-Monte Carlo (QMC) methods and spectral methods.
 We refer to \cite{DKS} for a comprehensive treatment of QMC. It is possible to construct QMC sequences of quadrature points such that the speed of convergence is $M^{1-\varepsilon}$, for any $\epsilon>0$ (with $M$ the number of quadrature points) \cite[Prop. 2.18, Thm. 3.20 and Sect. 3.4]{DP}, under the assumption that the integrand has continuous first order mixed derivatives. If the integrand has higher regularity, then higher order QMC quadrature rules can be constructed, with convergence rates that are robust with respect to the dimension of the parameter space \cite{DKLNS}. Spectral methods can be divided in stochastic Galerkin \cite{BTZ,SG,XK} and stochastic collocation \cite{BNT,NTW} approaches. They provide high order convergence rates if the Q.o.I. admits an analytic extension to the complex plane: for finite-dimensional parameter spaces, the rate is exponential with respect to number of evaluation points, but it depends on the dimension and deteriorates as the latter increases \cite{BNT,BeNTT}; the dimension-independent convergence rate, which still holds in infinite-dimensional parameter spaces, is algebraic, and it depends only on the `sparsity class of the unknown' \cite{CCS,SG,SS}. If the Q.o.I. is not globally smooth with respect to the parameter, but it is piecewise smooth, then one possibility is to employ discontinuity detection methods (as, for instance, the one suggested in \cite{ZWGB}) to detect the surfaces of non-smoothness, and then apply a high order quadrature rule separately on each subdomain on which the Q.o.I. is smooth. However, this approach is not applicable for complicated surfaces of discontinuity. This issue is discussed in more details in subsection \ref{ssect:discussion} of this work, which then motivates why MLMC is a valid option when non-smoothness occurs across manifolds that are not easy to track.

In the model problem that we consider, the randomness stems from uncertain variations of the scatterer boundary. Several approaches are possible to tackle shape uncertainty quantification: perturbation techniques \cite{HaSS,HL,CherS} (analogous to \cite{DW} using shape calculus to construct the Taylor expansions), level set methods \cite{NSM,NCSM}, the fictitious domain approach \cite{CC} and the mapping technique \cite{TX,XT}. Recently, a new approach has been suggested in \cite{Har} in the framework of a Helmholtz scattering problem, where a boundary integral formulation is used to reconstruct the expansion of the solution in spherical or cylindrical harmonics; however, explicit formulas for the coefficients seem to be available, for the moment, only when the parameter space is low-dimensional. In our paper, we adopt the mapping technique, because it allows to deal with not small perturbations, it provides a natural way of resolving the interface for the spatial discretization \cite[Sect. 5]{HPS}, and, transforming a PDE on a random domain to a PDE on a deterministic domain with stochastic coefficients, it simplifies both theoretical analysis and practical implementation. 
%The numerical analysis of shape uncertainty quantification has seen a considerable development in the recent years. Specifically, 
The regularity of the solution to a PDE with respect to the high-dimensional parameter describing the shape variations has been studied in \cite{CNT,HPS,HSSS,CSZ} and \cite{JSSZ}. The authors of these papers prove holomorphic dependence, with respect to the high-dimensional parameter, of the solution on the nominal, deterministic domain introduced by the domain mapping. The work \cite{CNT} deals with an elliptic boundary value problem, \cite{HPS} tackles also an elliptic interface problem, and \cite{HSSS} treats the same Helmholtz transmission problem as the one addressed in the present paper (and considers also some linear output functionals). The paper \cite{CSZ} provides, in the framework of the stationary Navier-Stokes equations, a unified mathematical treatment of the mapping method, independent of the domain parametrization, and introduces the concept of `shape holomorphy'. The techniques presented in \cite{CSZ} have been applied, in \cite{JSSZ}, to the Maxwell equations in frequency domain. However, the smooth dependence on the parameter breaks down for point evaluations of the solution to an interface problem on the physical domain, where the interface changes for every realization \cite[Ch. 8]{LSthesis}. This is the case treated in this paper. For an application of the mapping technique to the inverse problem setting, we refer to \cite{GP} and \cite{HKMS}, where the inverse problem in electrical impedance tomography is considered. In particular, in \cite{GP} the authors prove the regularity of the posterior measure with respect to the high-dimensional parameter associated to the shape variations.

\smallskip
\textbf{Scope and outline of the paper.} One goal of this paper is to highlight the presence and impact of the non-smooth dependence on the stochastic parameter in the case of an important class of transmission problems with stochastic interface, namely Helmholtz transmission problems. We prove, and confirm by numerical experiments, that MLMC offers a robust treatment for such class of problems, allowing to bypass the precise location of discontinuities in the parameter space. The methodology used clearly conveys that MLMC is a viable approach also for other problems lacking smoothness with respect to the stochastic parameter. The second goal is to provide a full numerical analysis for point evaluation in (Helmholtz) transmission problems with geometric uncertainties, including the regularity of the solution with respect to the parameter and to the spatial coordinate, and their implications in the convergence of MLMC.

 The paper is organized as follows. In Section \ref{sect:modelpb} we introduce our model transmission problem. Sections \ref{sect:pointvalsmooth} and \ref{sect:pointreg} are the core of this paper. In Section \ref{sect:pointvalsmooth},  we show that the point value of the solution in locations that might be crossed by the random interface is a Q.o.I. which does not depend smoothly on the parameter describing the shape variations. The main contribution there is Proposition \ref{prop:conty}, where we state the regularity of the Q.o.I. with respect to the high-dimensional parameter. In the same section, we discuss possible ways to handle the non-smooth parameter dependence, and provide our motivation for choosing the MLMC method. The latter is reviewed in Section \ref{sect:mlmcintro}, with focus on Q.o.I.s depending on the solution of a partial differential equation. In Section \ref{sect:pointreg}, we first analize the space regularity of the solution to the model transmission problem, and then state convergence results for MLMC when using a finite element discretization. Finally, in Section \ref{sect:numexp}, we show numerical experiments matching the theoretical predictions. For ease of presentation, technical details on the space regularity of the solution, used in the proofs of Proposition \ref{prop:conty} and  Theorem \ref{thm:ureg}, have been moved to Appendices A and B, since they consist in the adaptation of already existing results to our Helmholtz transmission problem.

\section{A model transmission problem}\label{sect:modelpb}
As a model problem, we address the Helmholtz transmission problem in $\bbR^2$, describing the scattering of an incoming wave $u_i$ from a penetrable object whose shape is subject to random variations. We formally define $\Gamma(\By)$, $\By\in\pspace_J$, to be the boundary of the scatterer, and denote by $\din(\By)$ the domain enclosed inside $\Gamma(\By)$. We consider a circle of fixed radius $R_{out}$ containing all realizations of the scatterer in its interior, and indicate and by $\doutR(\By)$ the part of the outer, unbounded domain contained inside this circle. Finally, $D_{R_{out}}:=\din(\By)\cup \Gamma(\By) \cup \doutR(\By)$. Geometry and notation are clarified in Fig.~\ref{fig:domain}.

\begin{figure}[t]
\begin{center}
\fbox{\begin{tikzpicture}[scale=0.6]
\footnotesize
   \useasboundingbox (-4.4,-4.2) rectangle (4.4,4.2);
   \draw[dotted,thick] (0,0) circle[radius=1.6];
 \filldraw[fill=SkyBlue!70!white,draw=SkyBlue!40!Blue,domain=0:6.28,samples=100]  plot ({1.6*cos(\x r)*(1+0.05*cos(3*\x r)+0.04*cos(8*\x r)+0.015*cos(11*\x r)}, {1.6*sin(\x r)*(1+0.05*cos(3*\x r)+0.04*cos(8*\x r)+0.015*cos(11*\x r)});
 \node at (0,-0.25) {$0$};
 \draw[very thin] (0,0)--(1.3,1);
 \node at (0.5,0.5) {$r(\By;\varphi)$};
 \draw[very thin] (0,0)--(4,0);
 \node at (3.3,0.3) {$R_{out}$};
  \node at (-0.4,-0.9) {$\din(\By)$};
   \node at (0,2.7) {$\doutR(\By)$};
\node at (-0.3,-1.8) {$\Gamma(\By)$};
   \node at (3.5,3.2) {$\partial D_{R_{out}}$};
 \draw[thick,->] (-4,0)--(-2.2,0);
 \draw[semithick] (-2.9,-0.2)--(-2.9,0.2);
 \draw[semithick] (-2.8,-0.2)--(-2.8,0.2);
 \node at (-3.4,0.2) {$u_i$};
 \draw (0,0) circle[radius=4];
 \end{tikzpicture}}\caption{Geometry of our model problem.}\label{fig:domain}
 \end{center}
  \end{figure}
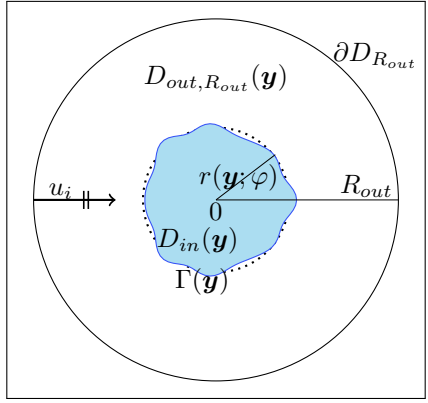

The transmission problem for the Helmholtz equation reads:
    \begin{subequations}
    \begin{align}[left=\empheqlbrace]
 &-\nabla\cdot\left(\alpha(\Gamma(\By);\Bx)\nabla u\right) - \kappa^2(\Gamma(\By);\Bx)u = 0 \quad \text{in }\din(\By)\cup\doutR(\By),\label{mdpb:eq}\\
 & {\llbracket u \rrbracket_{\Gamma(\By)}=0,\quad\llbracket \alpha(\Gamma(\By);\Bx)\nabla u \cdot \Bn \rrbracket_{\Gamma(\By)}=0},\label{mdpb:jump}\\
 &  \dfrac{\partial}{\partial \Bn_{out}} (u-u_i) = \operatorname{DtN}(u)-\operatorname{DtN}(u_i)\quad\text{on }\partial D_{R_{out}},\label{mdpb:dtn}\\
 & \text{for every }\By\in\pspace_J,\nonumber
    \end{align}\label{eq:modelpb}
    \end{subequations}
where we consider real-valued, piecewise-constant coefficients
\begin{equation}\label{eq:coeffs}
 \alpha(\Gamma(\By);\Bx) = \begin{cases}
                   1 &\mbox{if }\Bx\in \doutR(\By),\\
                   \alpha_2 &\mbox{if }\Bx \in \din(\By),
                  \end{cases}\qquad
 \kappa^2(\Gamma(\By);\Bx) = \begin{cases}
                   \kappa_1^2 &\mbox{if }\Bx\in \doutR(\By),\\
                   \alpha_2 \,\kappa_2^2 &\mbox{if }\Bx \in \din(\By).
                  \end{cases}
\end{equation}
We assume $u_i$ to be a plane wave, that is $u_i(\Bx)=e^{j\kappa_1\Bd\cdot\Bx}$, where $\Bd$ is a direction vector with $\rVert\Bd\rVert=1$ and $j=\sqrt{-1}$. The unknown $u=u(\By;\Bx)$ represents the total field, whereas $\kappa_1,\kappa_2>0$ denote the wavenumbers in free space and in the scatterer, respectively; $\alpha_2$ is a positive coefficient. In equation \eqref{mdpb:jump}, $\llbracket\cdot\rrbracket_{\Gamma(\By)}$ denotes the jump across the random interface $\Gamma(\By)$. Equation \eqref{mdpb:dtn} is the exact boundary condition on the disc of radius $R_{out}$, and corresponds to the radiation condition in free space (Sommerfeld radiation condition). Such boundary condition is stated in terms of the Dirichlet-to-Neumann map ($\operatorname{DtN}$) on the scattered wave, see \cite[Sect. 6.2.3]{Ned} for its definition.

We work in the \emph{large wavelength regime}, assuming the wavelength to be large enough compared to the size of the scatterer (see Assumption \ref{Ak1k2}). Mathematically, this means that we address the case when the bilinear form associated to \eqref{eq:modelpb} is coercive.

% The Helmholtz equation is sign indefinite, but we work in the regime where the bilinear form associated to \eqref{eq:modelpb} is coercive. This is ensured by the following \cite[Lemma 3.2.5]{LSthesis}:
% \begin{assumption}[Large wavelength assumption]\label{Ak1k2}
% The wavenumbers in \eqref{eq:coeffs} satisfy the condition:
%  \begin{equation}\label{eq:kub}
%   \kappa_1^2, \kappa_2^2 \leq \tau \,\inf_{w\in H^1(D_{R_{out}})}\frac{ |w|_{H^1(D_{R_{out}})}^2 + \left\|w\right\|_{L^2(\partial D_{R_{out}})}^2}{\left\|w\right\|_{L^2(D_{R_{out}})}^2},
%  \end{equation}
% with $\tau<\min\left\{\frac{1}{\alpha_2},\alpha_2\right\}$.
% \end{assumption}

We consider here an explicit description for the interface. We assume the scatterer to be star-shaped with respect to the origin, and set, in polar coordinates, $\Gamma(\By):=\{(r,\varphi): r=r(\By;\varphi),\varphi\in[0,2\pi)\}$, where $r$ is a stochastic, angle-dependent radius (see Fig.\ref{fig:domain}). We express the latter as:
\begin{equation}\label{eq:radiusy}
r(\By;\varphi) = r_0(\varphi) + \sum_{j=1}^{J} \beta_j y_j\psi_j(\varphi),\qquad \psi_j(\varphi)=\begin{cases}
                   \sin(\tfrac{j+1}{2}\varphi)& \text{for }j \text{ odd},\\
                   \cos(\tfrac{j}{2}\varphi)& \text{for }j\text{ even},
                  \end{cases}
\end{equation}
for every $\varphi\in [0,2\pi)$, $J\in\mathbb{N}$ and $\By=(y_1,\ldots,y_J)\in\pspace_J$. The quantity $r_0\in C^{k,\hexp}_{per}([0,2\pi))$, for some $k\geq 1$ and $0<\hexp<1$, is an approximation to the mean radius. The real parameters $\left\{y_j\right\}_{j=1}^{J}$ are the images of independent, identically distributed (i.i.d.) uniform random variables $Y_j \sim \mathcal{U}([-1,1])$, $1\leq j\leq J$. Thus, $\pspace_J=[-1,1]^J$ and $\mu_J$ is the product measure $\mu_J=\left(\tfrac{1}{2}\right)^{J}$. For every $J\in\bbN$, the radius \eqref{eq:radiusy} is a well-defined random variable on the closure of the subspace $\operatorname{span}\left\{1,\left(\psi_j\right)_{j=1}^J\right\}$ in the $C^{k,\hexp}_{per}([0,2\pi))$-norm. For the expansion \eqref{eq:radiusy}, we require:
 \begin{assumption}\label{Arcoeffs}
 The sequence $(\beta_j)_{j\geq 1}$ in \eqref{eq:radiusy} has a monotonic majorant in $\ell^p(\bbN)$ with $0<p<\frac{1}{2}$. Furthermore, $\sum_{j\geq 1}|\beta_j|\leq \frac{r_0^{-}}{2}$, with $r_0^{-}=  \emph{inf}_{\varphi\in[0,2\pi)}r_0(\varphi)>0$.
\end{assumption}
The bounds on $\sum_{j\geq 1}|\beta_j|$ and $r_0^{-}$ ensure positivity and boundedness of the radius for every realization. The condition on the decay of $(\beta_j)_{j\geq 1}$, instead, is a regularity assumption (with respect to $\varphi$) on the radius: the smaller the $p$, the smoother the radius \cite[Lemma 2.1.6]{LSthesis}. In particular, $p<\tfrac{1}{2}$ ensures that every realization of $r(\By;\cdot)\in C^{1,\hexp}(0,2\pi)$ for some $\hexp\in(0,1)$, with a $J$ and $\By$-independent bound (cf. proof of Lemma 2.1.6 in \cite{LSthesis}).
 
 \smallskip
 Our Q.o.I. is $Q(\By)=\Bu(\By):=\left\{u(\By;\Bx_i)\right\}_{i=0}^{N-1}\subset\bbC^N$, $\By\in\pspace_J$, the value of the solution $u$ to \eqref{eq:modelpb} at fixed points $\left\{\Bx_0,\ldots,\Bx_{N-1}\right\}\subset \bbR^2$. In particular, we are interested in the case that these evaluation points are close to the interface, so that they may lie on different sides of $\Gamma(\By)$ for different realizations of $\By$.
 
 \section{Point evaluation: non-smooth parameter dependence}\label{sect:pointvalsmooth}
 The aim of this section is to highlight the non-smooth dependence of $\Bu(\By)=\left\{u(\By;\Bx_i)\right\}_{i=0}^{N-1}$ on the high-dimensional parameter $\By\in\pspace_J$. To better explain the non-smooth behavior, we first consider, in subsection \ref{ssect:1dpv}, a one-dimensional transmission problem. Then, in subsection \ref{ssect:transmpv}, we move to the model problem introduced in the previous section. Due to the failure of  high order quadrature methods to compute $\bbE_{\mu_J}[\Bu]$, illustrated in subsection \ref{ssect:smolyak}, in subsection \ref{ssect:discussion} we discuss how this issue can be overcome, and motivate why we opt for MLMC.
 
 \subsection{A one-dimensional example}\label{ssect:1dpv}
 We consider the one-dimensional problem
 
 \begin{minipage}{0.46\textwidth}
  \begin{empheq}[left=\empheqlbrace]{align*}
 &-\left(\alpha(y,x)u'(y,x)\right)' = 0, \quad x\in (0,1),\nonumber\\
 & u(0)=1,\quad u(1)=0,\label{eq:1dpb}\\
 & \text{for every } y\in \left[\tfrac{1}{4},\tfrac{3}{4}\right],\nonumber
\end{empheq}
\end{minipage}\hfill
\begin{minipage}{0.53\textwidth}
\begin{equation}\mbox{with}\qquad\alpha(y,x) = \begin{cases}
                   \alpha_l &\mbox{if }x\in (0,y),\\
                   \alpha_r &\mbox{if }x \in (y,1),
                  \end{cases}\label{eq:1dpb}
                  \end{equation}
\end{minipage}

\medskip

\noindent where $'$ denotes the derivative with respect to $x$, and $\alpha_l,\alpha_r\in\mathbb{R}_{+}\setminus \left\{0\right\}$, $\alpha_l\neq\alpha_r$. The exact solution to \eqref{eq:1dpb} is 
\begin{equation*}
 u(y,x)=\begin{cases}
         -\frac{\alpha_r}{\alpha_l(1-y)+\alpha_r y}x+1 &\mbox{if }x\in (0,y),\\
         \frac{\alpha_l}{\alpha_l(1-y)+\alpha_r y}(1-x) &\mbox{if }x\in (y,1),
        \end{cases}
\end{equation*}
and presents a kink at the interface $y$. Consequently, the evaluation of the solution at a point $x_0\in (0,1)$ is only continuous (and in particular not analytic) as a function of $y$ if $x_0$ is a point that can be crossed by the interface, that is if $x_0\in \left[\tfrac{1}{4},\tfrac{3}{4}\right]$. The $y$-dependence of point values of the solution for the points $x_0=0.5$, $x_0=0.3$ and $x_0=0.2$ is plotted in Figure \ref{fig:1dsol}.

\begin{figure}[t]
\centering
 \hspace{-1cm}\includegraphics[scale=0.4]{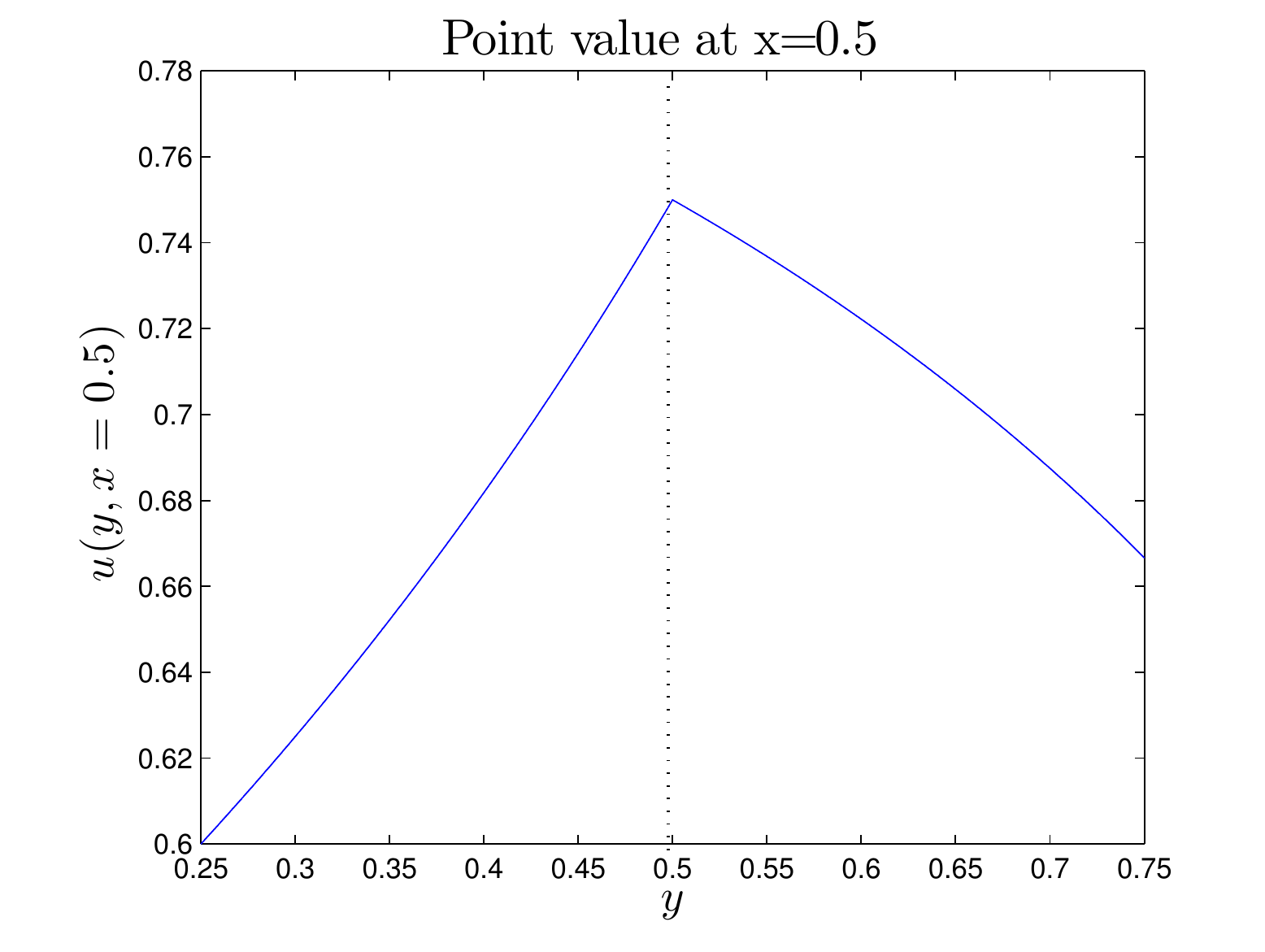}\hspace{-0.65cm}
  \includegraphics[scale=0.4]{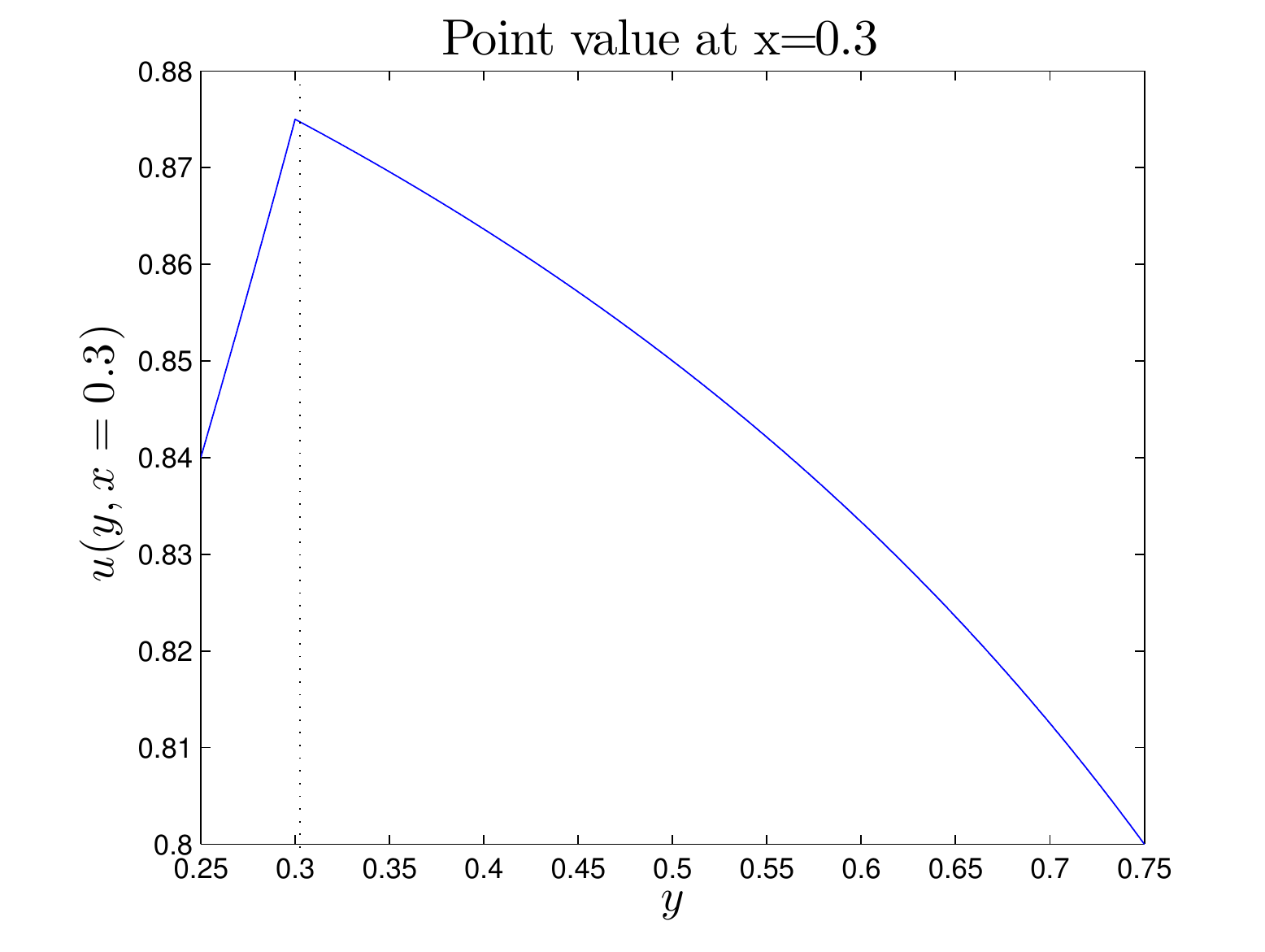}\hspace{-0.65cm}
   \includegraphics[scale=0.4]{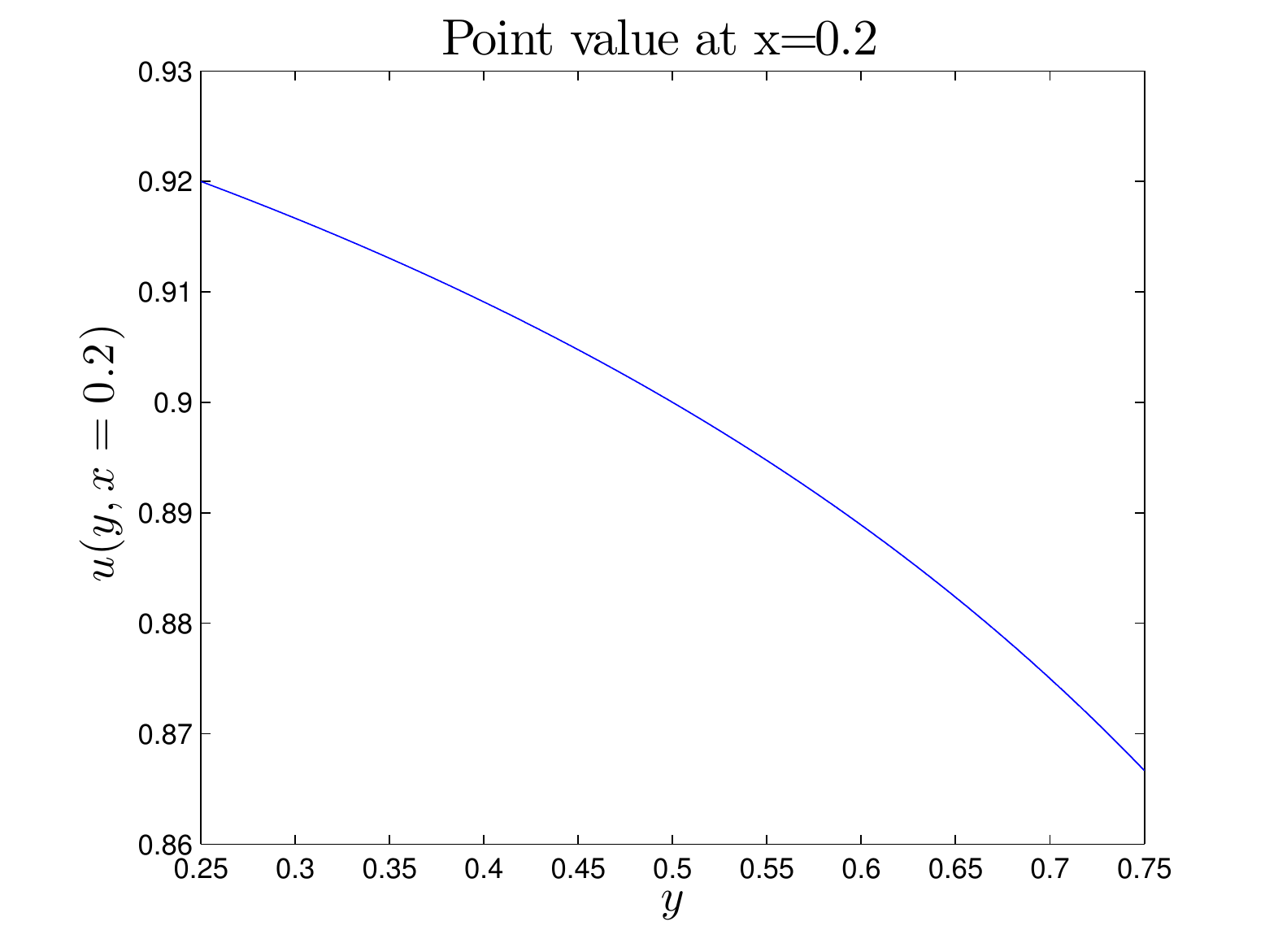}\caption{Point evaluations of the solution to \eqref{eq:1dpb} for $\alpha_l=3$ and $\alpha_r=1$, in dependence of $y\in \left[\tfrac{1}{4},\tfrac{3}{4}\right]$. For the evaluation at $x_0=0.5$ and $x_0=0.3$ (left and center plot), the dependence is not smooth as $x_0\in \left[\tfrac{1}{4},\tfrac{3}{4}\right]$, while the dependence is smooth for $x_0=0.2$ (right plot).}\label{fig:1dsol}
\end{figure}

\subsection{Parameter dependence for the transmission problem}\label{ssect:transmpv}

By analogy, we can expect a behavior similar to the one-dimensional case when considering the model problem \eqref{eq:modelpb} with coefficients \eqref{eq:coeffs} which are discontinuous across the interface.

To better understand, in the general case, the regularity of the point evaluation with respect to the parameter, we introduce a reference interface $\hat{\Gamma}:=\Gamma(\By=\bm{0})=\left\{r_0(\varphi), \varphi\in[0,2\pi)\right\}$, and a so-called \emph{nominal configuration}, corresponding to the domain configuration when the interface is $\hat{\Gamma}$ (see e.g. \cite{HSSS,CNT,HPS}). We denote by $\hdin$ and $\hdoutR$, respectively, the inner and outer domain in the nominal configuration. In the following, we use the terminology \emph{actual configuration} to denote the domain configuration when the interface is $\Gamma(\By)$, $\By\in\pspace_J$, and distinguish it from the nominal configuration. The nominal configuration can be mapped to the actual configuration by a parameter-dependent diffeomorphism $\Phi(\By):D_{R_{out}}\rightarrow D_{R_{out}}$, $\By\in\pspace_J$, $J\in\bbN$. This is the so-called mapping approach, first introduced in \cite{TX} and \cite{XT}. It is natural, from \eqref{eq:radiusy}, to consider the domain mapping as a perturbation of the identity (see also \cite[Sect. 2.8]{SZ} and \cite[Sect. 5.2]{CSZ}).
% \begin{assumption}\label{Aphi}
%  The diffeomorphism $\Phi:\pspace_J\times D_{R_{out}}\rightarrow D_{R_{out}}$ is such that:
%  \begin{enumerate}[(i)]
%   \item $\Phi$ and $\Phi^{-1}$ have the same smoothness, with respect to the parameter, as the radius $r$ in \eqref{eq:radiusy}; moreover, if, for $k\geq 1$, $\hexp\in[0,1]$, and every $J\in\bbN$, $\By\in\pspace_J$, $r(\By;\cdot)\in C^{k,\hexp}_{per}([0,2\pi))$ with a $J$- and $\By$-independent bound on its $C^{k,\hexp}_{per}$-norm, then, for every $J\in\bbN$ and $\By\in\pspace_J$, $\Phi(\By)\in C^{k,\hexp}(\overline{\hdin})\cup C^{k,\hexp}(\overline{\hdoutR})$ and $\Phi^{-1}(\By)\in C^{k,\hexp}(\overline{\din})\cup C^{k,\hexp}(\overline{\doutR})$, with $J$- and $\By$-independent bounds on the norms (that is, $\Phi$ and $\Phi^{-1}$ have the same spatial regularity as the radius);
%   \item $\Phi$ maps $\hat{\Gamma}$ to $\Gamma(\By)$, for every $\By\in\pspace_J$ and every $J\in\bbN$;
%   \item $\Phi(\By;\hat{\Bx})=\hat{\Bx}$ on $\partial D_{R_{out}}$, for every $\By\in\pspace_J$ and every $J\in\bbN$;
%   \item  there exist $\hat{\Bx}$-uniform upper and lower bounds $\sigma_{max}$, $\sigma_{min}>0$, independent of $J\in\bbN$ and $\By\in\pspace_J$, on the singular values of $D\Phi$ (the Jacobian matrix of $\Phi$) in $\hdin$ and $\hdoutR$.
%  \end{enumerate}
%  Given $\Phi$ fulfilling the above conditions, we require that Assumption \ref{Ak1k2} (the large wavelength assumption) holds with $\tau<\frac{\sigma_{min}^4}{\sigma_{max}^2}\min\left\{\frac{1}{\alpha_2},\alpha_2\right\}$.
% \end{assumption}
Here we consider
\begin{equation}\label{eq:phicircle}
 \Bx(\By)=\Phi(\By;\hat{\Bx})=\hat{\Bx}+\chi\left(\hat{\Bx}\right)\left(r(\By;\hat{\varphi}_{\hat{\Bx}})-r_0(\hat{\varphi}_{\hat{\Bx}})\right)\frac{\hat{\Bx}}{\lVert \hat{\Bx}\rVert},
\end{equation}
where $\hat{\Bx}$ denotes the coordinates in the nominal configuration, $\hat{\varphi}_{\hat{\Bx}}:=\arg(\hat{\Bx})$, and $\chi:D_{R_{out}}\rightarrow [0,1]$ is a mollifier. From now on, we assume the following on $\chi$: it acts on the radial component of $\hat{\Bx}$, it is supported in $[\tfrac{r_0^{-}}{4},\tilde{R}]$ for some $\tilde{R}\leq R_{out}$, it is strictly decreasing in $[\tfrac{r_0^{-}}{4},r_0]$ and strictly increasing in $[r_0,\tilde{R}]$, it has at least the same smoothness as the nominal radius $r_0$, and $\max\left\{\lVert\chi\rVert_{C^1(\overline{\hat{D}_{in}})}, \lVert\chi\rVert_{C^1(\overline{\hat{D}_{out,R_{out}}})}\right\}<\frac{\sqrt{2}}{r_0^{-}}$. These assumptions guarantee that $\Phi$ as in \eqref{eq:phicircle} is an orientation preserving diffeomorphism with the same spatial smoothness as the the radius $r$ in \eqref{eq:radiusy}, and the singular values of its Jacobian matrix have $J$-, $\By$- and $\hat{\Bx}$-independent lower and upper bounds $\sigma_{min}$ and $\sigma_{max}$, see Sect. 3.2 and Appendix E in \cite{HSSS} for details. Since in the following we will use H\"older norms of $\Phi(\By;\cdot)$ and its inverse, here we note that this domain mapping is well-defined as random variable taking values on H\"older spaces: although these spaces are not separable, it is clear from \eqref{eq:radiusy} and \eqref{eq:phicircle} that $\Phi$ and its inverse take values in separable subspaces of these spaces. 

With the mapping $\Phi$ at hand, we can rewrite the variational formulation of \eqref{eq:modelpb} on the nominal configuration as
%     \begin{subequations}
%     \begin{align}[left=\empheqlbrace]
%  &-\hat{\nabla}\cdot\left(\hat{\alpha}(\By;\hat{\Bx})\hat{\nabla} \uref\right) - \hat{\kappa}^2(\By;\hat{\Bx})\uref = 0 \quad \text{in }\hdin\cup\hdoutR,\\
%  & {\llbracket \uref \rrbracket_{\hat{\Gamma}}=0,\quad\llbracket \hat{\alpha}(\By;\hat{\Bx})\hat{\nabla} \uref \cdot \nref \rrbracket_{\hat{\Gamma}}=0},\label{mdpb:jumpref}\\
%  & \dfrac{\partial}{\partial \Bn_{out}} (\uref(\By)-u_i) = \operatorname{DtN}(\uref(\By))-\operatorname{DtN}(u_i)\quad\text{on }\partial D_{R_{out}},\\
%  & \text{for every }J\in\bbN, \By\in\pspace_J,\nonumber
%  \end{align}\label{eq:modelpbref}
%     \end{subequations}
\begin{equation}
\begin{split}
 &\text{Find }\uref(\By)\in H^1(D_{R_{out}}):\\
 &\int_{D_{R_{out}}}\hat{\alpha}(\By;\hat{\Bx})\hat{\nabla} \uref(\By)\cdot\hat{\nabla}\vref- \hat{\kappa}^2(\By;\hat{\Bx})\uref(\By) \vref\dd\xref-\int_{\partial D_{R_{out}}}\operatorname{DtN}(\uref(\By))\vref \dd S \\
 &= \int_{\partial D_{R_{out}}}\left(\dfrac{\partial u_i}{\partial \Bn_{out}}-\operatorname{DtN}(u_i)\right)\vref \dd S,\quad\text{for every }\vref\in H^1(D_{R_{out}})\text{ and } \By\in\pspace_J,
\end{split}
 \label{eq:varform}
\end{equation}
where $\hat{u}:=\Phi^{\ast}(u)$ ($\Phi^{\ast}$ denoting the pullback with respect to $\Phi$), $\nref$ is the normal to $\hat{\Gamma}$ pointing to $\hdoutR$, $\hat{\nabla}$ denotes differentiation with respect to $\hat{\Bx}$ and
\begin{equation}\label{eq:coeffshat}
\begin{split}
 \hat{\alpha}(\By;\hat{\Bx}) &= D\Phi(\By)^{-1}D\Phi(\By)^{-\top}\det D\Phi(\By) \alpha(\By;\Phi^{-1}(\By;\Bx)),\\
 \hat{\kappa}^2 (\By;\hat{\Bx})&= \det D\Phi(\By)\kappa^2(\By;\Phi^{-1}(\By;\Bx)),
\end{split}
\end{equation}
with $D\Phi$ the Jacobian matrix of $\Phi$. Since $\Phi(\By;\cdot)$ and its inverse take values in separable subspaces of H\"older spaces, the coefficients \eqref{eq:coeffshat} also do, and they are well-defined as H\"older-space valued random variables.

We have mentioned in Section \ref{sect:modelpb} that we work in the large-wavelength regime. In quantitative terms, this means that we assume the following:
\begin{assumption}[Large wavelength assumption]\label{Ak1k2}
The wavenumbers in \eqref{eq:coeffs} satisfy the condition:
 \begin{equation}\label{eq:kub}
  \kappa_1^2, \kappa_2^2 \leq \tau \,\inf_{w\in H^1(D_{R_{out}})}\frac{ |w|_{H^1(D_{R_{out}})}^2 + \left\|w\right\|_{L^2(\partial D_{R_{out}})}^2}{\left\|w\right\|_{L^2(D_{R_{out}})}^2},
 \end{equation}
with $\tau<\frac{\sigma_{min}^4}{\sigma_{max}^2}\min\left\{\frac{1}{\alpha_2},\alpha_2\right\}$ (where $\sigma_{min}$ and $\sigma_{max}$ are the $J$-, $\By$- and $\hat{\Bx}$-independent upper and lower bounds on the singular values of $D\Phi$).
\end{assumption}

In \cite[Sect. 5.3]{HSSS}, we have shown that, provided Assumption \ref{Ak1k2} holds, $\uref$ depends smoothly on $\By$ (it admits a holomorphic extension to polyellipses in the complex plane). This is possible because, on the nominal configuration, the interface $\hat{\Gamma}$ is fixed, for every parameter realization. However, this is not the case when considering the solution on the actual configuration.

Since we are interested in the evaluation of the solution $u$ to \eqref{eq:modelpb} at points that may be located on either side of the interface, we introduce the set
\begin{equation}\label{eq:hyperplane}
 \pspace_J^{\Gamma}(\Bx_0):=\left\{\By\in\pspace_J: \Bx_0\in \Gamma(\By)\right\}=\left\{\By\in\pspace_J: \Phi^{-1}(\By;\Bx_0)\in \hat{\Gamma}(\By)\right\}.
\end{equation}
Due to the affine parametrization \eqref{eq:radiusy} for the interface, for every $\Bx_0\in\bbR^2$ the set $\pspace_J^{\Gamma}(\Bx_0)$, if not empty, is a hyperplane, affine to a $(J-1)$-dimensional subspace of $\pspace_J$.

In the following proposition we show that, for the point evaluation in the actual configuration, the smoothness with respect to $\By$ is $C^0$ or $C^1$, and in general not $C^k$ for $k\geq 2$. We denote by $\mathcal{C}$ the continuity constant of the bilinear form $a_p(v,w):=\langle \hat{\nabla}\hat{v},\hat{\nabla}\hat{w}\rangle - \langle \hat{\nabla}\hat{v}\cdot \Bn_{out},\hat{w}\rangle_{\langle H^{-\frac{1}{2}}(\partial D_{R_{out}}),H^{\frac{1}{2}}(\partial D_{R_{out}})\rangle}$ on $H^1(D_{R_{out}})$, and by $\gamma_p$ the coercivity constant of $a_p$ restricted to functions that satisfy the radiation condition.
\begin{proposition}\label{prop:conty}
Let $u$ be the solution to \eqref{eq:modelpb} with coefficients \eqref{eq:coeffs}. Let Assumptions \ref{Arcoeffs} and \ref{Ak1k2} hold, and let us assume that we can build the mapping $\Phi$ in \eqref{eq:phicircle} such that $\frac{\sigma_{min}^4}{\sigma_{max}^4}\min\left\{\frac{1}{\alpha_2},\alpha_2\right\}\geq 1-\frac{\gamma_p}{\mathcal{C}}$. Consider $\Bx_0\in D_{R_{out}}$ such that $\pspace_J^{\Gamma}(\Bx_0)$ is not empty. 

For every $J\in\bbN$, the map $\By\mapsto u(\By;\Bx_0)$ from $\pspace_J$ to $\bbC$ is continuous, and $\lVert u(\cdot;\Bx_0)\rVert_{C^0(\mathcal{P}_J)}$ has a $J$-independent upper bound. 

If $\alpha_2=1$ in \eqref{eq:coeffs}, then the map $\By\mapsto u(\By;\Bx_0)$ is of class $C^1$, and $\lVert u(\cdot;\Bx_0)\rVert_{C^1(\mathcal{P}_J)}$ has a $J$-independent upper bound. 
\end{proposition}
\begin{proof}
We first consider the general case $\alpha_2\neq 1$. Using the mapping from the nominal configuration, we can write:
\begin{equation}\label{eq:c0boundy}
 \max_{\By\in\pspace_J}|u(\By;\Bx_0)|=\max_{\By\in\pspace_J}|\uref(\By;\Phi^{-1}(\By;\Bx_0))|\leq \max_{\By\in\pspace_J}\lVert\uref(\By;\cdot)\rVert_{C^0(\overline{D_{R_{out}}})}.
\end{equation}
This means that it is sufficient to show that the mapping $\By\mapsto \uref$ is continuous from $\pspace_J$ to $C^0(\overline{D_{R_{out}}})$, with a $J$-independent bound on the last term in \eqref{eq:c0boundy}. 

Assumption \ref{Arcoeffs} ensures that, for every $\By\in\pspace_J$, $\lVert r(\By;\cdot)\rVert_{C^{1,\hexp}_{per}([0,2\pi))}$ has a $J$- and $\By$-independent bound for some $\hexp\in(0,1)$. Thanks to the properties of $\Phi$, the coefficients $\hat{\alpha}(\By;\cdot)$ and $\hat{\kappa}^2(\By;\cdot)$ belong to $C^{\hexp}(\overline{\hdin})\cup C^{\hexp}(\overline{\hdoutR})$ for every $\By\in\pspace_J$, with $J$-independent bounds on the norms. Then, using Assumption \ref{Ak1k2} to ensure coercivity, we can apply Lemma 2 in \cite{Joc} and a slight generalization of Theorem 3.1 in \cite{BGL} to conclude that every realization of the scattered wave $\uref_s(\By;\cdot):=\uref(\By;\cdot)-u_i(\Phi(\By;\cdot))$ has a $J$- and $\By$-independent bound on its $H^{1+\hexp'}(D_{R_{out}})$-norm, for $0<\hexp'<\hexp$ sufficiently small. This implies that the right-hand side in \eqref{eq:c0boundy} has a $J$-independent bound, thanks to the Sobolev embedding theorem \cite[Thm. 7.26]{GT}. The adaptation of Theorem 3.1 in \cite{BGL} to our case is reported in Appendix A, and it requires that $\frac{\sigma_{min}^4}{\sigma_{max}^4}\min\left\{\frac{1}{\alpha_2},\alpha_2\right\}\geq 1-\frac{\gamma_p}{\mathcal{C}}$.

The smoothness with respect to $\By$ is limited by the spatial smoothness of $\uref$, due to the application of the chain rule on $\uref(\By;\Phi^{-1}(\By;\Bx_0))$. Thus, if $\alpha_2=1$ in \eqref{eq:coeffs}, we obtain higher regularity with respect to the parameter.

In particular, for a generic $\alpha_2$, the $J$- and $\By$-independent upper bound on $\lVert\uref(\By;\cdot)\rVert_{C^0(\overline{D_{R_{out}}})}$ and the H\"older regularity of the PDE coefficients (with $J$- and $\By$-independent norm bounds) also imply a $J$- and $\By$-upper bound on $\lVert\uref(\By;\cdot)\rVert_{C^{1+\hexp}(\overline{\hdin})}$ and $\lVert\uref(\By;\cdot)\rVert_{C^{1+\hexp}(\overline{\hdoutR})}$ \cite[Thm. 8.33]{GT}. We specify that the result in \cite{GT} is for a boundary value problem, but it can be adapted to a interface problem proceeding as elaborated in Appendix B. If $\alpha_2=1$, then the transmission conditions at $\hat{\Gamma}$ ensure that $\uref\in C^{1,\hexp}(\overline{D_{R_{out}}})$. It is shown in \cite[Lemma 4.3.8]{LSthesis} that $r=r(\By)$, as $C^1$-valued map, is analytic with respect to $\By$ even in the case that $\mathcal{P}_J$ has infinite dimension ($J\rightarrow\infty$). Then, thanks to the regularity of $\Phi$ in \eqref{eq:phicircle} with respect to $\By$ \cite[Lemma 4.3.9]{LSthesis}, the claim for $\alpha_2=1$ follows from chain rule.
\end{proof}

\begin{remark}[Alternative proof] The assumption that $\frac{\sigma_{min}^4}{\sigma_{max}^4}\min\left\{\frac{1}{\alpha_2},\alpha_2\right\}\geq 1-\frac{\gamma_p}{\mathcal{C}}$ in the above proposition is due to a technicality when adapting the proof of Theorem 3.1 in \cite{BGL} to our case, because of the presence of the $\operatorname{DtN}$ map (see Appendix A for details). Such requirement can be dropped if, instead of using the result in \cite{BGL}, we show $H^2$-regularity of $\hat{u}$ in each subdomain. In the latter case, though, we would have a $J$-independent bound on the right-hand side in \eqref{eq:c0boundy} only when, in Assumption \ref{Arcoeffs}, $p<\tfrac{1}{3}$. Moreover, the requirement $\frac{\sigma_{min}^4}{\sigma_{max}^4}\min\left\{\frac{1}{\alpha_2},\alpha_2\right\}\geq 1-\frac{\gamma_p}{\mathcal{C}}$ is not needed at all if instead of boundary conditions with the $\operatorname{DtN}$ map we have Dirichlet or Neumann boundary conditions \cite{BGL}.
 \end{remark}

\subsection{Failure of high order quadrature methods for the point evaluation}\label{ssect:smolyak}

The proof of Proposition \ref{prop:conty} shows that the smoothness of the point evaluation with respect to the high-dimensional parameter depends on the spatial smoothness of the solution $u$ to \eqref{eq:modelpb} across the interface. In particular, as the solution is in general only continuous (or $C^1$) across the interface, we cannot expect, in general, a holomorphic dependence of the point evaluation with respect to $\By$. Consequently, it is likely that high order quadrature methods such as stochastic Galerkin \cite{SG}, stochastic collocation \cite{BNT,SS,XH} or high order quasi-Monte Carlo rules \cite{DLGS}, will not show full convergence rates when trying to compute statistics of $u(\By;\Bx_0)$ for a point $\Bx_0\in\bbR^2$ that can be crossed by the interface. We show this effect on sparse grids.

We have run the Smolyak adaptive algorithm described in \cite{SS} using $\mathfrak{R}$-Leja quadrature points. We have set $J=16$ and $\beta_{2j-1}=\beta_{2j}=r_0 j^{-3}$ for $j=1,\ldots,8$, where $r_0= 0.01$ is the nominal radius (we work with non-dimensional quantities), chosen to be constant. The coefficients in \eqref{eq:modelpb} have been chosen as $\alpha_2=4$, $\kappa_1^2=\kappa_0^2$ and $\kappa_2^2=4\kappa_0^2$, with $k_0=209.44$.\footnote{Computational details: Smolyak algorithm run on a discrete solution obtained from finite element discretization of \eqref{eq:varform} with piecewise linear, globally continuous ansatz functions; the $\operatorname{DtN}$ map \eqref{mdpb:dtn} has been approximated using a circular PML, starting at $R_{out}=0.055$ and ending at $R_{out}'=0.075$, with absorption coefficient $0.5$; the mesh is quasi-uniform and consists of $558705$ nodes; domain mapping \eqref{eq:phicircle} with $\chi(\hat{\Bx})=\frac{\lVert\hat{\Bx}-\frac{r_0}{4}\rVert}{r_0-\frac{r_0}{4}}$ for $\lVert\hat{\Bx}\rVert\in \left[\frac{r_0}{4}, r_0\right]$, $\chi(\hat{\Bx})=\frac{R_{out}-\lVert \hat{\Bx}\rVert}{R_{out}-r_0}$ for $\lVert\hat{\Bx}\rVert\in \left[r_0, R_{out}\right]$ and $\chi(\hat{\Bx})=0$ elsewhere (the interface $\lVert \hat{\Bx}\rVert=\frac{r_0}{4}$ is resolved, so that the discontinuity in the mapping does not affect the finite element convergence); software: NGSolve finite element library (http://sourceforge.net/apps/mediawiki/ngsolve) coupled with the MKL version of PARDISO (https://software.intel.com/en-us/intel-mkl) for the direct solver.} The quantity of interest is $\Re \Bu_h:=\left\{\Re u_h(\By;\Bx_i)\right\}_{i=0}^{N-1}$, for some $N\in\bbN$, where $u_h(\By;\Bx_i)$ denotes the discrete approximation to $u(\By;\Bx_i)$, $i=0,\ldots,N-1$. In the case of holomorphic parameter dependence of the Q.o.I., we would expect a convergence rate of $s=2-\varepsilon$ with respect to the cardinality of the index set $\Lambda$, for $\varepsilon>0$ arbitrary small \cite{SS}. For our experiments, we report the \emph{estimated} error $\sum_{\nu\in\mathcal{N}(\Lambda)}\lVert\Delta^Q_{\nu}(\Re \Bu_h)\rVert_{\infty}$ computed by the algorithm at each iteration; here $\mathcal{N}(\Lambda)$ is the set of neighbors of the index set $\Lambda$, and $\Delta^Q_{\nu}$ are the difference operators (see e.g. \cite{SS} for their definition). The left plot in Figure \ref{fig:smolyak} shows the convergence of the algorithm when applied to one point evaluation ($N=1$), for different points on the horizontal axis. We can see that the curve saturates if the point is crossed by the interface $\Gamma(\By)$ for many parameter realizations ($\Bx_0=(r_0,0)$, with $r_0=0.01$). If the point is never crossed by the interface ($\Bx_0=(0.005,0)$ and $\Bx_0=(0.015,0)$), then the algorithm converges with full rate. If the point is crossed by the interface but only for few parameter realizations ($\Bx_0=(0.009,0)$ and $\Bx_0=(0.011,0)$), we still observe good convergence, although the rate is slightly worse. These results for a single point evaluation can still be considered satisfactory, despite some decrease in the convergence rate when the point is crossed many times by the interface. However, in applications it could be interesting to have a field distribution, and thus the value of the field at many locations simultaneously. The center and right plots in Figure \ref{fig:smolyak} show the error estimated by the adaptive Smolyak algorithm when applied to more point evaluations simultaneously ($N\geq 2$). The more points we consider, the more surfaces of non-smoothness \eqref{eq:hyperplane} are present in the parameter space, and the more the convergence rate deteriorates. Not surprisingly, when considering $N=8$ point evaluations, we observe a convergence rate of $0.5$ with respect to the number of function evaluations, that is the same rate as a Monte Carlo quadrature rule.

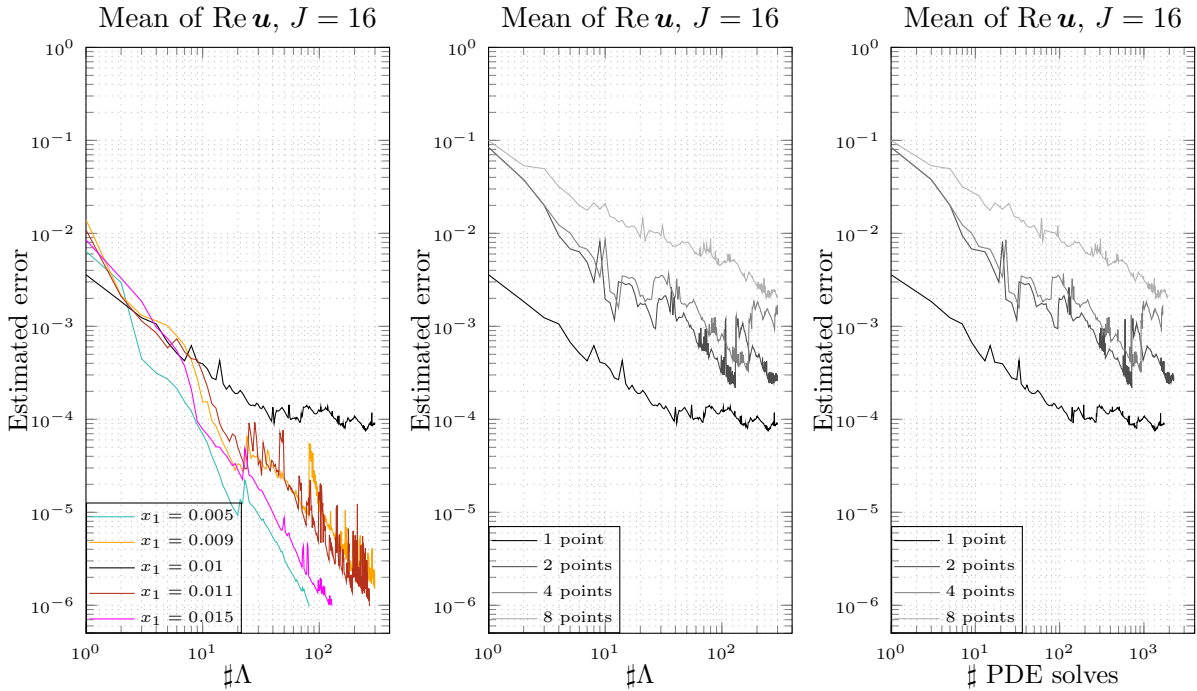
\begin{figure}[t]
\noindent
\begin{tikzpicture}
 \begin{loglogaxis}[
 title={Mean of $\Re \Bu$, $J=16$},
 xlabel={$\sharp \Lambda$},
 ylabel={Estimated error},
 grid=minor, legend entries={ {$x_1=0.005$}, {$x_1=0.009$}, {$x_1=0.01$}, {$x_1=0.011$}, {$x_1=0.015$}}, xmin=1,xmax=400,ymin=5e-7,ymax=1
  ]
  \addplot[Turquoise!85!black] table[x expr =(\coordindex+1),y expr =\thisrowno{0}]{erroriter1pt_005.txt};
  \addplot[Orange] table[x expr =(\coordindex+1),y expr =\thisrowno{0}]{erroriter1pt_009.txt};
  \addplot[black] table[x expr =(\coordindex+1),y expr =\thisrowno{0}]{erroriter1pt.txt};
  \addplot[BrickRed] table[x expr =(\coordindex+1),y expr =\thisrowno{0}]{erroriter1pt_011.txt};
  \addplot[Magenta] table[x expr =(\coordindex+1),y expr =\thisrowno{0}]{erroriter1pt_015.txt};
 \end{loglogaxis}
 \end{tikzpicture}
\begin{tikzpicture}
 \begin{loglogaxis}[
 title={Mean of $\Re \Bu$, $J=16$},
 xlabel={$\sharp \Lambda$},
 ylabel={Estimated error},
 grid=minor, legend entries={ {1 point}, {2 points}, {4 points}, {8 points}}, xmin=1,xmax=400,ymin=5e-7,ymax=1
  ]
    \addplot[black] table[x expr =(\coordindex+1),y expr =\thisrowno{0}]{erroriter1pt.txt};
  \addplot[white!30!black] table[x expr =(\coordindex+1),y expr =\thisrowno{0}]{erroriter2pts.txt};
  \addplot[white!50!black] table[x expr =(\coordindex+1),y expr =\thisrowno{0}]{erroriter4pts.txt};
  \addplot[white!70!black] table[x expr =(\coordindex+1),y expr =\thisrowno{0}]{erroriter8pts.txt};
 \end{loglogaxis}
 \end{tikzpicture}
\begin{tikzpicture}
 \begin{loglogaxis}[
 title={Mean of $\Re \Bu$, $J=16$},
 xlabel={$\sharp$ PDE solves},
 ylabel={Estimated error},
 grid=minor, legend entries={ {1 point}, {2 points}, {4 points}, {8 points}}, xmin=1,xmax=4000,ymin=5e-7,ymax=1
  ]
        \settable{numind1pt.txt}{erroriter1pt.txt}
  \addplot[black] table[x=data,y expr =\thisrowno{0}]{\datatable};
      \settable{numind2pts.txt}{erroriter2pts.txt}
  \addplot[white!30!black] table[x=data,y expr =\thisrowno{0}]{\datatable};
    \settable{numind4pts.txt}{erroriter4pts.txt}
  \addplot[white!50!black] table[x=data,y expr =\thisrowno{0}]{\datatable};
  \settable{numind8pts.txt}{erroriter8pts.txt}
  \addplot[white!70!black] table[x=data,y expr =\thisrowno{0}]{\datatable};
 \end{loglogaxis}
 \end{tikzpicture}
 \caption{Convergence of adaptive Smolyak algorithm for one point evaluation at different points (left) and more point evaluations simultaneously, the latter with respect to the cardinality of the index set $\Lambda$ (center) and to the number of PDE solves (right). In the first plot, $x_1$ denotes the first coordinate of $\Bx_0$. The error reported is the one estimated by the algorithm, $\sum_{\nu\in\mathcal{N}(\Lambda)}\lVert\Delta^Q_{\nu}(\Re \Bu_h)\rVert_{\infty}$, see \cite{SS} for its definition.}\label{fig:smolyak}
 \end{figure}

\subsection{Possible remedies}\label{ssect:discussion}
Two strategies can be identified in order to handle the loss of smoothness of the Q.o.I. with respect to the parameter: detect the surface of non-smoothness and apply a high order quadrature method in each subdomain of $\pspace_J$ where the Q.o.I. is smooth, or adopt a low order quadrature method requiring less smoothness of the Q.o.I.

The problem with the first strategy is that, in our case, the discontinuities are not easy to track. For a single point evaluation, we could apply already existing discontinuity detection techniques, see  \cite{ZWGB} and references therein. For multiple evaluations, the complexity of the surface of non-smoothness increases with the number of points, as we have a hyperplane of discontinuity for each of them. In such a case, the method proposed in \cite{ZWGB} cannot be applied anymore; we do not exclude that the algorithm in \cite{ZWGB} could be adapted to tackle multiple discontinuity detection, but its complexity would probably grow with the number of hyperplanes of discontinuity. Another possibility to pursue the first strategy is to adopt an approach based on X-FEM in the parameter space, as proposed in \cite{NCSM} in the framework of a level set approach to describe the uncertain geometry (and named X-SFEM by the authors). However, the algorithm proposed there to track the uncertain boundary seems to be applicable only to a low-dimensional parameter space (cf. in particular Sect. 6.2.3 in \cite{NCSM}).

Our choice is therefore the second strategy, namely to use a quadrature rule that does not suffer from the so-called `curse of dimensionality' and requires weak assumptions on the regularity of the Q.o.I.. Namely, a quadrature rule that does not need information about the location of the surfaces of non-smoothness to provide the full convergence rate. We opt therefore for a Monte Carlo approach and in particular, in order to reduce the computational effort, to its multilevel version (MLMC). The latter requires only square integrability of the Q.o.I., which is a much weaker smoothness assumption than those for high order quadrature methods. Moreover, since the convergence rate of the MLMC integration does not depend on the dimension of the parameter space, such quadrature rule is well suited for high-dimensional problems.

% We remark that, in light of Proposition \ref{prop:conty}, a quasi-Monte Carlo quadrature rule \cite{Nied} could work in the case of $C^1$-smoothness of the Q.o.I., that is in the case that $\alpha_2=1$ in \eqref{eq:coeffs}.

\section{Multilevel Monte Carlo for high-dimensional problems}\label{sect:mlmcintro}
% In light of the discussion in the previous section, we intend to use MLMC to compute statistics of a non-smooth Q.o.I., in our case point values of the solution to \eqref{eq:modelpb}. 

In this section we provide a brief overview of MLMC, in particular for a Q.o.I. depending on the solution of an elliptic PDE. Our survey is based on \cite{Giles08}, \cite{CGST} and \cite{BSZ}, and we use the same notation as in the introduction.

Due to the need to solve a PDE to compute the Q.o.I., usually we do not have at our disposal the quantity $Q$ itself, but an approximation to it. We consider a sequence $(\mathcal{X}^l)_{l\geq 0}$ of finite-dimensional subspaces of $\mathcal{X}$
\begin{equation}\label{eq:fespaces}
 \mathcal{X}^0\subset \mathcal{X}^1\subset \ldots \subset \mathcal{X}^l\subset\ \ldots \subset \mathcal{X},
\end{equation}
with $\mathcal{X}^l$ associated to the discretization parameter $h_l\in\bbR$, $l\in\bbN_0=\bbN\cup\left\{0\right\}$. Thinking of $h_l$ as the meshsize at level $l$, we can assume, without loss of generality, that $h_0\geq h_1\geq h_2\geq\ldots $. We denote by $u_l$ the discrete solution to \eqref{eq:diffeq} on the \emph{level} $l$:
\begin{equation}
 \begin{split}
   \text{Find } u_l \text{ s.t.}: \quad
&  u_l(\By)\in \mathcal{X}^l \quad\text{for every } \By\in\pspace_J,\\
& \mathcal{D}_l(\By;u_l(\By))=0\quad \text{for every } \By\in\pspace_J,
 \end{split}\label{eq:diffeqdiscr}
\end{equation}
where the subscript in $\mathcal{D}_l$ denotes the discretization of $\mathcal{D}$ in \eqref{eq:diffeq} at level $l$. With $u_l$ at our disposal, for some $l\in\bbN_0$, we can compute the approximation of $Q$ at the level $l$ for every $\By\in\pspace_J$, which we denote by $Q_l(\By):=q_l(\By;u_l(\By))$. Note that the discretization error in $Q_l$ might be due not only to the replacement of $u$ by $u_l$, but also to an error coming from the computation of $q$ on $u_l$ (for example, if $q$ is an output functional defined as an integral quantity that needs numerical integration). 

The MLMC method is a modification to the single-level Monte Carlo (MC) algorithm in order to improve the computational efficiency. In single-level Monte Carlo, the quantity $\bbE_{\mu_J}[Q]$ is estimated by
\begin{equation}\label{eq:mcest}
E_M[Q_L]:=\frac{1}{M}\sum_{i=1}^M Q_L^i \in \mathcal{X}^L,
\end{equation}
where we have assumed that $u$ is approximated by the solution $u_L$ to \eqref{eq:diffeqdiscr} at a \emph{fixed} level $L\in\bbN_0$, and $Q_L^i:=q_L(\By^i;u_L(\By^i))$, $i=1,\ldots,M$, $M\in\bbN$, are independent, identically distributed realizations of $Q_L(\By)$, $\By\in\pspace_J$. Note that the definition \eqref{eq:mcest} is independent of $J\in\bbN$. The approximation error of the MC estimator can be decomposed as (cf. \cite[Sect. 4.2]{BSZ})
\begin{equation}
 \lVert\bbE_{\mu_J}[Q]-E_M[Q_L]\rVert_{L^2(\pspace_J,\mathcal{Y})}\leq \bbE_{\mu_J}\left(\lVert Q-Q_L\rVert_{\mathcal{Y}}\right)+\frac{1}{\sqrt{M}}\Var[Q_L],\label{eq:slmc}
\end{equation}
provided $Q$ and $Q^L$ have finite variance. The norm on the left-hand side is defined as
\begin{equation}
 \lVert v\rVert_{L^p(\pspace_J,\mathcal{Y})}=\begin{cases}
				    \left(\int_{\pspace_J}\lVert v(\By)\rVert_{\mathcal{Y}}^p\dd\mu_J(\By)\right)^{\frac{1}{p}} & \text{if }p<\infty,\\
				    \text{esssup}_{\By\in\pspace_J} \lVert v(\By)\rVert_{\mathcal{Y}} & \text{if }p=\infty
                                 \end{cases}
\end{equation}
(analogous definition holds when replacing $\mathcal{Y}$ by $\mathcal{X}$ or any other separable Banach space). The first summand in \eqref{eq:slmc} measures the bias of $Q_L$ with respect to $Q$, and depends on the spatial discretization error, that is on the accuracy with which $Q_L$ approximates $Q$ for every parameter realization. The second summand is the so-called sampling error, depending on the variance $\Var$ of $Q_L$ and the number of samples $M$. To balance the two error contributions for a certain threshold on the total error, $\sqrt{M}$ has to be chosen to be inversely proportional to the discretization error. For fine meshes, this can be very expensive.

The idea of the multilevel version of Monte Carlo is to reduce the error contribution from the second summand in \eqref{eq:slmc} with a lower computational effort than the MC method. Setting by convention $Q_{-1}:=0$ and exploiting that 
\begin{equation*}
 \bbE_{\mu_J}[Q_L]=\sum_{l=0}^L\bbE_{\mu_J}[Q_l-Q_{l-1}],
\end{equation*}
the classical MLMC method consists in estimating $\bbE_{\mu_J}[Q]$ by
\begin{equation}\label{eq:mlmc}
 E^L[Q]:=\sum_{l=0}^L E_{M_l}[Q_l-Q_{l-1}],
\end{equation}
for a given $L\in\bbN$ (again, note that the definition \eqref{eq:mlmc} is independent of $J$). In this way, the variance is reduced at each level $l\in\bbN_0$ estimating the mean of the tail $Q_l-Q_{l-1}$, and this is the reason why the MLMC is also said to be a variance reduction technique. If $\Var[Q_l-Q_{l-1}]$ decreases as $l$ increases, as we will see to be usually the case, then it is possible to save computational effort with respect to MC, taking more samples on the coarser grids and only few samples on the finer ones. In other words, the advantage of MLMC consists in balancing the two opposite effects, as $l$ increases, of the decay of $\Var[Q_l-Q_{l-1}]$ and the increase of $\operatorname{Work}_l$, the computational cost to compute a sample of $Q_l-Q_{l-1}$. This balancing is achieved by determining the optimal number of samples $M_l$ for each of the levels $l=0,\ldots,L$ (and possibly also the optimal maximal level $L$) in order to achieve a certain accuracy for the total error at minimal computational cost.

\begin{theorem}[Theorem 1 in \cite{CGST}]\label{thm:mlmccvg}
 Suppose that, for every $J\in\bbN$, there exist positive constants $\alpha,\beta,\gamma, C_1, C_2$ and $C_3$, independent of $l\in\bbN_0$, such that $\alpha\geq\tfrac{1}{2}\min(\beta,\gamma)$, and
 \begin{enumerate}[(i)]
  \item $\lVert\bbE_{\mu_J}[Q_l-Q]\rVert_{\mathcal{Y}}\leq C_1 h_l^{\alpha}$,
  \item $\emph{Var}_{\mu_J}[Q_l-Q_{l-1}]\leq C_2 h_l^{\beta}$,
  \item $\operatorname{Work}_l\leq C_3 h_l^{-\gamma}$.
 \end{enumerate}

 Then, for every $\varepsilon<e^{-1}$, there exist a value $L$ and a sequence $(M_l)_{l=0}^L$ such that
 \begin{equation}
  \lVert E^L[Q]-\bbE_{\mu}[Q]\rVert_{L^2(\pspace_J,\mathcal{Y})}<\varepsilon,
 \end{equation}
 for every $J\in\bbN$, and there exists a positive constant $C_4$ such that the total computational cost $\operatorname{Work}_{tot}(E^L)$ is bounded by
 \begin{equation}\label{eq:cost}
  \operatorname{Work}_{tot}(E^L)\leq\begin{cases}
C_3 C_4 \varepsilon^{-2}& \text{if }\beta>\gamma,\\
C_3 C_4 \varepsilon^{-2}(\log\varepsilon)^2 & \text{if }\beta=\gamma,\\
C_3 C_4 \varepsilon^{-2-\frac{\gamma-\beta}{\alpha}} & \text{if }\beta<\gamma.
                            \end{cases}
 \end{equation}
 If the constants $C_1$ and $C_2$ are independent of $J\in\bbN$, then $C_4$ is independent of $J\in\bbN$.
\end{theorem}
The proof can be found in Appendix A of \cite{CGST}, where it can be checked that the constant $C_4$ is dependent on $C_1$ and $C_2$ but not on $C_3$. In general, however, we cannot expect the cost $C_3$ of a single solve $\operatorname{Work}_l$ to be independent of $J$.

The previous theorem indicates that, in order to compute the optimal distribution of samples on each level, it is necessary to determine the values of the exponents $\alpha,\beta$ and $\gamma$. The value of $\gamma$ depends on the method used to discretize \eqref{eq:diffeq} and on the quantity of interest $Q$. For example, if $Q$ is the solution ifself and \eqref{eq:diffeqdiscr} correspond to linear finite element discretizations, then a multigrid solver for the linear system has linear complexity with respect to the number of degrees of freedom, and we can set $\gamma=d$, with $d=1,2,3$ the spatial dimension of the problem. The exponents $\alpha$ and $\beta$ can be determined, instead, from the convergence estimate of $Q_l$ to $Q$ as $l\rightarrow\infty$.

\begin{proposition}\label{prop:cvgrates}
% Let the following assumptions be fulfilled:
% \begin{enumerate}[(i)]
% \item the Q.o.I. and its approximations are such that $Q\in L^2(\pspace_J,\mathcal{Y})$ and $Q_l\in L^2(\pspace_J,\mathcal{Y})$ for every $l\in\bbN_0$,  with $J$-independent norm bounds;
%  \item there exists a constant $C_A>0$ independent of $h_l$, $l\in\bbN_0$, of $\By\in\pspace_J$ and of $J\in\bbN$, and a positive real number $r$, independent of $l$, such that
%  \begin{equation}\label{eq:approxest}
%  \lVert q(\By;u(\By))-q_l(\By;u_l(\By))\rVert_{\mathcal{Y}}\leq C_A h_l^r\lVert u(\By)\rVert_{\mathcal{W}},\;\text{for every }\By\in\pspace_J\setminus\mathcal{N}_{\pspace_J}\text{ and every }J\in\bbN,
% \end{equation}
% for a subspace $\mathcal{W}\subset \mathcal{X}$, and $\mathcal{N}_{\pspace_J}$ any nondense subset of $\pspace_J$ with measure zero, for every $J\in\bbN$; moreover, $u\in L^2(\pspace_J,\mathcal{W})$ with a $J$-independent norm bound.
% \end{enumerate}
Assume that $Q\in L^1(\pspace_J,\mathcal{Y})$ for every $J\in\bbN$, and that there exists a constant $C_A>0$, independent of $h_l$, $l\in\bbN_0$, of $\By\in\pspace_J$ and of $J\in\bbN$, and a positive real number $t$, independent of $l$, such that the approximations of $Q$ fulfill
 \begin{equation}\label{eq:approxest}
 \lVert q(\By;u(\By))-q_l(\By;u_l(\By))\rVert_{\mathcal{Y}}\leq C_A h_l^t\lVert u(\By)\rVert_{\mathcal{W}},\;\text{for every }\By\in\pspace_J\setminus\mathcal{N}_{\pspace_J}\text{ and every }J\in\bbN,
\end{equation}
for a subspace $\mathcal{W}\subset \mathcal{X}$, and $\mathcal{N}_{\pspace_J}$ any nondense subset of $\pspace_J$ with measure zero. Moreover, let $u\in L^2(\pspace_J,\mathcal{W})$ with a $J$-independent norm bound.

Then, given a geometric sequence $(h_l)_{l\geq 0}$ of discretization parameters:
  \begin{itemize}
   \item the inequality $(i)$ in Theorem \ref{thm:mlmccvg} holds with a constant $C_1$ independent of $J$ and $\alpha=t$;
   \item the inequality $(ii)$ in Theorem \ref{thm:mlmccvg} holds with  a constant $C_2$ independent of $J$ and $\beta=2t$.
  \end{itemize}
\end{proposition}
\begin{proof}
 We first address the bound $(i)$ in Theorem \ref{thm:mlmccvg}. From the properties of the Bochner integral and \eqref{eq:approxest} we have, for every $l\in\bbN_0$:
 \begin{equation*}
 \lVert \bbE_{\mu_J}[Q_l-Q]\rVert_{\mathcal{Y}}\leq\bbE_{\mu_J}[\lVert Q_l-Q\rVert_{\mathcal{Y}}]\leq C_A h_l^t\lVert u\rVert_{L^1(\pspace_J,\mathcal{W})}.
 \end{equation*}
 Thus we obtain the bound $(i)$ in Theorem \ref{thm:mlmccvg} with $C_1=C_A \sup_{J\in\bbN}\lVert u\rVert_{L^1(\pspace_J,\mathcal{W})}$ and $\alpha=t$.
 
 For the bound $(ii)$ in Theorem \ref{thm:mlmccvg}, we have:
 \begin{align}
  \Var[Q_l-Q_{l-1}]&\leq \bbE_{\mu_J}[\lVert Q_l-Q_{l-1}\rVert^2_\mathcal{Y}]=\lVert Q_l-Q_{l-1}\rVert^2_{L^2(\pspace_J,\mathcal{Y})}\label{eq:varub}\\
  & \leq 2\lVert Q_l-Q\rVert^2_{L^2(\pspace_J,\mathcal{Y})} + 2\lVert Q_{l-1}-Q\rVert^2_{L^2(\pspace_J,\mathcal{Y})}.\label{eq:varub2}
 \end{align}
Owing to \eqref{eq:approxest}, the first summand is bounded by
\begin{equation*}
 \lVert Q_l-Q\rVert^2_{L^2(\pspace_J,\mathcal{Y})}\leq C_A^2 h_l^{2t}\sup_{J\in\bbN}\lVert u \rVert^2_{L^2(\pspace_J,\mathcal{W})}.
\end{equation*}
The analogous holds for the second summand in \eqref{eq:varub2} replacing $h_l$ by $h_{l-1}$. We remind that we assume a geometric sequence of discretization parameters, that is $\frac{h_{l-1}}{h_l}\leq C_H$ for every $l\in\bbN_0$ and some $C_H>0$. Then, substituting the above bounds in \eqref{eq:varub2}, we obtain 
\begin{align*}
 \lVert Q_l-Q_{l-1}\rVert^2_{L^2(\pspace_J,\mathcal{Y})}& \leq 2 C_A^2(h_l^{2t}+h_{l-1}^{2t})\lVert u \rVert^2_{L^2(\pspace_J,\mathcal{W})}\\
 &\leq 2 C_A^2(1+C_H^{2t})h_l^{2t}\lVert u \rVert^2_{L^2(\pspace_J,\mathcal{W})},
\end{align*}
that is the bound $(ii)$ in Theorem \ref{thm:mlmccvg} holds with $C_2=2 C_A^2(1+C_H^{2t})\sup_{J\in\bbN}\lVert u \rVert^2_{L^2(\pspace,\mathcal{W})}$ and $\beta=2t$.
\end{proof}
It is clear from the proof that we could slightly relax the assumption on the constant $C_A$, requiring it to belong to $L^2(\pspace_J,\bbR)$ with a $J$-independent bound, instead of being $\By$-independent.

We also note that in \eqref{eq:varub} we have shown that $\lVert Q_l-Q_{l-1}\rVert^2_{L^2(\pspace,Y)}\leq C_2h_l^{\beta}$, which is a stronger requirement that $\operatorname{Var}_{\mu}[Q_l-Q_{l-1}]\leq C_2 h_l^{\beta}$. Such choice gives automatically $\alpha\geq \frac{\beta}{2}$ in Theorem \ref{thm:mlmccvg} (cf. \cite[Sect. 2.1]{Giles15}).

% \begin{remark}[$J$-robustness]
%  The estimates in the above proposition have been stated in such a way that the MLMC is robust with respect to the dimension $J$ of the parameter space. Under the assumptions of the above theorem, the bound \eqref{eq:cost} depends on $J$ only through $C_3$, and as we already mentioned the dependence of $C_3$ on $J$ is in general hard to avoid. For instance, in our model problem the constant $C_3$ depends linearly on $J$ due to the evaluation of \eqref{eq:radiusy}; in this case the independence of $C_4$ of $J$ ensures that the total cost increases at most linearly with $J$.
% \end{remark}

\begin{remark}[log factors in convergence rates]
 The statement of Proposition \ref{prop:cvgrates} can be adapted easily to the case when the convergence rate in \eqref{eq:approxest} is of the kind $h_l^t|\log h_l|^{\bar{t}}\lVert u(\By)\rVert_W$, for some power $\bar{t}>0$ of $\log h_l$. Indeed, as $l\rightarrow\infty$, $h_l^t|\log h_l|^{\bar{t}}\leq h^{t'}$ for any $t'<t$ and any $\bar{t}>0$, and one can use the result of Proposition \ref{prop:cvgrates} with $t'$ in place of $t$. For the estimate \eqref{eq:cost} we have then the following situations: if $2t>\gamma$, we can choose $t'$ such that $2t>2t'>\gamma$, and still obtain $\operatorname{Work}_{tot}(E^L)\leq \varepsilon^{-2}$; if $2t=\gamma$, then, using $t'$, we switch from the second to the third case, with $\operatorname{Work}_{tot}(E^L)\leq \varepsilon^{-2+\delta}$, for any $\delta>0$; if $2t<\gamma$, then $\operatorname{Work}_{tot}(E^L)\leq \varepsilon^{-2-\frac{\gamma-2t}{t}+\delta}$ for any $\delta>0$.
\end{remark}

In the next section we establish under which conditions the assumptions of Proposition \ref{prop:cvgrates} are fulfilled for our model transmission problem when discretized using finite elements. For the point evaluation, in \eqref{eq:approxest} we have $q=(\By;u(\By))=u(\By;\Bx_0)=\hat{u}(\By;\hat{\Bx}_0(\By))$, with $\hat{\Bx}_0=\Phi^{-1}(\By;\hat{\Bx}_0)$, and $\mathcal{Y}=\bbR$, and we need to determine the exponent $t$, the space $\mathcal{W}$, and show that $u\in L^2(\pspace_J,\mathcal{W})$ with a $J$-independent bound. Note that the assumption $Q\in L^1(\pspace_J,\mathcal{Y})$ holds as we have shown in Proposition \ref{prop:conty} the continuity of the map $\By\mapsto u(\By;\Bx_0)$ from $\pspace_J$ to $\bbC$.

\section{Space regularity and finite element convergence for the model transmission problem}\label{sect:pointreg}
To determine the space $\mathcal{W}$ in \eqref{eq:approxest}, we have to understand which is the proper convergence estimate for the point evaluation. Once this has been settled, we can address under which conditions the solution $u$ to \eqref{eq:modelpb} belongs to $L^2(\pspace_J,\mathcal{W})$ with a $J$-independent bound.

A first observation is that, for every $J$, $\psubspace_J$ as defined in \eqref{eq:hyperplane} is a zero measure set in $\pspace_J$ (as it is a hyperplane). Thus, in \eqref{eq:approxest} we can set $\mathcal{N}_{\pspace_J}=\psubspace_J$, and it is sufficient to determine the convergence estimate in the case that $\Bx_0\in\din$ or $\Bx_0\in\doutR$.

A second observation is that \eqref{eq:approxest} has to be established for every $\By\in\pspace_J\setminus\mathcal{N}_{\pspace_J}$ fixed. Therefore, instead of studying the convergence for $u(\By;\Bx_0)$, we can work in the reference configuration and study the convergence estimate for $\hat{u}(\By;\hat{\Bx}_0)$ with $\hat{\Bx}_0=\Phi^{-1}(\By;\Bx_0)$. Note, however, that truncating \eqref{eq:radiusy} at the term with index $J$ means setting to zero all the entries of $\By$ in position grater than $J$, and, for a fixed realization, $\hat{\Bx}_0$ depends on $J$, as the mapping $\Phi^{-1}(\By;\cdot)$ does. For this reason, we need a convergence estimate for $\hat{u}(\By;\hat{\Bx}_0)$ which is uniform in the second argument (that is, independent of $\hat{\Bx}_0$).

% Under the previous considerations, we address the finite element convergence of $\uref(\By;\hat{\Bx}_0(\By))$ for every $\By\in\pspace_J$ and every $J\in\bbN$, with $\uref$ solving \eqref{eq:varform}.

An option to determine the convergence rate would be to consider the point evaluation as the Dirac delta functional $\delta_{\xref_0}(\vref)=\vref(\xref_0)$, which is a bounded on the space $H^{1+\varepsilon}(\hdin)\cup H^{1+\varepsilon}(\hdoutR)$, for any $\varepsilon>0$. If $s$ is the finite element convergence rate in $H^1(D_{R_{out}})$  (with respect to the meshwidth)  for the solution to \eqref{eq:varform}, then we would infer the convergence rate $s-\varepsilon$ for the point evaluation, for $\varepsilon>0$ arbitrarily small \cite{BBP}.

Such convergence rate is not optimal, though. If we consider the finite element convergence estimates in the $L^{\infty}$-norm, it is possible to achieve a convergence rate of $h^{s+1}(\log|h|)^{\bar{s}}$ for $h\rightarrow 0$, with $\bar{s}=1$ if $s=1$, and $\bar{s}=0$ if $s\geq 2$ \cite{Schatz98}, provided the solution has $W^{s+1,\infty}$-regularity. We will state these convergence estimates rigorously in subsection \ref{ssect:fecvg}, after having established, in the next subsection, the regularity of the solution to \eqref{eq:varform} in the space
\begin{equation}\label{eq:w}
\mathcal{W}=C^k(\overline{\hdin})\cup C^{k}(\overline{\hdoutR}),
\end{equation}
for some $k\geq 2$, equipped with the norm $\lVert\cdot\rVert_{\mathcal{W}}:=\max\left\{\lVert\cdot\rVert_{C^{k}(\overline{\hdin})},\lVert\cdot\rVert_{C^{k}(\overline{\hdoutR})}\right\}$.

\subsection{Space regularity of the solution}\label{ssect:spacereg}

To obtain upper bounds on $C^k(\overline{\hdin})\cup C^{k}(\overline{\hdoutR})$ for some $k\geq 2$, we consider Schauder estimates (see \cite[Ch. 6]{GT} and \cite[Ch. 6]{WYW}), as they require milder space regularity of the coefficients \eqref{eq:coeffshat} than Sobolev estimates \cite[Ch. 8]{GT} followed by an application of the Sobolev embedding theorem \cite[Thm. 7.26]{GT}.

Starting from $k=2$, we notice that it is not possible to obtain bounds on the norm of $\uref$ in $C^2(\overline{\hdin})\cup C^2(\overline{\hdoutR})$, because estimates in this last norm are in general not well defined (cf. p.52 and Problem 4.9 in \cite{GT}). For this reason, we state estimates in the H\"older spaces $C^{k,\hexp}_{pw}(\overline{D_{R_{out}}}):=C^{k,\hexp}(\overline{\hdin})\cup C^{k,\hexp}(\overline{\hdoutR})$, $\hexp\in(0,1)$.

\begin{theorem}\label{thm:ureg}
 Let $\hexp\in(0,1)$ and $k\geq 2$, and let $\hat{\Gamma}$ and $\partial D_{R_{out}}$ be simple closed curves of class $C^{k,\hexp}$. Let the coefficients in \eqref{eq:varform} be such that, for every $\By\in\pspace_J$: $\hat{\alpha}(\By;\cdot)$ has $J$-, $\By$- and $\hat{\Bx}$-uniform upper and lower bounds $\Lambda_{min}$ and $\Lambda_{max}$ on its singular values, $\lVert\hat{\alpha}(\By;\cdot)\rVert_{C_{pw}^{k-1,\hexp}(\overline{D_{R_{out}}})}\leq C_{\alpha}$ and $\lVert\hat{\kappa}^2(\By;\cdot)\rVert_{C_{pw}^{k-2,\hexp}(\overline{D_{R_{out}}})}\leq C_{\kappa}$, with $C_{\alpha}$ and $C_{\kappa}$ independent of $J\in\bbN$ and $\By\in\pspace_J$. Then the solution $\uref$ to \eqref{eq:varform} is such that
 \begin{equation}\label{eq:holderbound1}
  \lVert\uref(\By)\rVert_{C_{pw}^{k,\hexp}(\overline{D_{R_{out}}})}\leq C\left(\lVert\uref(\By)\rVert_{C^0(\overline{D_{R_{out}}})}+\lVert u_i\rVert_{C^{k,\hexp}(\overline{D_{R_{out}}})}\right),
 \end{equation}
 with a constant $C=C(k,\hexp,C_{\alpha},C_{\kappa},\Lambda_{min},\Lambda_{max})$ independent of $J\in\bbN$ and $\By\in\pspace_J$.
\end{theorem}
\begin{proof}
 This result is a slight modification for interface problems of the Schauder estimates in \cite[Ch. 6]{GT} and \cite[Ch. 6]{WYW}, taking care of $J$- and $\By$-independence in the norm bounds. We refer to Appendix B for details.
\end{proof}

We have already seen in the proof of Proposition \ref{prop:conty} that the smoothness of the PDE coefficients in \eqref{eq:varform} derives from the smoothness of the radius. More precisely, if the sequence $\left(\beta_j\right)_{j\geq 1}$ in \eqref{eq:radiusy} fulfills Assumption \ref{Arcoeffs} and the nominal radius $r_0$ is sufficiently smooth, then, for every $J\in\bbN$ and every $\By\in\pspace_J$, $r(\By)\in C^{k,\hexp}_{per}([0,2\pi))$, with $J$- and $\By$-independent norm bound and
\begin{equation}\label{eq:k}
\begin{cases}
 k=\lfloor \frac{1}{p}-1\rfloor,\text{ and } \hexp<\frac{1}{p}-1-k & \text{if } \frac{1}{p}-1 \text{ is not integer},\\
 k=\frac{1}{p}-2,\text{ and any } \hexp\in(0,1) &\text{otherwise}.
\end{cases}
\end{equation}
Using the expression \eqref{eq:phicircle} for the domain mapping and \eqref{eq:coeffshat} for the PDE coefficients, we have $k\geq 2$ in Theorem \ref{thm:ureg} if $p<\tfrac{1}{3}$ in Assumption \ref{Arcoeffs}. The bounds on the singular values of $\hat{\alpha}(\By;\cdot)$ hold if Assumption \ref{Ak1k2} does.

To bound $\lVert\uref(\By)\rVert_{C^0(\overline{D_{R_{out}}})}$, we note that, if Assumption \ref{Ak1k2} holds (and $r_0$ is sufficiently smooth), and if $p<\tfrac{1}{3}$ in Assumption \ref{Arcoeffs}, then $\lVert\uref(\By)\rVert_{H^2(D_{in})\cup H^2(D_{out,R_{out}})}$ for every $\By\in\pspace_J$, and the norm has a $J$-independent bound \cite[Thm. 6.1.7]{LSthesis}. Then the Sobolev embedding theorem \cite[Thm. 7.26]{GT} and the continuity of $\uref$ across $\hat{\Gamma}$ imply
\begin{equation}\label{eq:c0bounduref}
  \lVert\uref(\By)\rVert_{C^0(\overline{D_{R_{out}}})}\leq C\left(\lVert u_i\rVert_{H^{\frac{3}{2}}(\partial D_{R_{out}})}+\Big\lVert \dfrac{\partial u_i}{\partial \Bn_{out}}\Big\rVert_{H^{\frac{1}{2}}(\partial D_{R_{out}})}\right),
\end{equation}
for every $J\in\bbN$ and $\By\in\pspace_J$, with a constant $C=C(\gamma_{-})$ independent of $J\in\bbN$ and $\By\in\pspace_J$, but dependent on the coercivity constant $\gamma_{-}$ of the bilinear form in \eqref{eq:varform} (which is $J$- and $\By$-independent, see \cite[Lemma 3.2.5]{LSthesis}).

% A $J$- and $\By$-uniform bound on $\lVert\uref(\By)\rVert_{C^0(\overline{D_{R_{out}}})}$ can be obtained, as in the proof of Proposition \ref{prop:conty}\ls{Use Sobolev embedding}, from Lemma 2 in \cite{Joc} and a slight modification of Theorem 3.1 in \cite{BGL}. Leaving again the details to Appendix A, we have that, under Assumptions \ref{Ak1k2}, \ref{Aphi} and \ref{Arcoeffs} (with $r_0$ sufficiently smooth):
% \begin{equation}\label{eq:c0bounduref}
%  \lVert\uref(\By)\rVert_{C^0(\overline{D_{R_{out}}})}\leq C \lVert u_i \rVert_{C^{1,\hexp}(\overline{D_{R_{out}}})}.
% \end{equation}
% The constant $C=C(\gamma_{-},C_{\alpha}^0,C_{\kappa}^0, C_{\Phi})$ in \eqref{eq:c0bounduref} is independent of $J\in\bbN$ and $\By\in\pspace_J$, but it depends on: the $J$- and $\By$-independent bound $\gamma_{-}$ on the coercivity constant of the bilinear form associated to \eqref{eq:varform} (this bound is ensured by Assumptions \ref{Ak1k2} and \ref{Aphi}), the $J$ and $\By$-independent bounds $C_{\alpha}^0$ and $C_{\kappa}^0$ on $\lVert\hat{\alpha}(\By;\cdot)\rVert_{C_{pw}^{\hexp}(\overline{D_{R_{out}}})}$ and $\lVert\hat{\kappa}(\By;\cdot)\rVert_{C_{pw}^{\hexp}(\overline{D_{R_{out}}})}$, for some $\hexp\in(0,1)$ (the existence of such $\hexp$ is ensured by Assumption \ref{Arcoeffs}), the $J$ and $\By$-independent bound $C_{\Phi}$ on $\lVert \Phi(\By;\cdot) \rVert_{C_{pw}^{1,\hexp}(\overline{\doutR})}$ (ensured by Assumption \ref{Aphi}).

We arrive then to the following important corollary to Theorem \ref{thm:ureg}.

\begin{corollary}\label{cor:uwbound}
 Let the sequence $\left(\beta_j\right)_{j\geq 1}$ in \eqref{eq:radiusy} fulfill Assumption \ref{Arcoeffs} with $p<\tfrac{1}{3}$, let $r_0\in C^{k,\hexp}_{per}([0,2\pi))$ with $k\geq 2$ and $\hexp\in(0,1)$, and let Assumption \ref{Ak1k2} hold. Then the solution $\uref$ to \eqref{eq:varform} belongs to $C_{pw}^{k,\hexp}(\overline{D_{R_{out}}})$ with $k\geq 2$ and $\hexp\in(0,1)$ as in \eqref{eq:k}, and 
 \begin{equation}\label{eq:uwbound}
    \lVert\uref(\By)\rVert_{C_{pw}^{k,\hexp}(\overline{D_{R_{out}}})}\leq C \lVert u_i\rVert_{C^{k,\hexp}(\overline{D_{R_{out}}})}.
 \end{equation}
The constant $C=C\left(k,\hexp,\gamma_{-},C_{\alpha},\right.$ $\left.C_{\kappa},\sigma_{min},\sigma_{max}\right)$ is independent of $J\in\bbN$ and $\By\in\pspace_J$ (here $\sigma_{min}$ and $\sigma_{max}$ are the $J$- and $\By$-independent bounds on the singular values of $D\Phi$, and the other constants are as defined in this subsection).
 
 In particular, $\uref\in L^2(\pspace_J,\mathcal{W})$ with a $J$-independent bound and $\mathcal{W}$ as in \eqref{eq:w} (with $k$ from \eqref{eq:k}).
\end{corollary}

\subsection{Finite element convergence for the point evaluation}\label{ssect:fecvg}
We consider the finite element space of globally continuous ansatz functions which are polynomials of degree $s$ on each element of a quasi-uniform mesh with meshsize $h_l>0$ on the \emph{nominal} configuration. We denote this space by $\mathcal{S}^s_{h_l}(D_{R_{out}})$. Setting $\mathcal{X}^l:=\mathcal{S}^s_{h_l}(D_{R_{out}})$ and considering a nested sequence of meshes and thus a geometric sequence of meshsize parameters $(h_l)_{l\geq 0}$, we are in the  framework for MLMC as in \eqref{eq:fespaces}.

Our starting point is the $L^{\infty}$-estimate for finite element solutions to elliptic boundary value problems. 

% Given a bounded domain  a parameter-dependent bilinear form $A_{bvp}$ such that, for every $\By\in\pspace_J$ and $J\in\bbN$, $A_{bvp}(\By):  H^1(\mathfrak{D})\times \rightarrow \bbR$ is defined as

% for $w,v\in H^1(\mathfrak{D})$. For every $J\in\bbN$, every $\By\in\pspace_J$ and every $\xref\in\mathfrak{D}$, $\aref_A\in\bbR^{n\times n}$, $\bref_A\in\bbR^n$ and $\cref_A\in\bbR$. Let $\mathcal{S}^{h_l}_s(\mathfrak{D})$, $l\in\bbN$, be the finite element space of globally continuous, piecewise $s$th order polynomial functions on a quasi uniform mesh on $\mathfrak{D}$ with meshsize $h_l$. We assume that, for every $J\in\bbN$ and every $\By\in\pspace_J$, $w(\By)\in H^1(\mathfrak{D})$ and its finite element approximations $w_{h_l}(\By)\in \mathcal{S}^s_{h_l}(\mathfrak{D})$, $l\in\bbN$, satisfy
% \begin{equation}\label{eq:ellipticbvp}
%  
% \end{equation}
% 
% Then we have the following approximation result for the boundary value problem \eqref{eq:ellipticbvp}:

\begin{theorem}[Theorem 2.1 in \cite{Schatz98}]\label{thm:schatz}
For a domain $\mathfrak{D}\subset\bbR^n$, $n\geq 1$, we consider the bilinear form
\begin{equation*}
 a_{bvp}(\By;\hat{w},\hat{v}):=\int_{\mathfrak{D}}\aref_A(\By;\xref)\hat{\nabla} \hat{w}\cdot \hat{\nabla} \hat{v} + \hat{\beta}_A(\By;\xref)\cdot\hat{\nabla} \hat{w} \hat{v} + \cref_A(\By;\xref)\hat{w}\hat{v}\dd\xref,\quad \By\in\pspace_J, J\in\bbN,
\end{equation*}
for every $\hat{w},\hat{v}\in H^1(\mathfrak{D})$, with $\aref_A(\By;\xref)\in\bbR^{n\times n}$, $\hat{\beta}_A(\By;\xref)\in\bbR^n$ and $\cref_A(\By;\xref)\in\bbR$ for every $\hat{\Bx}\in\mathfrak{D}$, $J\in\bbN$, $\By\in\pspace_J$. For $s\geq 1$, let the following assumptions be satisfied:
\begin{enumerate}[(i)]
 \item $\partial\mathfrak{D}$ is of class $C^{s+4}$;
 \item for every $J\in\bbN$ and every $\By\in\pspace_J$, $\aref_A\in C^{s+3}(\overline{\mathfrak{D}})$ and $\hat{\beta}_A, \cref_A\in C^{s+2}(\overline{\mathfrak{D}})$, with $J$- and $\By$-independent bounds on the norms\footnote{According to Remark 1.1 in \cite{Schatz98}, we would need $\aref_A\in C^{s+2}(\mathfrak{D})$. However, the reference provided there for this claim is \cite{Kras}, according to which (see p.107) we need the higher order coefficient in $C^{s+2}(\overline{\mathfrak{D}})$ if the operator is not in divergence form, and thus we need the higher order coefficient in $C^{s+3}(\overline{\mathfrak{D}})$ when considering the operator in divergence form.};
 \item $a_{bvp}(\cdot,\cdot)$ has a $J$- and $\By$-uniform lower, positive bound on the coercivity constant;
 \item the matrix $\aref_A$  has a $J$- and $\By$- and $\hat{\Bx}$-uniform lower, positive bound on the ellipticity constant.
\end{enumerate}
Let $\hat{w}(\By;\cdot)\in C^{1}(\overline{\mathfrak{D}})$ and $\hat{w}_{h_l}(\By;\cdot)\in\mathcal{S}^s_{h_l}(\mathfrak{D})$ satisfy $a_{bvp}(\By;\hat{w}(\By)-\hat{w}_{h_l}(\By),\hat{v}_{h_l})=0$ for all $\hat{v}_{h_l}\in \mathcal{S}^s_{h_l}(\mathfrak{D})$. Then there exists a constant $C$, independent of  $\hat{w}$, $\hat{w}_{h_l}$, $l\in\bbN$, of $J\in\bbN$ and of $\By\in\pspace_J$ such that
 \begin{equation}\label{eq:linfbvp}
  \lVert \hat{w}(\By)-\hat{w}_{h_l}(\By)\rVert_{L^{\infty}(\mathfrak{D})}\leq Ch_l\left(\log \frac{1}{h_l}\right)^{\bar{s}}\inf_{\chi\in\mathcal{S}^s_{h_l}(\mathfrak{D})}\lVert \hat{w}(\By)-\chi\rVert_{C^1(\overline{\mathfrak{D}})},
 \end{equation}
for every $l\in\bbN$, $J\in\bbN$ and $\By\in\pspace_J$, with $\bar{s}=1$ if $s=1$ and $\bar{s}=0$ if $s\geq 2$.
\end{theorem}
\begin{proof}
 Repeating the proof of Theorem 2.1 in \cite{Schatz98}, it is easy to check that, under the assumption of $J$- and $\By$-uniform bounds on the norms of the coefficients and on the coercivity and ellipticity constants of the bilinear form, the constant $C$ in \eqref{eq:linfbvp} is $J$- and $\By$-independent, too. 
\end{proof}

To be more precise, Theorem 2.1 in \cite{Schatz98} provides a sharper estimate using a weighted $W^{1,\infty}$-norm instead of the $C^1$-norm on the right-hand side. However, what we are interested in is the convergence rate rather than a quantitative estimate, and for this the $C^1$-norm is sufficient. Moreover, an extension of $L^{\infty}$-estimates to the case that $\mathfrak{D}$ is a convex polygon can be found in \cite{GLRS} (although in the case of constant coefficients).

Going back to our model problem, in the variational formulation \eqref{eq:varform}, differently from the assumptions of Theorem \ref{thm:schatz}, the coefficients are smooth in $\hdin$ and in $\hdoutR$, but in general they are not smooth across $\gammaref$. We can expect that, if the interface $\gammaref$ is resolved `well enough' (in a sense to be made precise), then we still achieve the same convergence rates as in Theorem \ref{thm:schatz} when discretizing our interface problem. Finite element estimates taking into account the resolution of the interface have been proven in \cite{LMWS} for the convergence in the $H^1$- and $L^2$-norms. It is plausible that similar results hold for the convergence in the $L^{\infty}$-norm, but, to the author's knowledge, they seem not to be available in the literature. Also in more recent applications of $L^{\infty}$-estimates to interface problems \cite{GSS}, the issue of the approximation of $\gammaref$ is not addressed. Since proving it goes far beyond the scope of this paper, we formulate the following assumption, and test numerically its plausibility for our model problem in the next subsection.

\begin{assumption}\label{Ainterfcvg}
If as domain $\mathfrak{D}$ we consider $D_{R_{out}}=\hdin\cup\gammaref\cup\hdoutR$, if $\aref_A\in C^{s+3}_{pw}(\overline{D_{R_{out}}})$ and $\bref_A, \cref_A\in C^{s+2}_{pw}(\overline{D_{R_{out}}})$ with $J$- and $\By$-independent norm bounds, and if every finite element mesh provides a piecewise $s^{th}$-order polynomial approximation for $\gammaref$, then the result of Theorem \ref{thm:schatz} still holds, in the sense that, for $\hat{w}\in C^1_{pw}(\overline{\mathfrak{D}})$ and $\hat{w}_{h_l}\in\mathcal{S}^s_{h_l}(\mathfrak{D})$ satisfying $a_{bvp}(\By;\hat{w}(\By)-\hat{w}_{h_l}(\By),\hat{v}_{h_l})=0$ for all $\hat{v}_{h_l}\in \mathcal{S}^s_{h_l}(\mathfrak{D})$:
%the solution $\hat{u}$ to the associated interface problem and its finite element approximations $\uref_{h_l}$, $l\in\bbN$, fulfill
 \begin{equation}\label{eq:linfinterface}
  \lVert \hat{w}(\By)-\hat{w}_{h_l}(\By)\rVert_{L^{\infty}(D_{R_{out}})}\leq Ch_l\left(\log \frac{1}{h_l}\right)^{\bar{s}}\inf_{\chi\in \mathcal{S}^s_{h_l}(D_{R_{out}})}\lVert \hat{w}(\By)-\chi\rVert_{C^1(\overline{\hdin})\cup C^{1}(\overline{\hdoutR})},
 \end{equation}
 with $\bar{s}$ as in Theorem \ref{thm:schatz} and $C$ a $J$- and $\By$-independent constant.
\end{assumption}
If we set $\aref_A=\hat{\alpha}$, $\bref_A\equiv 0$ and $\cref_A=\hat{\kappa}^2$, then \eqref{eq:linfinterface} gives us the convergence rate for the solution to \eqref{eq:varform} (the boundary condition with the $\operatorname{DtN}$ map is smooth).

As in Corollary \ref{cor:uwbound}, we can deduce the regularity of the coefficients $\hat{\alpha}$ and $\kappa^2$ in \eqref{eq:coeffshat} from the decay of the coefficients in the radius expansion \eqref{eq:radiusy}. Combining this with Corollary \ref{cor:uwbound} itself, we obtain
\begin{theorem}\label{thm:finalcvg}
  Let the sequence $\left(\beta_j\right)_{j\geq 1}$ in \eqref{eq:radiusy} fulfill Assumption \ref{Arcoeffs} with $p<\frac{1}{s+5}$, $s\in\bbN$, and let the wavenumbers fulfill Assumption \ref{Ak1k2}. Let Assumption \ref{Ainterfcvg} hold and let the finite element meshes provide a piecewise $s^{th}$-order polynomial approximation to $\gammaref$. Then the  finite element solutions $\uref_{h_l}\in\mathcal{S}^s_{h_l}(D_{R_{out}})$ to \eqref{eq:varform}, $l\in\bbN$, satisfy:
  \begin{equation}\label{eq:finalcvg}
  \lVert \uref(\By)-\uref_{h_l}(\By)\rVert_{L^{\infty}(D)}\leq Ch_l^{s+1}\left(\log \frac{1}{h_l}\right)^{\bar{s}}\lVert \uref(\By)\rVert_{C^{s+1}(\overline{\hdin})\cup C^{s+1}(\overline{\hdoutR})},
  \end{equation}
with $\bar{s}$ as in Theorem \ref{thm:schatz} and a constant $C$ independent of $l\in\bbN$, of $J\in\bbN$ and of $\By\in\pspace_J$ (but dependent on the mesh regularity parameters, on some $J$- and $\By$-independent bounds on the norms of the coefficients in \eqref{eq:varform} and on a $J$- and $\By$-independent lower bound on the coercivity constant).

Moreover, the norm on the right-hand side in \eqref{eq:finalcvg} is bounded independently of $J\in\bbN$ and $\By\in\pspace_J$.
\end{theorem}
\begin{proof}
 The decay of the sequence $\left(\beta_j\right)_{j\geq 1}$ ensures that, for every $\By\in\pspace_J$ and every $J\in\bbN$, the radius \eqref{eq:radiusy} belongs to $C^{s+4,\hexp}_{per}([0,2\pi)$ for some $\hexp\in (0,1)$, with a $J$- and $\By$-independent norm bound, see subsection \ref{ssect:spacereg}. Proceeding as in the proof of Corollary \ref{cor:uwbound}, the smoothness of the mapping $\Phi$ ensures that, for every $\By\in\pspace_J$ and every $J\in\bbN$, $\hat{\alpha}(\By;\cdot)$ and $\hat{\kappa}^2(\By;\cdot)$ belong to $C^{s+3,\hexp}_{pw}(\overline{D_{R_{out}}})$, with $J$- and $\By$-independent norm bounds. The $J$- and $\By$-independent lower and upper bounds on the singular values of $D\Phi$ ensure a $J$- and $\By$-independent lower bound on the ellipticity constant of $\hat{\alpha}$, which, together with Assumption \ref{Ak1k2}, implies a $J$- and $\By$-uniform lower bound on the coercivity constant of the bilinear form in \eqref{eq:varform} \cite[Lemma 3.2.5]{LSthesis}. Then Theorem \ref{thm:schatz}, together with Assumption \ref{Ainterfcvg}, implies the estimate \eqref{eq:linfinterface}. Finally, the interpolation properties of the spaces $\mathcal{S}^s_{h_l}(D_{R_{out}})$ ensure that
 \begin{equation}
  \inf_{\chi\in\mathcal{S}^s_{h_l}(D_{R_{out}})}\lVert \uref(\By)-\chi\rVert_{C^1(\overline{\hdin})\cup C^{1}(\overline{\hdoutR})}\leq C' h_l^{s}\lVert\uref(\By)\rVert_{C^{s+1}(\overline{\hdin})\cup C^{s+1}(\overline{\hdoutR})},
 \end{equation}
for a constant $C'$ dependent on the mesh regularity parameters but clearly not on $J\in\bbN$ and $\By\in\pspace_J$. The norm on the right-hand side has a $J$- and $\By$-independent bound thanks to Corollary \ref{cor:uwbound}.
\end{proof}

In Theorem \ref{thm:finalcvg} we have not formulated any regularity assumption on $\partial D_{R_{out}}$ as we assume it to be a circle, and thus of class $C^{\infty}$. %Also, the interface $\gammaref$ is $O(h^{2s})$-resolved if it is a $C^{s+1}$-curve approximated by a $s$th order spline, see \cite{LMWS}.

\begin{corollary}
 Under the assumptions of Theorem \ref{thm:finalcvg}, assumption $(ii)$ of Proposition \ref{prop:cvgrates} holds with $t=2-\varepsilon$ and any $\varepsilon>0$ for linear finite elements, and with $t=s+1$ for Lagrangian finite elements of degree $s$ with $s\geq 2$. 
\end{corollary}

\begin{remark}[Regularity of coefficients]\label{rmk:milderass}
In order to have a $J$- and $\By$-independent bound on $\lVert \uref(\By)\rVert_{C^{s+1}(\overline{\hdin})\cup C^{s+1}(\overline{\hdout})}$, it is sufficient that the decay parameter $p$ for the sequence $(\beta_j)_{j\geq 1}$ satisfies $p<\frac{1}{s+2}$, see Theorem \ref{thm:ureg}. The stronger requirement that $p<\frac{1}{s+5}$ is due to a technicality in the proof of the $L^{\infty}$-estimate \eqref{eq:linfbvp} presented in \cite{Schatz98}, requiring stronger smoothness on the PDE coefficients. In particular, it is needed for the decay estimate of the Green's function associated to \eqref{eq:varform}. One might ask whether such stronger requirement is necessary. 

The decay estimate on the Green's function and the space regularity required on the coefficients is reported Lemma 1.1 and Remark  1.1 of \cite{Schatz98}, which refer to \cite{Kras} (whose assumptions on the coefficients can be found on p. 107). It might be, however, that the estimate reported in \cite{Kras} still holds on milder assumptions on the regularity of the boundary and of the coefficients (cf. estimate (8.3) and Theorem 19.VII in \cite{Mir}, and Theorem 8.1.11, Corollary 8.1.12 and Remark 8.1.13 in \cite{BS}).
\end{remark}

\subsection{Finite element convergence: numerical experiments}\label{ssect:fecvgnumexp}
In this subsection we show numerical results to validate the convergence estimates of the previous subsections. We address the case $s=1$ in Theorem \ref{thm:finalcvg}, because in the MLMC simulations we will use linear finite elements.

As in subsection \ref{ssect:smolyak}, we work with non-dimensional quantities. In \eqref{eq:coeffs}, we set $\alpha_1=4$, $\alpha_2=1$, $\kappa_1=\kappa_0$ and $\kappa_2=2\kappa_0$, where $\kappa_0=209.44$ denotes the wavenumber in free space. The incident wave $u_i(\Bx)=e^{j\kappa_1\Bd\cdot\Bx}$ is coming from the left, that is $\Bd=(1,0)$. The $\operatorname{DtN}$ map is approximated truncating the domain with a circular Perfecly Matched Layer (PML, see \cite{B,CM}) starting at $R_{out}=0.055$, with thickness $0.02$ and absorption coefficient (or damping parameter) $0.5$ \cite{CM}. The nominal geometry is a circle with radius $r_0=0.01$. In \eqref{eq:radiusy}, we consider $\beta_{2j-1}=\beta_{2j}=\frac{0.1 r_0}{j^{\frac{1}{p}}}$, $j=1,\ldots,\tfrac{J}{2}$, with three decays $p=\tfrac{1}{2},\tfrac{1}{3},\tfrac{1}{4}$, and four dimension truncations $J=8,16,32,64$. The case $J=8$ will not be used in the numerical experiments for MLMC, but we consider it here in order to better investigate the dependence of the convergence estimates of Theorem \ref{thm:finalcvg} on the dimension of the parameter space. The domain mapping is \eqref{eq:phicircle} with mollifier
\begin{equation}\label{eq:chi2}
\chi(\hat{\Bx})=
\begin{cases}
 0 & \text{if }\lVert\hat{\Bx}\rVert \leq \frac{r_0}{4},\\
 \frac{\lVert\hat{\Bx}\rVert-\frac{r_0}{4}}{r_0-\frac{r_0}{4}} & \text{if }\frac{r_0}{4}<\lVert\hat{\Bx}\rVert\leq r_0,\\
 \frac{R_{out}-\lVert\hat{\Bx}\rVert}{R_{out}-r_0} & \text{if }r_0\leq\lVert\hat{\Bx}\rVert\leq R_{out}.
\end{cases} 
\end{equation}
The non-smoothness of this mollifier at $\lVert \hat{\Bx}\rVert=\frac{r_0}{4}$ can be easily handled treating the circle of radius $\tfrac{r_0}{4}$ as an additional interface resolved by the finite element meshes, cf. Assumption \ref{Ainterfcvg}, Theorem \ref{thm:finalcvg} and \cite{LMWS}.

We consider six nested, unstructed quasi-uniform meshes on the reference configuration, with $581$, $2250$, $8855$, $35133$, $139961$ and $558705$ degrees of freedom, respectively, and use an additional refinement, with $2232545$ degrees of freedom, to obtain reference solutions. The circles of radius $r_0$ and $\tfrac{r_0}{4}$ have been approximated by piecewise linear curves.

Each finite element solution has been obtained using the NGSolve finite element library\footnote{http://sourceforge.net/apps/mediawiki/ngsolve} (version 5.1), coupled to the MKL PARDISO\footnote{https://software.intel.com/en-us/intel-mkl. See also http://www.pardiso-project.org/ for other versions of the PARDISO solver.} direct solver to solve the algebraic system resulting from the discretization.

We study the convergence of the point evaluation of $\Re u(\By;\cdot)$, the real part of the solution to \eqref{eq:modelpb}, for two points in the actual configuration: $\Bx_0^1=(r_0,0)$ and $\Bx_0^2=(0,r_0)$. For $\Bx_0^1$, we consider the realization $\By$ with all entries set to $1$, so that, for every $J$, $\hat{\Bx}_0^1=\Phi^{-1}(\By;\Bx_0^1)$ is located in $\hdin$ (although the coordinates of $\hat{\Bx}_0^1$ depend on $J$). For $\Bx_0^2$, we consider the realization $\By$ with all entries set to $-1$, so that, for every $J$, $\hat{\Bx}_0^2=\Phi^{-1}(\By;\Bx_0^2)\in\hdoutR$.

The results are reported in Figures \ref{fig:fecvgin} and \ref{fig:fecvgout}. From Theorem \ref{thm:finalcvg}, we expect a convergence rate close to $2$ with respect to meshwidth, and thus a rate close to $1$ with respect to the number of degrees of freedom $N_{dof}$, for every decay $p$ and every dimension $J$. However, we expect the constant multiplying the rate in \eqref{eq:linfinterface} (incorporating the norm of the solution) to have a $J$-independent upper bound only for $p<\tfrac{1}{6}$, and thus in none of our test cases. Taking into account Remark \ref{rmk:milderass}, we could expect $J$-independence of the constant for $p<\tfrac{1}{3}$. Figures \ref{fig:fecvgin} and \ref{fig:fecvgout} show that the convergence rate predicted by the theory is correct, but the constant seems to have a $J$-independent upper bound for \emph{all} values of $p$ considered. 

The last observation can indicate two things. A possibility is that our theory of subsections \ref{ssect:spacereg} and \ref{ssect:fecvg} is not sharp and can be improved. Another possible interpretation is that, due to the decay of the coefficient sequence $(\beta_j)_{j\geq 1}$, there is a `natural' dimension truncation from the mesh, that does not allow to track the high frequency perturbations. Furthermore, because of the nonlinear dependence of the Q.o.I. on the high-dimensional parameter, it could be that, also when the mesh is able to capture some high-frequency shape variations, they contribute to a variation in the Q.o.I. which is smaller than the discretization error. To give an idea about the size of the shape perturbations, the maximum shape variation for $p=\tfrac{1}{2}$ is around $0.3055 r_0$ for $J=16$ and $0.3169 r_0$ for $J=32$, which means that the harmonics added from $J=16$ to $J=32$ contribute for $1.14\cdot 10^{-4} r_0$ to the maximum shape variation. The meshsize around $r_0$ is instead of the order of $1.3\cdot 10^{-5}$ on the finest mesh. In Figures \ref{fig:fecvgin} and \ref{fig:fecvgout}, for $p=\tfrac{1}{2}$, we see indeed a slight difference in the convergence curves at the finest level, but it is negligible. Passing from $J=32$ to $J=64$, the contribution of the higher order shape variations is even smaller than from $J=16$ to $J=32$, and the convergence curves are indistinguishable. To further investigate the influence of shape variations, we may ask ourselves how far are the solutions corresponding to $J=16$, $J=32$ and $J=64$, for a fixed decay $p$ of the coefficient sequence. The fact that the convergence lines are very close to each other gives us no information about this. We have performed a crossed comparison for each of the cases $\tfrac{1}{p}=2$ and $\tfrac{1}{p}=3$: we have considered as reference solution the one obtained on the finest grid for $J=16$, and studied the convergence to this value for the solutions corresponding to $J=32$ and $J=64$. The outcome for the evaluation at  $\Bx_0^2=(0,r_0)$ and with all entries of $\By$ set to $-1$ is shown in Figure \ref{fig:fecvgcross}. The left plot in Figure \ref{fig:fecvgcross} tells us that, for each of the cases $\tfrac{1}{p}=2$ and $\tfrac{1}{p}=3$, the solution for $J=32$ converges to a value that differs from the exact solution for $J=16$ by a quantity that is some orders of magnitude smaller than the finite element error on the last mesh considered. The right plot in Figure \ref{fig:fecvgcross} shows instead that, for $\tfrac{1}{p}=2$, the exact solution for $J=16$ and the exact solution for $J=64$ differ by a quantity of the order of $10^{-4}$, and this affects only the convergence on the last two meshes. Returning to the left plot in Figure \ref{fig:fecvgout}, we see that the line for $J=64$ slightly departs from the line for $J=16$. This does not happen for $J=64$ and the faster decay $\tfrac{1}{p}=3$, and in the correponding line in the right plot of Figure \ref{fig:fecvgcross} we observe convergence until the last mesh considered. From these last experiments we can conclude that the high frequency perturbations of the shape can be observed only when going to very fine meshes, supporting the hypothesis of `natural' dimension truncation coming from the discretization.

Finally, we mention that the achievement of the full convergence rate prescribed by Theorem \ref{thm:schatz} when using a piecewise linear approximation for $\gammaref$ supports the validity of Assumption \ref{Ainterfcvg}.

 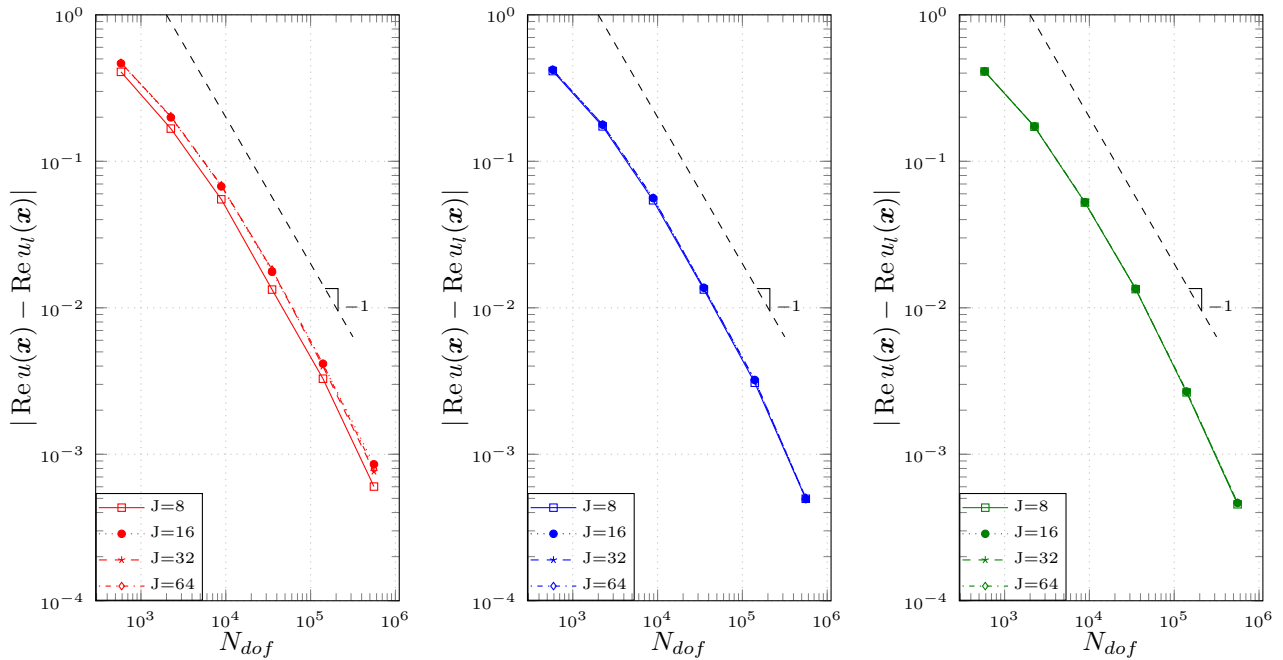
\begin{figure}
  \begin{center}
\begin{tikzpicture}
  \begin{loglogaxis}[
 xlabel={\footnotesize $N_{dof}$},
 ylabel={\footnotesize $| \Re u(\Bx)-\Re u_l(\Bx)|$},
 grid=major, legend entries={{J=8},{J=16},{J=32},{J=64}}, ymin=1e-4,ymax=1
  ]
    \settable{ndofs.txt}{FEcvg010_z2_d8.txt}
    \addplot[mark=square,mark size=1.4pt,Red,solid,mark options={solid}] table[x=data,y expr = \thisrowno{0}]{\datatable};
      \settable{ndofs.txt}{FEcvg010_z2_d16.txt}
    \addplot[mark=*,mark size=1.4pt,Red,dotted,mark options={solid}] table[x=data,y expr = \thisrowno{0}]{\datatable};
      \settable{ndofs.txt}{FEcvg010_z2_d32.txt}
    \addplot[mark=star,mark size=1.4pt,Red,dashed,mark options={solid}] table[x=data,y expr = \thisrowno{0}]{\datatable};
      \settable{ndofs.txt}{FEcvg010_z2_d64.txt}
    \addplot[mark=diamond,mark size=1.4pt,Red,dash pattern=on 0.5pt off 2pt on 2pt off 2pt,mark options={solid}] table[x=data,y expr = \thisrowno{0}]{\datatable};
    
        \addplot[black,domain=2*10^3:10^5.5,samples=97,style=dashed]{x^(-1)/(0.5*10^(-3))*(1)}
          coordinate [pos=0.85] (A)
  coordinate [pos=0.92] (B)
  ;
  \draw (A) -| (B) node[pos=0.9,anchor= west] {\hspace{-0.2cm} $\scriptscriptstyle -1$};
  \end{loglogaxis}
\end{tikzpicture}
\begin{tikzpicture}
  \begin{loglogaxis}[
 xlabel={\footnotesize $N_{dof}$},
 ylabel={\footnotesize $| \Re u(\Bx)-\Re u_l(\Bx)|$},
 grid=major, legend entries={{J=8},{J=16},{J=32},{J=64}}, ymin=1e-4,ymax=1
  ]
    \settable{ndofs.txt}{FEcvg010_z3_d8.txt}
    \addplot[mark=square,mark size=1.4pt,Blue,solid,mark options={solid}] table[x=data,y expr = \thisrowno{0}]{\datatable};
      \settable{ndofs.txt}{FEcvg010_z3_d16.txt}
    \addplot[mark=*,mark size=1.4pt,Blue,dotted,mark options={solid}] table[x=data,y expr = \thisrowno{0}]{\datatable};
      \settable{ndofs.txt}{FEcvg010_z3_d32.txt}
    \addplot[mark=star,mark size=1.4pt,Blue,dashed,mark options={solid}] table[x=data,y expr = \thisrowno{0}]{\datatable};
      \settable{ndofs.txt}{FEcvg010_z3_d64.txt}
    \addplot[mark=diamond,mark size=1.4pt,Blue,dash pattern=on 0.5pt off 2pt on 2pt off 2pt,mark options={solid}] table[x=data,y expr = \thisrowno{0}]{\datatable};
    
        \addplot[black,domain=2*10^3:10^5.5,samples=97,style=dashed]{x^(-1)/(0.5*10^(-3))*(1)}
          coordinate [pos=0.85] (A)
  coordinate [pos=0.92] (B)
  ;
  \draw (A) -| (B) node[pos=0.9,anchor= west] {\hspace{-0.2cm} $\scriptscriptstyle -1$};
  \end{loglogaxis}
\end{tikzpicture}
\begin{tikzpicture}
  \begin{loglogaxis}[
 xlabel={\footnotesize $N_{dof}$},
 ylabel={\footnotesize $| \Re u(\Bx)-\Re u_l(\Bx)|$},
 grid=major, legend entries={{J=8},{J=16},{J=32},{J=64}}, ymin=1e-4,ymax=1
  ]
    \settable{ndofs.txt}{FEcvg010_z4_d8.txt}
    \addplot[mark=square,mark size=1.4pt,Green,solid,mark options={solid}] table[x=data,y expr = \thisrowno{0}]{\datatable};
      \settable{ndofs.txt}{FEcvg010_z4_d16.txt}
    \addplot[mark=*,mark size=1.4pt,Green,dotted,mark options={solid}] table[x=data,y expr = \thisrowno{0}]{\datatable};
      \settable{ndofs.txt}{FEcvg010_z4_d32.txt}
    \addplot[mark=star,mark size=1.4pt,Green,dashed,mark options={solid}] table[x=data,y expr = \thisrowno{0}]{\datatable};
      \settable{ndofs.txt}{FEcvg010_z4_d64.txt}
    \addplot[mark=diamond,mark size=1.4pt,Green,dash pattern=on 0.5pt off 2pt on 2pt off 2pt,mark options={solid}] table[x=data,y expr = \thisrowno{0}]{\datatable};
    
        \addplot[black,domain=2*10^3:10^5.5,samples=97,style=dashed]{x^(-1)/(0.5*10^(-3))*(1)}
          coordinate [pos=0.85] (A)
  coordinate [pos=0.92] (B)
  ;
  \draw (A) -| (B) node[pos=0.9,anchor= west] {\hspace{-0.2cm} $\scriptscriptstyle -1$};
  \end{loglogaxis}
\end{tikzpicture}
 \end{center}\caption{Finite element convergence for the point evaluation at $\Bx=(0.01,0)$, with, in \eqref{eq:radiusy}, $y_j=1$ for $j=1,\ldots,J$, using linear finite elements. Coefficient sequence $\beta_{2j-1}=\beta_{2j}=j^{-\frac{1}{p}}$ with $\frac{1}{p}=2$ (left), $\frac{1}{p}=3$ (center) and $\frac{1}{p}=4$ (right) and $j=1\ldots \tfrac{J}{2}$.}\label{fig:fecvgin}
 \end{figure}
 
  \begin{figure}
  \begin{center}
\begin{tikzpicture}
  \begin{loglogaxis}[
 xlabel={\footnotesize $N_{dof}$},
 ylabel={\footnotesize $| \Re u(\Bx)-\Re u_l(\Bx)|$},
 grid=major, legend entries={{J=8},{J=16},{J=32},{J=64}}, ymin=5e-6,ymax=1e-1
  ]
    \settable{ndofs.txt}{FEcvg001out_z2_d8.txt}
    \addplot[mark=square,mark size=1.4pt,Red,solid,mark options={solid}] table[x=data,y expr = \thisrowno{0}]{\datatable};
      \settable{ndofs.txt}{FEcvg001out_z2_d16.txt}
    \addplot[mark=*,mark size=1.4pt,Red,dotted,mark options={solid}] table[x=data,y expr = \thisrowno{0}]{\datatable};
      \settable{ndofs.txt}{FEcvg001out_z2_d32.txt}
    \addplot[mark=star,mark size=1.4pt,Red,dashed,mark options={solid}] table[x=data,y expr = \thisrowno{0}]{\datatable};
      \settable{ndofs.txt}{FEcvg001out_z2_d64.txt}
    \addplot[mark=diamond,mark size=1.4pt,Red,dash pattern=on 0.5pt off 2pt on 2pt off 2pt,mark options={solid}] table[x=data,y expr = \thisrowno{0}]{\datatable};
    
        \addplot[black,domain=2*10^3:10^5.5,samples=97,style=dashed]{x^(-1)/(0.5*10^(-3))*(1e-1)}
          coordinate [pos=0.85] (A)
  coordinate [pos=0.92] (B)
  ;
  \draw (A) -| (B) node[pos=0.9,anchor= west] {\hspace{-0.2cm} $\scriptscriptstyle -1$};
  \end{loglogaxis}
\end{tikzpicture}
\begin{tikzpicture}
  \begin{loglogaxis}[
 xlabel={\footnotesize $N_{dof}$},
 ylabel={\footnotesize $| \Re u(\Bx)-\Re u_l(\Bx)|$},
 grid=major, legend entries={{J=8},{J=16},{J=32},{J=64}}, ymin=5e-6,ymax=1e-1
  ]
    \settable{ndofs.txt}{FEcvg001out_z3_d8.txt}
    \addplot[mark=square,mark size=1.4pt,Blue,solid,mark options={solid}] table[x=data,y expr = \thisrowno{0}]{\datatable};
      \settable{ndofs.txt}{FEcvg001out_z3_d16.txt}
    \addplot[mark=*,mark size=1.4pt,Blue,dotted,mark options={solid}] table[x=data,y expr = \thisrowno{0}]{\datatable};
      \settable{ndofs.txt}{FEcvg001out_z3_d32.txt}
    \addplot[mark=star,mark size=1.4pt,Blue,dashed,mark options={solid}] table[x=data,y expr = \thisrowno{0}]{\datatable};
      \settable{ndofs.txt}{FEcvg001out_z3_d64.txt}
    \addplot[mark=diamond,mark size=1.4pt,Blue,dash pattern=on 0.5pt off 2pt on 2pt off 2pt,mark options={solid}] table[x=data,y expr = \thisrowno{0}]{\datatable};
    
        \addplot[black,domain=2*10^3:10^5.5,samples=97,style=dashed]{x^(-1)/(0.5*10^(-3))*(1e-1)}
          coordinate [pos=0.85] (A)
  coordinate [pos=0.92] (B)
  ;
  \draw (A) -| (B) node[pos=0.9,anchor= west] {\hspace{-0.2cm} $\scriptscriptstyle -1$};
  \end{loglogaxis}
\end{tikzpicture}
\begin{tikzpicture}
  \begin{loglogaxis}[
 xlabel={\footnotesize $N_{dof}$},
 ylabel={\footnotesize $| \Re u(\Bx)-\Re u_l(\Bx)|$},
 grid=major, legend entries={{J=8},{J=16},{J=32},{J=64}}, ymin=5e-6,ymax=1e-1
  ]
    \settable{ndofs.txt}{FEcvg001out_z4_d8.txt}
    \addplot[mark=square,mark size=1.4pt,Green,solid,mark options={solid}] table[x=data,y expr = \thisrowno{0}]{\datatable};
      \settable{ndofs.txt}{FEcvg001out_z4_d16.txt}
    \addplot[mark=*,mark size=1.4pt,Green,dotted,mark options={solid}] table[x=data,y expr = \thisrowno{0}]{\datatable};
      \settable{ndofs.txt}{FEcvg001out_z4_d32.txt}
    \addplot[mark=star,mark size=1.4pt,Green,dashed,mark options={solid}] table[x=data,y expr = \thisrowno{0}]{\datatable};
      \settable{ndofs.txt}{FEcvg001out_z4_d64.txt}
    \addplot[mark=diamond,mark size=1.4pt,Green,dash pattern=on 0.5pt off 2pt on 2pt off 2pt,mark options={solid}] table[x=data,y expr = \thisrowno{0}]{\datatable};
    
        \addplot[black,domain=2*10^3:10^5.5,samples=97,style=dashed]{x^(-1)/(0.5*10^(-3))*(1e-1)}
          coordinate [pos=0.85] (A)
  coordinate [pos=0.92] (B)
  ;
  \draw (A) -| (B) node[pos=0.9,anchor= west] {\hspace{-0.2cm} $\scriptscriptstyle -1$};
  \end{loglogaxis}
\end{tikzpicture}
 \end{center}\caption{Finite element convergence for the point evaluation at $\Bx=(0,0.01)$, with, in \eqref{eq:radiusy}, $y_j=-1$ for $j=1,\ldots,J$, using linear finite elements. Coefficient sequence $\beta_{2j-1}=\beta_{2j}=j^{-\frac{1}{p}}$ with $\frac{1}{p}=2$ (left), $\frac{1}{p}=3$ (center) and $\frac{1}{p}=4$ (right) and $j=1\ldots \tfrac{J}{2}$.}\label{fig:fecvgout}
 \end{figure}
 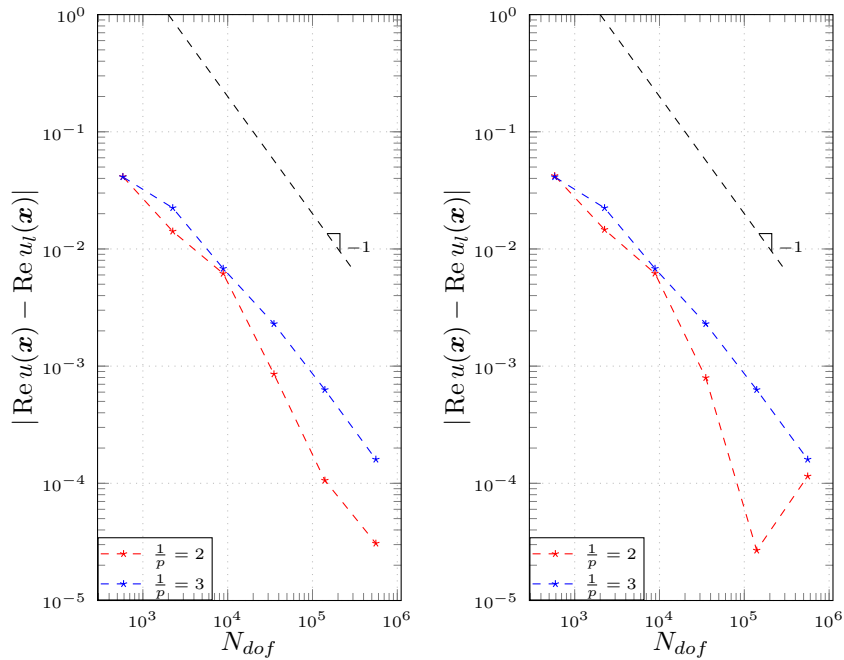
\begin{figure}
  \begin{center}
\begin{tikzpicture}
  \begin{loglogaxis}[
 xlabel={\footnotesize $N_{dof}$},
 ylabel={\footnotesize $| \Re u(\Bx)-\Re u_l(\Bx)|$},
 grid=major, legend entries={{$\tfrac{1}{p}=2$},{$\tfrac{1}{p}=3$}}, ymin=1e-5,ymax=1
  ]
    \settable{ndofs.txt}{FEcvgcross_z2_d32.txt}
    \addplot[mark=star,mark size=1.4pt,Red,dashed,mark options={solid}] table[x=data,y expr = \thisrowno{0}]{\datatable};
      \settable{ndofs.txt}{FEcvgcross_z3_d32.txt}
    \addplot[mark=star,mark size=1.4pt,Blue,dashed,mark options={solid}] table[x=data,y expr = \thisrowno{0}]{\datatable};    
        \addplot[black,domain=2*10^3:10^5.5,samples=97,style=dashed]{x^(-1)/(0.5*10^(-3))*(1)}
          coordinate [pos=0.85] (A)
  coordinate [pos=0.92] (B)
  ;
  \draw (A) -| (B) node[pos=0.9,anchor= west] {\hspace{-0.2cm} $\scriptscriptstyle -1$};
  \end{loglogaxis}
\end{tikzpicture}
\begin{tikzpicture}
  \begin{loglogaxis}[
 xlabel={\footnotesize $N_{dof}$},
 ylabel={\footnotesize $| \Re u(\Bx)-\Re u_l(\Bx)|$},
 grid=major, legend entries={{$\tfrac{1}{p}=2$},{$\tfrac{1}{p}=3$}}, ymin=1e-5,ymax=1
  ]
    \settable{ndofs.txt}{FEcvgcross_z2_d64.txt}
    \addplot[mark=star,mark size=1.4pt,Red,dashed,mark options={solid}] table[x=data,y expr = \thisrowno{0}]{\datatable};
      \settable{ndofs.txt}{FEcvgcross_z3_d64.txt}
    \addplot[mark=star,mark size=1.4pt,Blue,dashed,mark options={solid}] table[x=data,y expr = \thisrowno{0}]{\datatable};    
        \addplot[black,domain=2*10^3:10^5.5,samples=97,style=dashed]{x^(-1)/(0.5*10^(-3))*(1)}
          coordinate [pos=0.85] (A)
  coordinate [pos=0.92] (B)
  ;
  \draw (A) -| (B) node[pos=0.9,anchor= west] {\hspace{-0.2cm} $\scriptscriptstyle -1$};
  \end{loglogaxis}
\end{tikzpicture}
\end{center}\caption{Point evaluation at $\Bx=(0,0.01)$, with, in \eqref{eq:radiusy}, $y_j=-1$ for $j=1,\ldots,J$, using linear finite elements: convergence of solution for $J=32$ to solution for $J=16$ (left) and convergence of solution for $J=64$ to solution for $J=16$ (right). In each of the two cases, the coefficient sequences $\beta_{2j-1}=\beta_{2j}=j^{-\frac{1}{p}}$, $j=1,\ldots,\tfrac{J}{2}$, with $\frac{1}{p}=2$ and $\frac{1}{p}=3$ are considered.}\label{fig:fecvgcross}
\end{figure}

\section{MLMC for point evaluation: numerical experiments}\label{sect:numexp}
In this section we report the numerical results for the estimation of $\bbE_{\mu}\left[\Re \Bu\right]$, where $\Bu(\By)=\left\{u(\By;\Bx_i)\right\}_{i=0}^{N-1}$ is a set of $N$ point evaluations of the solution $u$ to \eqref{eq:modelpb}. We consider the cases of $N=1,2,4,8$ point evaluations, and define, for $N$ fixed, $\Bx_i = r_0(\cos\varphi_i,\sin\varphi_i)$ with $\varphi_i=2\pi\frac{i}{N}$, $i=0,\ldots,N-1$.

In the radius expansion, we compare the three decays of the coefficient sequence $\beta_{2j-1}=\beta_{2j}=0.1 r_0 j^{-\frac{1}{p}}$, $j=1,\ldots,\tfrac{J}{2}$, with $\tfrac{1}{p}=2,3,4$, and dimensions $J=16,32,64$ of the parameter space.

The physical and geometrical parameters and the domain mapping are as in subsection \ref{ssect:fecvgnumexp}. For the MLMC levels, we consider the first five meshes used in the finite element convergence studies of the previous section, that is unstructed, quasi-uniform meshes with $581$, $2250$, $8855$, $35133$, and $139961$ degrees of freedom, corresponding to $L=0,\ldots,4$, respectively. The finite element setting is as in the previous section (same PML parameters, first order elements, same finite element solver).

The MLMC estimators have been computed using the gMLQMC library\footnote{https://gitlab.math.ethz.ch/gantnerr/gMLQMC} \cite{Gantner}, with distribution of the samples among the levels determined by solving the optimization problem of minimizing the total error for a given amount of total computational cost. The work per sample has been estimated as $\operatorname{Work}_l=N_{dof,l}\cdot J$, $l=0,\ldots L$, where $N_{dof,l}$ is the number of finite element degrees of freedom at level $l$, and $J$ the dimension of the parameter space. The total work is calculated as $\operatorname{Work}_{tot}=\sum_{l=0}^L\operatorname{Work}_l$. To compute the total error, we have taken into account the logarithmic factor in the convergence rate as from Theorem \ref{thm:finalcvg}. The distribution of the samples among the levels used in all our experiments is reported in Table \ref{tab:samples}.

The error $\lVert \bbE_{\mu}[\Re\Bu]-E^L[\Re\Bu]\rVert_{L^2(\pspace_J,\bbR^N)}$ has been approximated by the average over $10$ realizations of it, considering, on $\bbR^N$, the Euclidean norm. As reference solution for $\bbE_{\mu}[\Re\Bu]$, we use the MLMC estimator $E^L[\Re\Bu]$ for $L=5$, where the mesh at the fifth level consists of $558705$ degrees of freedom.

Figure \ref{fig:mlmc1pt} shows the error versus work for one point evaluation, that is when $\Bu=u(\Bx)$ with $\Bx=(r_0,0)$. For this case, the error has been computed not only with respect to the MLMC estimator for $L=5$, but also with respect to the solution obtained by the Smolyak algorithm with $\mathfrak{R}$-Leja quadrature points before the estimated error saturates (cf. Fig. \ref{fig:smolyak}). The dashed line reports the theoretical rate of error versus work estimated when running the optimization algorithm to choose the number of samples at each level. In Figure \ref{fig:slmc}, we compare, for the case of a $16$-dimensional parameter space, the performance of the MLMC estimator with the single level estimator when samples chosen as $M=N_{dof}^2(\log N_{dof})^{-2}$ (for the single level estimator the error is computed over $15$ repetitions).

Figures \ref{fig:mlmc2pts}, \ref{fig:mlmc4pts} and \ref{fig:mlmc8pts} show the performance of MLMC when considering, respectively, $2$, $4$ and $8$ point evaluations.

\begin{table}
\footnotesize
\centering
 \begin{tabular}{|c|c|c|c|c|c|c|}
 \hline
 Maximal level & $M_0$ & $M_1$ & $M_2$ & $M_3$ & $M_4$ & $M_5$\\
 \hline
 $L=0$ & 1 & & & & & \\
 $L=1$ & 31 & 6 & & & & \\
 $L=2$ & 570 & 107 & 20 & & & \\
 $L=3$ & 9075 & 1697 & 305 & 54 & & \\
 $L=4$ & 134460 & 25144 & 4513 & 790 & 136 & \\
 $L=5$ & 1923719 & 359729 & 64557 & 11293 & 1943 & 331 \\
 \hline
\end{tabular} \caption{Number of samples $(M_l)_{l=0}^L$ for the numerical experiments of this section.}\label{tab:samples}
\end{table}

 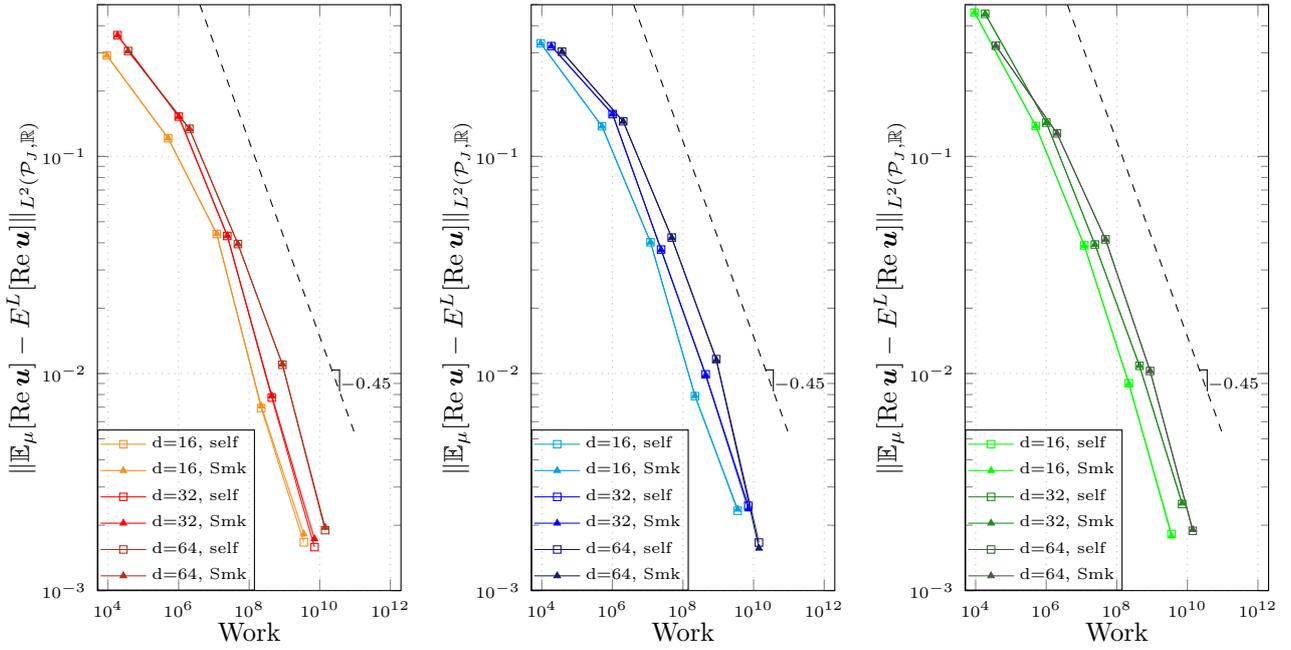
\begin{figure}
 \begin{center}
\noindent\begin{tikzpicture}
 \begin{loglogaxis}[
 xlabel={\footnotesize Work},
 ylabel={\footnotesize $\lVert \bbE_{\mu}[\Re\Bu]-E^L[\Re\Bu]\rVert_{L^2(\pspace_J,\bbR)}$},
 grid=major, legend entries={{d=16, self},{d=16, Smk},{d=32, self},{d=32, Smk},{d=64, self},{d=64, Smk}}, xmin=5e3, xmax=2e12, ymin=1e-3, ymax=5e-1
  ]
 
    \settable{work16.txt}{L2err_self_z2d16_1pt.txt}
    \addplot[mark=square,mark size=1.4pt,BurntOrange] table[x=data,y expr = \thisrowno{0}]{\datatable};
        \settable{work16.txt}{L2err_Smolyak_z2d16.txt}
    \addplot[mark=triangle*,mark size=1.4pt,BurntOrange] table[x=data,y expr = \thisrowno{0}]{\datatable};
    
        \settable{work32.txt}{L2err_self_z2d32_1pt.txt}
    \addplot[mark=square,mark size=1.4pt,Red] table[x=data,y expr = \thisrowno{0}]{\datatable};
        \settable{work32.txt}{L2err_Smolyak_z2d32.txt}
    \addplot[mark=triangle*,mark size=1.4pt,Red] table[x=data,y expr = \thisrowno{0}]{\datatable};
    
            \settable{work64.txt}{L2err_self_z2d64_1pt.txt}
    \addplot[mark=square,mark size=1.4pt,BrickRed] table[x=data,y expr = \thisrowno{0}]{\datatable};
        \settable{work64.txt}{L2err_Smolyak_z2d64.txt}
    \addplot[mark=triangle*,mark size=1.4pt,BrickRed] table[x=data,y expr = \thisrowno{0}]{\datatable};
    
                \addplot[black,domain=4*10^6:10^11,samples=97,style=dashed]{x^(-0.45)/((4^(-0.45))*10^(-(6*0.45)))*(5e-1)}
          coordinate [pos=0.85] (A)
  coordinate [pos=0.9] (B)
  ;
  \draw (A) -| (B) node[pos=0.85,anchor= west] {\hspace{-0.27cm} $\scriptscriptstyle -0.45$};
  \end{loglogaxis}
  \end{tikzpicture}
  \begin{tikzpicture}
   \begin{loglogaxis}[
 xlabel={\footnotesize Work},
 ylabel={\footnotesize $\lVert \bbE_{\mu}[\Re\Bu]-E^L[\Re\Bu]\rVert_{L^2(\pspace_J,\bbR)}$},
 grid=major, legend entries={{d=16, self},{d=16, Smk},{d=32, self},{d=32, Smk},{d=64, self},{d=64, Smk}}, xmin=5e3, xmax=2e12, ymin=1e-3, ymax=5e-1
  ]
    
    \settable{work16.txt}{L2err_self_z3d16_1pt.txt}
    \addplot[mark=square,mark size=1.4pt,Cerulean] table[x=data,y expr = \thisrowno{0}]{\datatable};
        \settable{work16.txt}{L2err_Smolyak_z3d16.txt}
    \addplot[mark=triangle*,mark size=1.4pt,Cerulean] table[x=data,y expr = \thisrowno{0}]{\datatable};
    
        \settable{work32.txt}{L2err_self_z3d32_1pt.txt}
    \addplot[mark=square,mark size=1.4pt,blue] table[x=data,y expr = \thisrowno{0}]{\datatable};
        \settable{work32.txt}{L2err_Smolyak_z3d32.txt}
    \addplot[mark=triangle*,mark size=1.4pt,blue] table[x=data,y expr = \thisrowno{0}]{\datatable};
    
            \settable{work64.txt}{L2err_self_z3d64_1pt.txt}
    \addplot[mark=square,mark size=1.4pt,MidnightBlue] table[x=data,y expr = \thisrowno{0}]{\datatable};
        \settable{work64.txt}{L2err_Smolyak_z3d64.txt}
    \addplot[mark=triangle*,mark size=1.4pt,MidnightBlue] table[x=data,y expr = \thisrowno{0}]{\datatable};
    
                \addplot[black,domain=4*10^6:10^11,samples=97,style=dashed]{x^(-0.45)/((4^(-0.45))*10^(-(6*0.45)))*(5e-1)}
          coordinate [pos=0.85] (A)
  coordinate [pos=0.9] (B)
  ;
  \draw (A) -| (B) node[pos=0.85,anchor= west] {\hspace{-0.27cm} $\scriptscriptstyle -0.45$};
  \end{loglogaxis}
  \end{tikzpicture}
    \begin{tikzpicture}
   \begin{loglogaxis}[
 xlabel={\footnotesize Work},
 ylabel={\footnotesize $\lVert \bbE_{\mu}[\Re\Bu]-E^L[\Re\Bu]\rVert_{L^2(\pspace_J,\bbR)}$},
 grid=major, legend entries={{d=16, self},{d=16, Smk},{d=32, self},{d=32, Smk},{d=64, self},{d=64, Smk}}, xmin=5e3, xmax=2e12, ymin=1e-3, ymax=5e-1
  ]
    
    \settable{work16.txt}{L2err_self_z4d16_1pt.txt}
    \addplot[mark=square,mark size=1.4pt,green] table[x=data,y expr = \thisrowno{0}]{\datatable};
        \settable{work16.txt}{L2err_Smolyak_z4d16.txt}
    \addplot[mark=triangle*,mark size=1.4pt,green] table[x=data,y expr = \thisrowno{0}]{\datatable};
    
        \settable{work32.txt}{L2err_self_z4d32_1pt.txt}
    \addplot[mark=square,mark size=1.4pt,ForestGreen] table[x=data,y expr = \thisrowno{0}]{\datatable};
        \settable{work32.txt}{L2err_Smolyak_z4d32.txt}
    \addplot[mark=triangle*,mark size=1.4pt,ForestGreen] table[x=data,y expr = \thisrowno{0}]{\datatable};
    
            \settable{work64.txt}{L2err_self_z4d64_1pt.txt}
    \addplot[mark=square,mark size=1.4pt,OliveGreen!50!Black] table[x=data,y expr = \thisrowno{0}]{\datatable};
        \settable{work64.txt}{L2err_Smolyak_z4d64.txt}
    \addplot[mark=triangle*,mark size=1.4pt,OliveGreen!50!Black] table[x=data,y expr = \thisrowno{0}]{\datatable};
    
                \addplot[black,domain=4*10^6:10^11,samples=97,style=dashed]{x^(-0.45)/((4^(-0.45))*10^(-(6*0.45)))*(5e-1)}
          coordinate [pos=0.85] (A)
  coordinate [pos=0.9] (B)
  ;
  \draw (A) -| (B) node[pos=0.85,anchor= west] {\hspace{-0.27cm} $\scriptscriptstyle -0.45$};
  \end{loglogaxis}
  \end{tikzpicture}
  \end{center}\caption{MLMC convergence for $5$ mesh levels ($L=4$) and one point evaluation ($\Bu=u(\Bx)$ with $\Bx=(r_0,0)$). Coefficient sequence $\beta_j=(j')^{-\frac{1}{p}}$ with $\frac{1}{p}=2$ (left), $\frac{1}{p}=3$ (center) and $\frac{1}{p}=4$ (right). Reference solution computed with MLMC on $6$ levels ($L=5$, label `self') and with the adaptive Smolyak algorithm (label `Smk'). The dashed line corresponds to the theoretical error versus work rate estimated when running the optimization algorithm to choose the number of samples at each level.}\label{fig:mlmc1pt}
  \end{figure}
 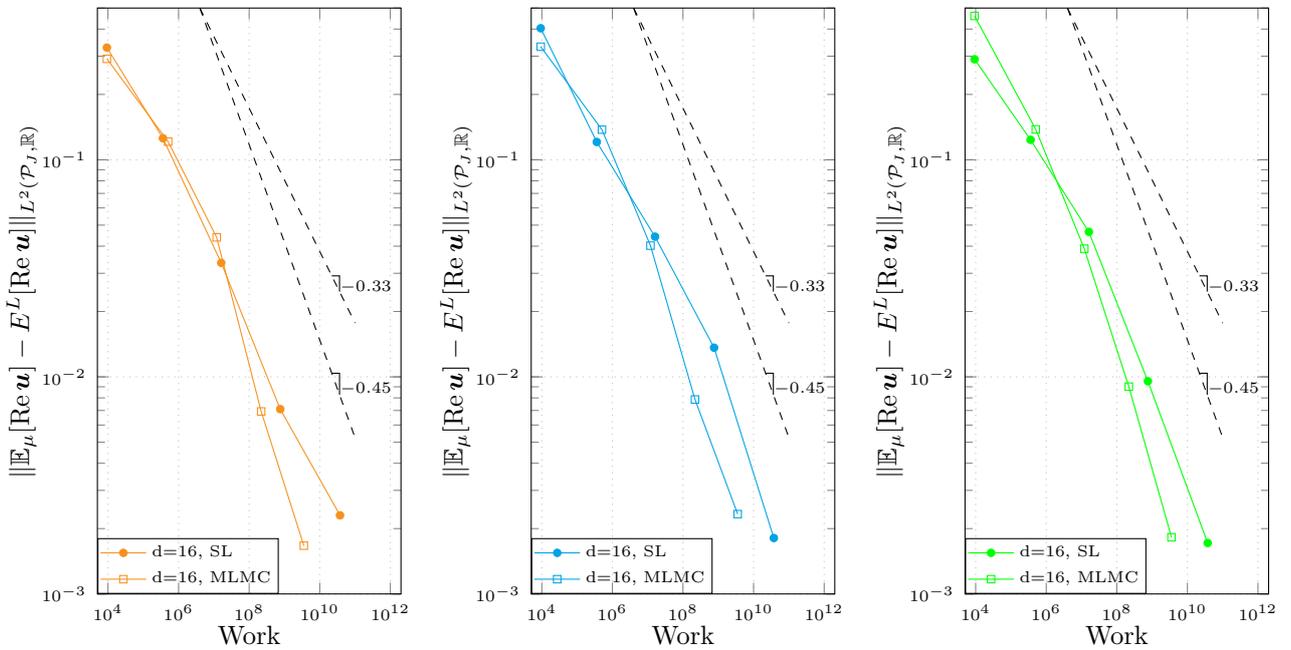
\begin{figure}
 \begin{center}
\noindent\begin{tikzpicture}
 \begin{loglogaxis}[
 xlabel={\footnotesize Work},
 ylabel={\footnotesize $\lVert \bbE_{\mu}[\Re\Bu]-E^L[\Re\Bu]\rVert_{L^2(\pspace_J,\bbR)}$},
 grid=major, legend entries={{d=16, SL},{d=16, MLMC}}, xmin=5e3, xmax=2e12, ymin=1e-3, ymax=5e-1
  ]
  \settable{workMC.txt}{L2errMC_z2.txt}
    \addplot[mark=*,mark size=1.4pt,BurntOrange] table[x=data,y expr = \thisrowno{0}]{\datatable};
  
    \settable{work16.txt}{L2err_self_z2d16_1pt.txt}
    \addplot[mark=square,mark size=1.4pt,BurntOrange] table[x=data,y expr = \thisrowno{0}]{\datatable};
    
                \addplot[black,domain=4*10^6:10^11,samples=97,style=dashed]{x^(-0.45)/((4^(-0.45))*10^(-(6*0.45)))*(5e-1)}
          coordinate [pos=0.85] (A)
  coordinate [pos=0.9] (B)
  ;
           \addplot[black,domain=4*10^6:10^11,samples=97,style=dashed]{x^(-0.33)/((4^(-0.33))*10^(-(6*0.33)))*(5e-1)}
          coordinate [pos=0.85] (C)
  coordinate [pos=0.9] (D)
  ;
  \draw (A) -| (B) node[pos=0.85,anchor= west] {\hspace{-0.27cm} $\scriptscriptstyle -0.45$};
  \draw (C) -| (D) node[pos=0.85,anchor= west] {\hspace{-0.27cm} $\scriptscriptstyle -0.33$};
  \end{loglogaxis}
  \end{tikzpicture}
  \begin{tikzpicture}
   \begin{loglogaxis}[
 xlabel={\footnotesize Work},
 ylabel={\footnotesize $\lVert \bbE_{\mu}[\Re\Bu]-E^L[\Re\Bu]\rVert_{L^2(\pspace_J,\bbR)}$},
 grid=major, legend entries={{d=16, SL},{d=16, MLMC}}, xmin=5e3, xmax=2e12, ymin=1e-3, ymax=5e-1
  ]
  
  \settable{workMC.txt}{L2errMC_z3.txt}
    \addplot[mark=*,mark size=1.4pt,Cerulean] table[x=data,y expr = \thisrowno{0}]{\datatable};
    
    \settable{work16.txt}{L2err_self_z3d16_1pt.txt}
    \addplot[mark=square,mark size=1.4pt,Cerulean] table[x=data,y expr = \thisrowno{0}]{\datatable};
    
    \addplot[black,domain=4*10^6:10^11,samples=97,style=dashed]{x^(-0.45)/((4^(-0.45))*10^(-(6*0.45)))*(5e-1)}
          coordinate [pos=0.85] (A)
  coordinate [pos=0.9] (B)
  ;
  \addplot[black,domain=4*10^6:10^11,samples=97,style=dashed]{x^(-0.33)/((4^(-0.33))*10^(-(6*0.33)))*(5e-1)}
          coordinate [pos=0.85] (C)
  coordinate [pos=0.9] (D)
  ;
  \draw (A) -| (B) node[pos=0.85,anchor= west] {\hspace{-0.27cm} $\scriptscriptstyle -0.45$};
  \draw (C) -| (D) node[pos=0.85,anchor= west] {\hspace{-0.27cm} $\scriptscriptstyle -0.33$};
  \end{loglogaxis}
  \end{tikzpicture}
    \begin{tikzpicture}
   \begin{loglogaxis}[
 xlabel={\footnotesize Work},
 ylabel={\footnotesize $\lVert \bbE_{\mu}[\Re\Bu]-E^L[\Re\Bu]\rVert_{L^2(\pspace_J,\bbR)}$},
 grid=major, legend entries={{d=16, SL}, {d=16, MLMC}}, xmin=5e3, xmax=2e12, ymin=1e-3, ymax=5e-1
  ]
  
  \settable{workMC.txt}{L2errMC_z4.txt}
    \addplot[mark=*,mark size=1.4pt,green] table[x=data,y expr = \thisrowno{0}]{\datatable};
    
    \settable{work16.txt}{L2err_self_z4d16_1pt.txt}
    \addplot[mark=square,mark size=1.4pt,green] table[x=data,y expr = \thisrowno{0}]{\datatable};
    
    \addplot[black,domain=4*10^6:10^11,samples=97,style=dashed]{x^(-0.45)/((4^(-0.45))*10^(-(6*0.45)))*(5e-1)}
          coordinate [pos=0.85] (A)
  coordinate [pos=0.9] (B)
  ;
  
   \addplot[black,domain=4*10^6:10^11,samples=97,style=dashed]{x^(-0.33)/((4^(-0.33))*10^(-(6*0.33)))*(5e-1)}
          coordinate [pos=0.85] (C)
  coordinate [pos=0.9] (D)
  ;
  \draw (A) -| (B) node[pos=0.85,anchor= west] {\hspace{-0.27cm} $\scriptscriptstyle -0.45$};
  \draw (C) -| (D) node[pos=0.85,anchor= west] {\hspace{-0.27cm} $\scriptscriptstyle -0.33$};
  \end{loglogaxis}
  \end{tikzpicture}
  \end{center}\caption{Comparison of MLMC and single level MC for $d=16$, $5$ mesh levels and one point evaluation ($\Bu=u(\Bx)$ with $\Bx=(r_0,0)$). Coefficient sequence $\beta_j=(j')^{-\frac{1}{p}}$ with $\frac{1}{p}=2$ (left), $\frac{1}{p}=3$ (center) and $\frac{1}{p}=4$ (right). Reference solution computed with MLMC on $6$ levels. The dashed line corresponds to the theoretical error versus work rates: $0.45$ for MLMC and $0.33$ for single level MC.}\label{fig:slmc}
  \end{figure}
 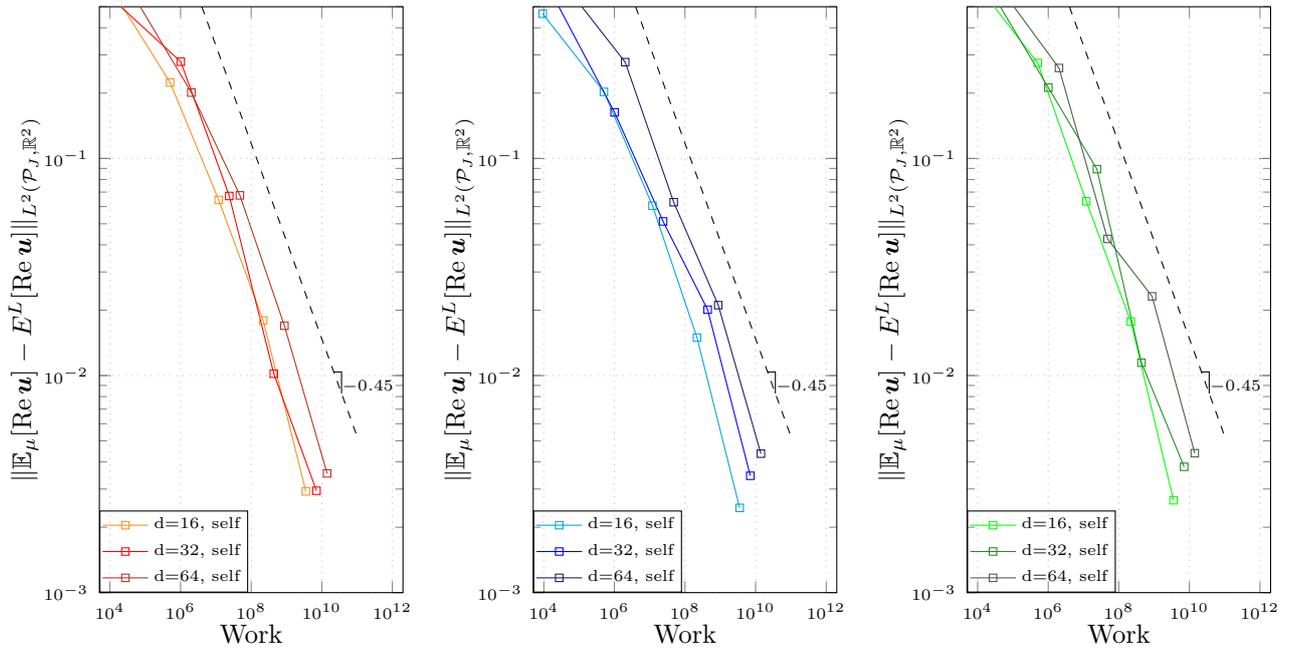
\begin{figure}
 \begin{center}
\noindent\begin{tikzpicture}
 \begin{loglogaxis}[
 xlabel={\footnotesize Work},
 ylabel={\footnotesize $\lVert \bbE_{\mu}[\Re\Bu]-E^L[\Re\Bu]\rVert_{L^2(\pspace_J,\bbR^2)}$},
 grid=major, legend entries={{d=16, self},{d=32, self},{d=64, self}}, xmin=5e3, xmax=2e12, ymin=1e-3, ymax=5e-1
  ]
    \settable{work16.txt}{L2err_self_z2d16_2pt.txt}
    \addplot[mark=square,mark size=1.4pt,BurntOrange] table[x=data,y expr = \thisrowno{0}]{\datatable};
    
        \settable{work32.txt}{L2err_self_z2d32_2pt.txt}
    \addplot[mark=square,mark size=1.4pt,Red] table[x=data,y expr = \thisrowno{0}]{\datatable};
    
            \settable{work64.txt}{L2err_self_z2d64_2pt.txt}
    \addplot[mark=square,mark size=1.4pt,BrickRed] table[x=data,y expr = \thisrowno{0}]{\datatable};
    
                \addplot[black,domain=4*10^6:10^11,samples=97,style=dashed]{x^(-0.45)/((4^(-0.45))*10^(-(6*0.45)))*(5e-1)}
          coordinate [pos=0.85] (A)
  coordinate [pos=0.9] (B)
  ;
  \draw (A) -| (B) node[pos=0.85,anchor= west] {\hspace{-0.27cm} $\scriptscriptstyle -0.45$};
  \end{loglogaxis}
  \end{tikzpicture}
  \begin{tikzpicture}
   \begin{loglogaxis}[
 xlabel={\footnotesize Work},
 ylabel={\footnotesize $\lVert \bbE_{\mu}[\Re\Bu]-E^L[\Re\Bu]\rVert_{L^2(\pspace_J,\bbR^2)}$},
 grid=major, legend entries={{d=16, self},{d=32, self},{d=64, self}}, xmin=5e3, xmax=2e12, ymin=1e-3, ymax=5e-1
  ]
    \settable{work16.txt}{L2err_self_z3d16_2pt.txt}
    \addplot[mark=square,mark size=1.4pt,Cerulean] table[x=data,y expr = \thisrowno{0}]{\datatable};
    
        \settable{work32.txt}{L2err_self_z3d32_2pt.txt}
    \addplot[mark=square,mark size=1.4pt,blue] table[x=data,y expr = \thisrowno{0}]{\datatable};
    
            \settable{work64.txt}{L2err_self_z3d64_2pt.txt}
    \addplot[mark=square,mark size=1.4pt,MidnightBlue] table[x=data,y expr = \thisrowno{0}]{\datatable};
    
                \addplot[black,domain=4*10^6:10^11,samples=97,style=dashed]{x^(-0.45)/((4^(-0.45))*10^(-(6*0.45)))*(5e-1)}
          coordinate [pos=0.85] (A)
  coordinate [pos=0.9] (B)
  ;
  \draw (A) -| (B) node[pos=0.85,anchor= west] {\hspace{-0.27cm} $\scriptscriptstyle -0.45$};
  \end{loglogaxis}
  \end{tikzpicture}
    \begin{tikzpicture}
   \begin{loglogaxis}[
 xlabel={\footnotesize Work},
 ylabel={\footnotesize $\lVert \bbE_{\mu}[\Re\Bu]-E^L[\Re\Bu]\rVert_{L^2(\pspace_J,\bbR^2)}$},
 grid=major, legend entries={{d=16, self},{d=32, self},{d=64, self}}, xmin=5e3, xmax=2e12, ymin=1e-3, ymax=5e-1
  ]
    \settable{work16.txt}{L2err_self_z4d16_2pt.txt}
    \addplot[mark=square,mark size=1.4pt,green] table[x=data,y expr = \thisrowno{0}]{\datatable};
    
        \settable{work32.txt}{L2err_self_z4d32_2pt.txt}
    \addplot[mark=square,mark size=1.4pt,ForestGreen] table[x=data,y expr = \thisrowno{0}]{\datatable};
    
            \settable{work64.txt}{L2err_self_z4d64_2pt.txt}
    \addplot[mark=square,mark size=1.4pt,OliveGreen!50!Black] table[x=data,y expr = \thisrowno{0}]{\datatable};
    
                \addplot[black,domain=4*10^6:10^11,samples=97,style=dashed]{x^(-0.45)/((4^(-0.45))*10^(-(6*0.45)))*(5e-1)}
          coordinate [pos=0.85] (A)
  coordinate [pos=0.9] (B)
  ;
  \draw (A) -| (B) node[pos=0.85,anchor= west] {\hspace{-0.27cm} $\scriptscriptstyle -0.45$};
  \end{loglogaxis}
  \end{tikzpicture}
  \end{center}\caption{MLMC convergence for $5$ mesh levels ($L=4$) and two point evaluations ($\Bu=\left\{u(\Bx_i)\right\}_{i=1}^2$, with $\Bx_1=(r_0,0)$ and $\Bx_2=(-r_0,0)$). Coefficient sequence $\beta_j=(j')^{-\frac{1}{p}}$ with $\frac{1}{p}=2$ (left), $\frac{1}{p}=3$ (center) and $\frac{1}{p}=4$ (right). Reference solution computed with MLMC on $6$ levels ($L=5$). The dashed line corresponds to the theoretical error versus work rate estimated when running the optimization algorithm to choose the number of samples at each level.}\label{fig:mlmc2pts}
  \end{figure}
 \begin{figure}
 \begin{center}
\noindent\begin{tikzpicture}
 \begin{loglogaxis}[
 xlabel={\footnotesize Work},
 ylabel={\footnotesize $\lVert \bbE_{\mu}[\Re\Bu]-E^L[\Re\Bu]\rVert_{L^2(\pspace_J,\bbR^4)}$},
 grid=major, legend entries={{d=16, self},{d=32, self},{d=64, self}}, xmin=5e3, xmax=2e12, ymin=1e-3, ymax=5e-1
  ]
    \settable{work16.txt}{L2err_self_z2d16_4pt.txt}
    \addplot[mark=square,mark size=1.4pt,BurntOrange] table[x=data,y expr = \thisrowno{0}]{\datatable};
    
        \settable{work32.txt}{L2err_self_z2d32_4pt.txt}
    \addplot[mark=square,mark size=1.4pt,Red] table[x=data,y expr = \thisrowno{0}]{\datatable};
    
            \settable{work64.txt}{L2err_self_z2d64_4pt.txt}
    \addplot[mark=square,mark size=1.4pt,BrickRed] table[x=data,y expr = \thisrowno{0}]{\datatable};
    
                \addplot[black,domain=4*10^6:10^11,samples=97,style=dashed]{x^(-0.45)/((4^(-0.45))*10^(-(6*0.45)))*(5e-1)}
          coordinate [pos=0.85] (A)
  coordinate [pos=0.9] (B)
  ;
  \draw (A) -| (B) node[pos=0.85,anchor= west] {\hspace{-0.27cm} $\scriptscriptstyle -0.45$};
  \end{loglogaxis}
  \end{tikzpicture}
  \begin{tikzpicture}
   \begin{loglogaxis}[
 xlabel={\footnotesize Work},
 ylabel={\footnotesize $\lVert \bbE_{\mu}[\Re\Bu]-E^L[\Re\Bu]\rVert_{L^2(\pspace_J,\bbR^4)}$},
 grid=major, legend entries={{d=16, self},{d=32, self},{d=64, self}}, xmin=5e3, xmax=2e12, ymin=1e-3, ymax=5e-1
  ]
    \settable{work16.txt}{L2err_self_z3d16_4pt.txt}
    \addplot[mark=square,mark size=1.4pt,Cerulean] table[x=data,y expr = \thisrowno{0}]{\datatable};
    
        \settable{work32.txt}{L2err_self_z3d32_4pt.txt}
    \addplot[mark=square,mark size=1.4pt,blue] table[x=data,y expr = \thisrowno{0}]{\datatable};
    
            \settable{work64.txt}{L2err_self_z3d64_4pt.txt}
    \addplot[mark=square,mark size=1.4pt,MidnightBlue] table[x=data,y expr = \thisrowno{0}]{\datatable};
    
                \addplot[black,domain=4*10^6:10^11,samples=97,style=dashed]{x^(-0.45)/((4^(-0.45))*10^(-(6*0.45)))*(5e-1)}
          coordinate [pos=0.85] (A)
  coordinate [pos=0.9] (B)
  ;
  \draw (A) -| (B) node[pos=0.85,anchor= west] {\hspace{-0.27cm} $\scriptscriptstyle -0.45$};
  \end{loglogaxis}
  \end{tikzpicture}
    \begin{tikzpicture}
   \begin{loglogaxis}[
 xlabel={\footnotesize Work},
 ylabel={\footnotesize $\lVert \bbE_{\mu}[\Re\Bu]-E^L[\Re\Bu]\rVert_{L^2(\pspace_J,\bbR^4)}$},
 grid=major, legend entries={{d=16, self},{d=32, self},{d=64, self}}, xmin=5e3, xmax=2e12, ymin=1e-3, ymax=5e-1
  ]
    \settable{work16.txt}{L2err_self_z4d16_4pt.txt}
    \addplot[mark=square,mark size=1.4pt,green] table[x=data,y expr = \thisrowno{0}]{\datatable};
    
        \settable{work32.txt}{L2err_self_z4d32_4pt.txt}
    \addplot[mark=square,mark size=1.4pt,ForestGreen] table[x=data,y expr = \thisrowno{0}]{\datatable};
    
            \settable{work64.txt}{L2err_self_z4d64_4pt.txt}
    \addplot[mark=square,mark size=1.4pt,OliveGreen!50!Black] table[x=data,y expr = \thisrowno{0}]{\datatable};
    
                \addplot[black,domain=4*10^6:10^11,samples=97,style=dashed]{x^(-0.45)/((4^(-0.45))*10^(-(6*0.45)))*(5e-1)}
          coordinate [pos=0.85] (A)
  coordinate [pos=0.9] (B)
  ;
  \draw (A) -| (B) node[pos=0.85,anchor= west] {\hspace{-0.27cm} $\scriptscriptstyle -0.45$};
  \end{loglogaxis}
  \end{tikzpicture}
  \end{center}\caption{MLMC convergence for $5$ mesh levels ($L=4$) and four point evaluations ($\Bu=\left\{u(\Bx_i)\right\}_{i=1}^4$, with $\Bx_1=(r_0,0)$, $\Bx_2=(0,r_0)$, $\Bx_3=(-r_0,0)$, $\Bx_4=(0,-r_0)$). Coefficient sequence $\beta_j=(j')^{-\frac{1}{p}}$ with $\frac{1}{p}=2$ (left), $\frac{1}{p}=3$ (center) and $\frac{1}{p}=4$ (right). Reference solution computed with MLMC on $6$ levels ($L=5$). The dashed line corresponds to the theoretical error versus work rate estimated when running the optimization algorithm to choose the number of samples at each level.}\label{fig:mlmc4pts}
  \end{figure}
 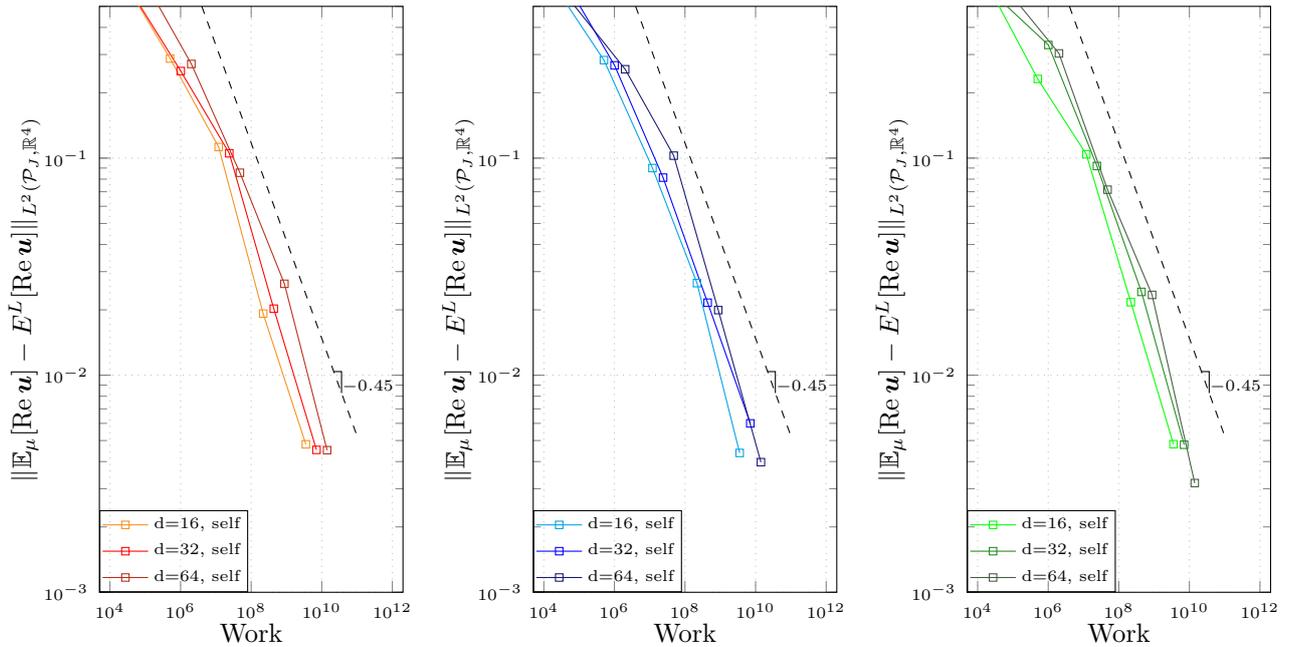
\begin{figure}
 \begin{center}
\noindent\begin{tikzpicture}
 \begin{loglogaxis}[
 xlabel={\footnotesize Work},
 ylabel={\footnotesize $\lVert \bbE_{\mu}[\Re\Bu]-E^L[\Re\Bu]\rVert_{L^2(\pspace_J,\bbR^4)}$},
 grid=major, legend entries={{d=16, self},{d=32, self},{d=64, self}}, xmin=5e3, xmax=2e12, ymin=1e-3, ymax=5e-1
  ]
    \settable{work16.txt}{L2err_self_z2d16_8pt.txt}
    \addplot[mark=square,mark size=1.4pt,BurntOrange] table[x=data,y expr = \thisrowno{0}]{\datatable};
    
        \settable{work32.txt}{L2err_self_z2d32_8pt.txt}
    \addplot[mark=square,mark size=1.4pt,Red] table[x=data,y expr = \thisrowno{0}]{\datatable};
    
            \settable{work64.txt}{L2err_self_z2d64_8pt.txt}
    \addplot[mark=square,mark size=1.4pt,BrickRed] table[x=data,y expr = \thisrowno{0}]{\datatable};
    
                \addplot[black,domain=4*10^6:10^11,samples=97,style=dashed]{x^(-0.45)/((4^(-0.45))*10^(-(6*0.45)))*(5e-1)}
          coordinate [pos=0.85] (A)
  coordinate [pos=0.9] (B)
  ;
  \draw (A) -| (B) node[pos=0.85,anchor= west] {\hspace{-0.27cm} $\scriptscriptstyle -0.45$};
  \end{loglogaxis}
  \end{tikzpicture}
  \begin{tikzpicture}
   \begin{loglogaxis}[
 xlabel={\footnotesize Work},
 ylabel={\footnotesize $\lVert \bbE_{\mu}[\Re\Bu]-E^L[\Re\Bu]\rVert_{L^2(\pspace_J,\bbR^4)}$},
 grid=major, legend entries={{d=16, self},{d=32, self},{d=64, self}}, xmin=5e3, xmax=2e12, ymin=1e-3, ymax=5e-1
  ]
    \settable{work16.txt}{L2err_self_z3d16_8pt.txt}
    \addplot[mark=square,mark size=1.4pt,Cerulean] table[x=data,y expr = \thisrowno{0}]{\datatable};
    
        \settable{work32.txt}{L2err_self_z3d32_8pt.txt}
    \addplot[mark=square,mark size=1.4pt,blue] table[x=data,y expr = \thisrowno{0}]{\datatable};
    
            \settable{work64.txt}{L2err_self_z3d64_8pt.txt}
    \addplot[mark=square,mark size=1.4pt,MidnightBlue] table[x=data,y expr = \thisrowno{0}]{\datatable};
    
                \addplot[black,domain=4*10^6:10^11,samples=97,style=dashed]{x^(-0.45)/((4^(-0.45))*10^(-(6*0.45)))*(5e-1)}
          coordinate [pos=0.85] (A)
  coordinate [pos=0.9] (B)
  ;
  \draw (A) -| (B) node[pos=0.85,anchor= west] {\hspace{-0.27cm} $\scriptscriptstyle -0.45$};
  \end{loglogaxis}
  \end{tikzpicture}
    \begin{tikzpicture}
   \begin{loglogaxis}[
 xlabel={\footnotesize Work},
 ylabel={\footnotesize $\lVert \bbE_{\mu}[\Re\Bu]-E^L[\Re\Bu]\rVert_{L^2(\pspace_J,\bbR^4)}$},
 grid=major, legend entries={{d=16, self},{d=32, self},{d=64, self}}, xmin=5e3, xmax=2e12, ymin=1e-3, ymax=5e-1
  ]
    \settable{work16.txt}{L2err_self_z4d16_8pt.txt}
    \addplot[mark=square,mark size=1.4pt,green] table[x=data,y expr = \thisrowno{0}]{\datatable};
    
        \settable{work32.txt}{L2err_self_z4d32_8pt.txt}
    \addplot[mark=square,mark size=1.4pt,ForestGreen] table[x=data,y expr = \thisrowno{0}]{\datatable};
    
            \settable{work64.txt}{L2err_self_z4d64_8pt.txt}
    \addplot[mark=square,mark size=1.4pt,OliveGreen!50!Black] table[x=data,y expr = \thisrowno{0}]{\datatable};
    
                 \addplot[black,domain=4*10^6:10^11,samples=97,style=dashed]{x^(-0.45)/((4^(-0.45))*10^(-(6*0.45)))*(5e-1)}
          coordinate [pos=0.85] (A)
  coordinate [pos=0.9] (B)
  ;
  \draw (A) -| (B) node[pos=0.85,anchor= west] {\hspace{-0.27cm} $\scriptscriptstyle -0.45$};
  \end{loglogaxis}
  \end{tikzpicture}
  \end{center}\caption{MLMC convergence for $5$ mesh levels ($L=4$) and eight point evaluations ($\Bu=\left\{u(\Bx_i)\right\}_{i=1}^8$, with $\Bx_i=(\cos(\varphi_i),\sin(\varphi_i))$, and $\varphi_i=\frac{2\pi}{8}(i-1)$, $i=1,\ldots,8$). Coefficient sequence $\beta_j=(j')^{-\frac{1}{p}}$ with $\frac{1}{p}=2$ (left), $\frac{1}{p}=3$ (center) and $\frac{1}{p}=4$ (right). Reference solution computed with MLMC on $6$ levels ($L=5$). The dashed line corresponds to the theoretical error versus work rate estimated when running the optimization algorithm to choose the number of samples at each level.}\label{fig:mlmc8pts}
  \end{figure}

From Figures \ref{fig:mlmc1pt}, \ref{fig:slmc}, \ref{fig:mlmc2pts}, \ref{fig:mlmc4pts} and \ref{fig:mlmc8pts} we can draw the following conclusions:
\begin{itemize}
 \item the convergence rate of error versus work predicted when running the optimization algorithm (dashed line with slope $-0.45$) is achieved, in all experiments; for low error thresholds, significant cost savings can be observed when comparing MLMC with single level MC;
 \item the right shift of the error curves as the dimension $J$ of the parameter space increases is only due to the fact that we compute the work of a single solve as  $\operatorname{Work}_l=N_{dof,l}\cdot J$, $l=0,\ldots L$; this increase of the computational cost with respect to $J$ is inevitable unless an algorithm to adapt $J$ to the discretization level is considered, as also suggested in the conclusions in \cite{CGST};
 \item the rate of convergence of MLMC is dimension robust; as already observed for the finite element convergence, and thus not surprisingly here, the results are even better than predicted by theory, in the sense that dimension robustness occurs also for $\frac{1}{p}=2$ ($\frac{1}{p}=3$ is a the limit case, see Theorem \ref{thm:finalcvg});
 \item requiring only square integrability of the Q.o.I., MLMC is robust with respect to the number of singularities in the parameter space, and provides full convergence rate for $N=2,4,8$ point evaluations; the plots show that the error increases as the number of point evaluations considered increases, but this is because the dimension of $\Bu$ does.
\end{itemize}

\section{Conclusions and extensions}
We have shown that the MLMC method is effective in computing statistics (in particular the mean) of a Q.o.I. whose dependence on the parameter is non-smooth, with discontinuities which are not easy to track. As model we have considered the computation of point values of the solution to a Helmholtz transmission problem with stochastic interface. For this case, we have analyzed the convergence rate of the finite element discretization and shown how it can be used to compute the optimal distribution of samples in the MLMC algorithm. Particular attention has been dedicated to the robustness of the convergence rates with respect to the dimension of the parameter space. The numerical experiments confirm the theoretical results, and show that maybe the result on the $J$-independence of the finite element convergence rate for the point evaluation can be improved.

Concerning the application to the point evaluation, we highlight that the results are not confined to our model problem. The affine parametrization of the stochastic interface does not need to be in polar coordinates and with respect to the Fourier basis: a more general expansion for a stochastic interface is possible, as long as $C^{1,\beta}$-smoothness is guaranteed. Moreover, the analysis on the space regularity of the solution to the PDE carries over to any other elliptic PDE associated to a coercive bilinear form, with a parameter-independent lower bound on the coercivity constant. Finally, the methodology presented in this paper still holds for three-dimensional problems.

The results of this work open the way to further investigations. As observed in Proposition \ref{prop:conty}, if the highest order PDE coefficient is continuous across the interface (i.e. $\alpha_2=1$ in our model problem), then the solution has $C^1$-dependence on the parameter, and it would be interesting to analyze the performance of quasi-Monte Carlo quadrature rules in this case. Another interesting aspect is the possibility to adapt the truncation dimension $J$ to the mesh levels in the MLMC algorithm. As observed in subsection \ref{ssect:fecvgnumexp},  indeed, it is likely that on coarser levels the high-frequency perturbations of the domain cannot be captured by the discretization, and this could be exploited to save computational effort and have a $J$-dependence of the cost of one solve which is milder than $\operatorname{Work}_l=N_{dof,l}\cdot J$ (for $l=0,\ldots,L$). This observation can also be found in \cite{CGST}; a first step in this direction has been done \cite{DKST}, where the truncation levels have been chosen empirically.

\renewcommand{\abstractname}{Acknowledgements}
\begin{abstract}
The author would like to thank Prof.~Ralf Hiptmair and Prof.~Christoph Schwab for their suggestions during the development of this work, and for their feedback on drafts of this paper. She would also like to thank Dr.~Vanja Nikoli\'c for pointing out the reference \cite{Joc}, and Robert Gantner for his support with the MLMC library. This work has been funded by ETH under CHIRP Grant CH1-02 11-1 and partially by the Technical University of M\"unich.
\end{abstract}

\titleformat{\section}{\large\bfseries}{\appendixname~\thesection .}{0.5em}{}
\begin{appendices}
\section{$H^{1+\beta'}$-regularity of the solution}
In the proof to Proposition \ref{prop:conty} we have used the fact that, if the coefficients $\hat{\alpha}(\By;\cdot)$ and $\hat{\kappa}^2(\By;\cdot)$ are piecewise H\"older continuous, then the scattered wave $\hat{u}_s(\By;\cdot)$ is in $H^{1+\beta'}(D_{R_{out}})$ for some $\beta'>0$, and from this continuity of $\hat{u}(\By;\cdot)$ follows. Here we present the $H^{1+\beta'}$-regularity result on the scattered wave, and we do it slightly modifying the proofs contained in \cite{Joc} and \cite[Sect. 3]{BGL}. 

The scattered wave fulfills the variational formulation
\begin{equation}
\begin{split}\label{eq:varformus}
 &\text{Find }\uref_s(\By)\in H^1(D_{R_{out}}): a(\By;\uref_s,\vref)=\langle f_s(\By),\vref\rangle_{\langle (H^1(D_{R_{out}}))',H^1(D_{R_{out}})\rangle}, \\
 &\text{for every }\vref\in H^1(D_{R_{out}}), J\in\bbN, \By\in\pspace_J,
 \end{split}
 \end{equation}
 where $(H^1(D_{R_{out}}))'$ denotes the dual space of $(H^1(D_{R_{out}}))$,
 \begin{equation}
 \begin{split}
 &a(\By;\uref_s,v):=\int_{D_{R_{out}}}\hat{\alpha}(\By;\hat{\Bx})\hat{\nabla} \uref_s(\By)\cdot\hat{\nabla}\vref- \hat{\kappa}^2(\By;\hat{\Bx})\uref_s(\By) \vref\dd\xref-\int_{\partial D_{R_{out}}}\operatorname{DtN}(\uref_s(\By))\vref \dd S, \hspace{-0.5cm}\\
 & f_s(\By):= \hat{\nabla}\cdot\left(\hat{\alpha}(\By;\hat{\Bx})\hat{\nabla} \uref_i(\By)\right)+ \hat{\kappa}^2(\By;\hat{\Bx})\uref_i(\By),
\end{split}\label{eq:alus}
\end{equation}
and $\uref_i(\hat{\Bx}):=u_i(\Phi(\By;\hat{\Bx}))$, for all $\hat{\Bx}\in D_{R_{out}}$. Our goal is to prove the following:
\begin{theorem}\label{thm:contreg}
 If $\hat{\alpha}(\By;\cdot)\in C_{pw}^{\beta}(\overline{D_{R_{out}}})$, $\hexp\in (0,\tfrac{1}{2})$, and $\hat{\kappa}^2(\By;\cdot)\in C_{pw}^{0}(\overline{D_{R_{out}}})$, with $J$- and $\By$-independent bounds on the norms, if Assumption \ref{Ak1k2} holds and $\frac{\sigma_{min}^4}{\sigma_{max}^4}\min\left\{\frac{1}{\alpha_2},\alpha_2\right\}\geq 1-\frac{\gamma_p}{\mathcal{C}}$ (with $\gamma_p$, $\mathcal{C}$ as in Proposition \ref{prop:conty}), then there exists $0<\beta'<\beta$ such that
 \begin{equation}
  \lVert\uref_s\rVert_{H^{1+\beta'}(D_{R_{out}})}\leq C \lVert u_i \rVert_{C^{1}(\overline{D_{R_{out}}})},
 \end{equation}
 with a constant $C$ independent of $J\in\bbN$ and $\By\in\pspace_J$.
\end{theorem}
We first note that showing the above result for \eqref{eq:varformus} is equivalent to showing the result for the variational formulation
\begin{equation}
\begin{split}
& \text{Find }\uref_s(\By)\in H^1(D_{R_{out}}): a_p(\By;\uref_s,v)=\langle \tilde{f}_s(\By),\vref\rangle_{\langle (H^1(D_{R_{out}}))',H^1(D_{R_{out}})\rangle}, \\
 &\text{for every }\vref\in H^1(D_{R_{out}}), \By\in\pspace_J, J\in\bbN,
 \end{split}
 \end{equation}
 with the low order term of the bilinear form moved to the right-hand side:
 \begin{equation}
 \begin{split}
 &a_p(\By;\uref_s,v):=\int_{D_{R_{out}}}\hat{\alpha}(\By;\hat{\Bx})\hat{\nabla} \uref_s(\By)\cdot\hat{\nabla}\vref\dd\xref-\int_{\partial D_{R_{out}}}\operatorname{DtN}(\uref_s(\By))\vref \dd S, \\
 &\tilde{f}_s(\By):= \hat{\nabla}\cdot\left(\hat{\alpha}(\By;\hat{\Bx})\hat{\nabla} \uref_i(\By)\right)+ \hat{\kappa}^2(\By;\hat{\Bx})\uref(\By).
\end{split}
\end{equation}

From now on, we use bold symbols for Sobolev spaces of vector-valued functions; for instance, $\VL^2(D_{R_{out}}):=\left(L^2(D_{R_{out}})\right)^2$. Using the notation of \cite{Joc} and \cite{BGL}, we define the operators $\mathcal{J}: (H^1(D_{R_{out}}))'\rightarrow H^1(D_{R_{out}})$ and $\mathcal{S}: \VL^2(D_{R_{out}})\rightarrow (H^1(D_{R_{out}}))'$ by:
\begin{align}
  \langle\hat{\nabla} (\mathcal{J}f),\hat{\nabla}\vref\rangle_{\langle \VL^2(D_{R_{out}}),\VL^2(D_{R_{out}})\rangle}&-\langle\operatorname{DtN}(\mathcal{J}f),\vref\rangle_{\langle H^{-\frac{1}{2}}(\partial D_{R_{out}}),H^{\frac{1}{2}}(\partial D_{R_{out}})\rangle}\nonumber\\
  & =\langle f,\vref\rangle_{\langle (H^1(D_{R_{out}}))',H^1(D_{R_{out}})\rangle}, \label{eq:J}\\
 \langle \mathcal{S}\VF,\vref\rangle_{\langle (H^1(D_{R_{out}}))',H^1(D_{R_{out}})\rangle} &= \langle \VF,\hat{\nabla}\vref\rangle_{\langle \VL^2(D_{R_{out}}),\VL^2(D_{R_{out}})\rangle} \nonumber\\
 &- \langle \VF\cdot \Bn_{out},\vref\rangle_{\langle H^{-\frac{1}{2}}(\partial D_{R_{out}}),H^{\frac{1}{2}}(\partial D_{R_{out}})\rangle},\label{eq:S}
\end{align}
for all $\vref\in H^1(D_{R_{out}})$, $f\in (H^1(D_{R_{out}}))'$ and $\VF\in \VL^2(D_{R_{out}})$.

On the lines of Lemmas 3.1 and 3.2 in \cite{BGL}, we prove the mapping properties of the operators $\mathcal{J}$ and $\mathcal{S}$.

\begin{lemma}[Analogous to Lemma 3.1 in \cite{BGL}]\label{lem:S} For all $s\in[0,1]$ and all $\VF\in \VH^s(D_{R_{out}})$, $\mathcal{S}\VF\in H^{s-1}(D_{R_{out}})$ and
 \begin{equation}
  \lVert \mathcal{S}\VF\rVert_{H^{s-1}(D_{R_{out}})}\leq \mathcal{C}^{1-s}\lVert\VF\rVert_{\VH^s(D_{R_{out}})},
 \end{equation}
with $\mathcal{C}$ as in Proposition \ref{prop:conty}. 
\end{lemma}
\begin{proof}
 For $s=0$:
 \begin{equation*}
  \langle \mathcal{S}\VF,\vref\rangle_{\langle (H^1(D_{R_{out}}))',H^1(D_{R_{out}})\rangle}\leq \mathcal{C}\lVert\VF\rVert_{\VL^2(D_{R_{out}})}\lVert\vref\rVert_{H^1(D_{R_{out}})},
 \end{equation*}
for every $\vref\in H^1(D_{R_{out}})$. For $s=1$:
 \begin{equation*}
  \langle \mathcal{S}\VF,\vref\rangle_{\langle L^2(D_{R_{out}}),L^2(D_{R_{out}})\rangle}=-\langle \hat{\nabla}\cdot\VF,\vref\rangle_{\langle L^2(D_{R_{out}}),L^2(D_{R_{out}})\rangle}\leq \lVert\VF\rVert_{\VH^1(D_{R_{out}})}\lVert\vref\rVert_{L^2(D_{R_{out}})},
  \end{equation*}
  for every $\vref\in L^2(D_{R_{out}})$. The claim follows then from the Riesz-Thorin Theorem.
\end{proof}
\begin{lemma}[Analogous to Lemma 3.2 in \cite{BGL}]\label{lem:J}
For all $q\in[0,\tfrac{1}{2})$, there exists $K=K(D_{R_{out}},q,\gamma_p)$ such that, for all $f\in H^{q-1}(D_{R_{out}})$, $\mathcal{J}f\in H^{1+q}(D_{R_{out}})$ and
 \begin{equation}
  \lVert\mathcal{J}f\rVert_{H^{1+q}(D_{R_{out}})}\leq K\lVert f \rVert_{H^{q-1}(D_{R_{out}})},
 \end{equation}
 and, for all $s\in[0,q]$ and all $f\in H^{s-1}(D_{R_{out}})$, $\mathcal{J}f\in H^{1+s}(D_{R_{out}})$ and 
\begin{equation}\label{eq:sbound}
   \lVert\mathcal{J}f\rVert_{H^{1+s}(D_{R_{out}})}\leq \left(\frac{1}{\gamma_p}\right)^{1-\tfrac{s}{q}}K^{\frac{s}{q}}\lVert f \rVert_{H^{s-1}(D_{R_{out}})},
\end{equation}
with $\gamma_p$ as in Proposition \ref{prop:conty}.
\end{lemma}
\begin{proof}
 For $s=0$ we have:
 \begin{align*}
  \gamma_p \lVert \mathcal{J}f\rVert^2_{H^1(D_{R_{out}})}&\leq \langle\hat{\nabla} (\mathcal{J}f),\hat{\nabla}(\mathcal{J}f)\rangle_{\langle \VL^2(D_{R_{out}}),\VL^2(D_{R_{out}})\rangle}-\langle\operatorname{DtN}(\mathcal{J}f),\mathcal{J}f\rangle_{\langle H^{-\frac{1}{2}}(\partial D_{R_{out}}),H^{\frac{1}{2}}(\partial D_{R_{out}})\rangle}\\
  &= \langle f,\mathcal{J}f\rangle_{\langle (H^1(D_{R_{out}}))',H^1(D_{R_{out}})\rangle}\\
  &\leq \lVert f\rVert_{(H^1(D_{R_{out}}))'}\lVert\mathcal{J}f\rVert_{H^1(D_{R_{out}})},
 \end{align*}
and thus
\begin{equation*}
 \lVert\mathcal{J}f\rVert_{H^1(D_{R_{out}})}\leq\frac{1}{\gamma_p}\lVert f \rVert_{(H^1(D_{R_{out}}))'}.
\end{equation*}
For $s=q$, the inequality follows from Theorem 4 and Remark 4.5 in \cite{Sav}. The latter ensures the existence of a constant $K=K(D_{R_{out}},q,\gamma_p)$ such that
\begin{equation*}
 \lVert\mathcal{J}f\rVert_{H^{1+q}(D_{R_{out}})}\leq K  \lVert f\rVert_{H^{q-1}(D_{R_{out}})},
\end{equation*}
for all $f\in H^{q-1}(D_{R_{out}})$. The estimate \eqref{eq:sbound} is then obtained by interpolation.
\end{proof}
In the next lemma, we use the symbol $\mathcal{E}_{\nu}$ to denote the multiplier associated to a tensor $\nu$ \cite{Joc,BGL}.
\begin{lemma}[Similar to Prop. 2.1 in \cite{BGL}]\label{lem:mult}
 If a tensor $\nu$ belongs to $C^{\hexp}_{pw}(\overline{D_{R_{out}}})$, then, for $q\in[0,\hexp)$, there exists a constant $C>0$ such that
 \begin{equation}\label{eq:multiplier}
  \lVert\mathcal{E}_{\nu}\rVert_{\VH^q\rightarrow \VH^q}\leq \nu_{\max}N_{\nu,q},\quad\text{with }N_{\nu,q}=\max\left\{1,\frac{C \lVert \nu\rVert_{C^{\hexp}_{pw}(\overline{D_{R_{out}}})}}{\nu_{\max}}\right\}
 \end{equation}
and $\nu_{\max}$ the maximum singular value of $\nu$. Moreover, for $s\in[0,q)$,
\begin{equation}\label{eq:multiplier2}
   \lVert\mathcal{E}_{\nu}\rVert_{\VH^s\rightarrow \VH^s}\leq \nu_{\max}N_{\nu,q}^{\frac{s}{q}}.
\end{equation}
\end{lemma}
\begin{proof}
For $s=0$, $\lVert\mathcal{E}_{\nu}\Bv\rVert_{\VL^2}\leq \nu_{\max}\lVert\Bv\rVert_{\VL^2}$, for every $\Bv\in\VL^2(D_{R_{out}})$. Equation \eqref{eq:multiplier} is a direct consequence of Lemma 2 in \cite{Joc}. Then \eqref{eq:multiplier2} is obtained by interpolation.
\end{proof}

We are now ready to address the proof of Theorem \ref{thm:contreg}. For this, we proceed on the lines of the proof of Theorem 3.1 in \cite{BGL}, with our modified definition of the operators $\mathcal{J}$ and $\mathcal{S}$ as in \eqref{eq:J}-\eqref{eq:S}.

For a positive number $k>0$, we can write:
\begin{equation*}
 \tilde{f}_s (\By) = \mathcal{S}(\hat{\alpha}(\By;\cdot)\hat{\nabla}\hat{u}_s(\By))=\mathcal{S}\hat{\nabla}(k\hat{u}_s(\By))-\mathcal{S}\left(\left(\bbI - \frac{1}{k}\hat{\alpha}(\By;\cdot)\right)\hat{\nabla}(k\hat{u}_s(\By))\right)\,\text{in }\left(H^1(D_{R_{out}})\right)',
\end{equation*}
for every $J\in\bbN$ and every $\By\in\pspace_J$, where $\bbI\in\bbR^{2\times 2}$ denotes the identity matrix (for the equation above, we remind that $\hat{\alpha}(\By;\hat{\Bx})=\bbI$ for $\hat{\Bx}\in\partial D_{R_{out}}$, for every $J\in\bbN$ and every $\By\in\pspace_J$).

We set $\bar{\alpha}(\By;\cdot):=\bbI- \frac{1}{k}\hat{\alpha}(\By;\cdot)$ and $\hat{w}(\By):=k\hat{u}_s(\By)$. Since $\hat{w}(\By)\in H^1(D_{R_{out}})$ fulfills the radiation condition, we have that $\mathcal{J}\mathcal{S}\hat{\nabla}\hat{w}=\hat{w}$. Thus, we can write:
\begin{equation*}
 \hat{w}(\By)-\mathcal{Q}(\By;\hat{w}(\By))=\mathcal{J} \tilde{f}_s (\By),\quad\text{for every }J\in\bbN\text{ and }\By\in\pspace_J,
\end{equation*}
with $\mathcal{Q}(\By):=\mathcal{J}\mathcal{S}(\mathcal{E}_{\bar{\alpha}(\By)}\hat{\nabla})$. If we can show that $\mathcal{Q}(\By)\in \mathcal{L}\left(H^{\beta'+1}(D_{R_{out}}),H^{\beta'+1}(D_{R_{out}})\right)$ and $\lVert\mathcal{Q}\rVert_{H^{\beta'+1}\rightarrow H^{\beta'+1}}\leq C_Q<1$ for every $J\in\bbN$, every $\By\in\pspace_J$ and some $\beta'\in(0,1)$, then
\begin{equation}\label{eq:kubound}
 k\lVert\hat{u}_s(\By)\rVert_{H^{\beta'+1}(D_{R_{out}})}\leq\frac{\lVert\mathcal{J}\rVert}{1-C_Q}\lVert\tilde{f}_s(\By)\rVert_{H^{\beta'-1}(D_{R_{out}})},
\end{equation}
for every $\By$. If $k$ is chosen to be independent of $J$ and $\By$, the claim of Theorem \ref{thm:contreg} follows once we prove that $\lVert\tilde{f}_s(\By)\rVert_{H^{\beta'-1}(D_{R_{out}})}$ has a $J$- and $\By$-independent bound.

We now show that, for every $\By$, $\mathcal{Q}(\By)$ is a contraction from $H^{\beta'+1}(D_{R_{out}})$ to $H^{\beta'+1}(D_{R_{out}})$. For $\hat{v}\in H^{\beta'+1}(D_{R_{out}})$, we have that $\hat{\nabla}\hat{v}\in \VH^{\beta'}(D_{R_{out}})$. Since, by assumption, $\hat{\alpha}(\By;\cdot)$ and thus $\bar{\alpha}(\By;\cdot)$ are piecewise H\"older continuous, Lemma \ref{lem:mult} ensures that $\bar{\alpha}(\By)\hat{\nabla}\hat{v}\in\VH^{\beta'}(D_{R_{out}})$ for any $0<\beta'<\beta$, and every $J\in\bbN$, $\By\in\pspace_J$. Finally, using Lemmas \ref{lem:S} and \ref{lem:J}, we obtain that $\mathcal{Q}(\By;\hat{v})\in H^{\beta'+1}(D_{R_{out}})$. Moreover, for every $J\in\bbN$, every $\By\in\pspace_J$ and for $0<\beta'<q<\hexp$:
\begin{align*}
 \lVert\mathcal{Q}(\By)\rVert_{H^{\beta'+1}\rightarrow H^{\beta'+1}}&\leq \lVert\mathcal{J}\rVert_{H^{\beta'-1}\rightarrow H^{\beta'+1}}\lVert\mathcal{S}\rVert_{H^{\beta'}\rightarrow H^{\beta'-1}}\lVert\mathcal{E}_{\bar{\alpha}(\By)}\rVert_{\VH^{\beta'}\rightarrow\VH^{\beta'}}\\
 &\leq \left(\frac{1}{\gamma_p}\right)^{1-\tfrac{\beta'}{q}}K^{\frac{\beta'}{q}}\mathcal{C}^{1-\beta'}\bar{\alpha}_{\max}(\By)N_{\bar{\alpha}(\By),q}^{\frac{\beta'}{q}}\\
 &\leq \frac{\mathcal{C}}{\gamma_p}(K\mathcal{C}^{1-q}N_{\bar{\alpha}(\By),q})^{\frac{\beta'}{q}}\bar{\alpha}_{\max}(\By),
\end{align*}
where for the last inequality we have used that $\frac{\mathcal{C}}{\gamma_p}\geq 1$, and $\bar{\alpha}_{\max}(\By)$ denotes the maximum singular value of $\bar{\alpha}(\By;\cdot)$. Let us denote by $\Lambda_{min}$ and $\Lambda_{max}$, respectively, the $J$- and $\By$-independent lower and upper bounds on the eigenvalues of $\hat{\alpha}$ (these bounds exist thanks to the bounds on the singular values of $D\Phi$). Then, if we choose $k$ such that $1-\frac{1}{k}\Lambda_{max}>0$, we have $0<1-\frac{1}{k}\Lambda_{max}\leq \bar{\alpha}_{\max}(\By) \leq 1-\frac{1}{k}\Lambda_{min}$ for every $J$ and every $\By$, and thus $\bar{\alpha}_{\max}(\By)$ has a $J$- and $\By$-independent upper bound. For the same reason and because of the $J$- and $\By$-uniform upper bound on the H\"older norm of $\hat{\alpha}$, if we choose $k$ independent of $J$ and $\By$ then $N_{\bar{\alpha}(\By),q}$ has a $J$- and $\By$-independent upper bound $N_{q}$.

The norm $\lVert\mathcal{Q}(\By)\rVert_{H^{\beta'+1}\rightarrow H^{\beta'+1}}$ has a $J$- and $\By$-independent bound which is smaller than one if
\begin{equation}
 \beta'< q \min\left\{1,\frac{\log\left(\frac{\gamma_p}{\mathcal{C}\left(1-\frac{1}{k}\Lambda_{min}\right)}\right)}{\log(K\mathcal{C}^{1-q}N_q)}\right\}.
\end{equation}
A $\beta'>0$ exists if $\left(1-\frac{1}{k}\Lambda_{min}\right)<\frac{\gamma_p}{\mathcal{C}}$, which, combined with the requirement that $0<1-\frac{1}{k}\Lambda_{max}$, implies that we must choose $\Lambda_{max}<k<\frac{\Lambda_{min}}{1-\frac{\gamma_p}{\mathcal{C}}}$.  Such a $k$ exists and can be chosen independently of $J$ and $\By$ if $\Lambda_{max}<\frac{\Lambda_{min}}{1-\frac{\gamma_p}{\mathcal{C}}}$, that is $\frac{\Lambda_{min}}{\Lambda_{max}}>1-\frac{\gamma_p}{\mathcal{C}}$ and this is ensured by the requirement $\frac{\sigma_{min}^4}{\sigma_{max}^4}\min\left\{\frac{1}{\alpha_2},\alpha_2\right\}\geq 1-\frac{\gamma_p}{\mathcal{C}}$. Note that $N_{q}\geq 1$ by definition, and $K\mathcal{C}^{1-q}\geq 1$ because $\mathcal{J}\mathcal{S}\hat{\nabla}\hat{w}=\hat{w}$ for all $\hat{w}$ in $H^1(D_{R_{out}})$ that satisfy the radiation condition (see Remark 3.2 in \cite{BGL}).

To complete the proof of Theorem \ref{thm:contreg}, we have to show a $J$- and $\By$-uniform bound on $\lVert\tilde{f}_s(\By)\rVert_{H^{\beta'-1}(D_{R_{out}})}$. We have:
\begin{align*}
 \lVert\tilde{f}_s(\By)\rVert_{H^{\beta'-1}(D_{R_{out}})}\leq &\,\lVert \mathcal{E}_{\hat{\alpha}(\By)} \hat{\nabla} \uref_i(\By)\rVert_{H^{\beta'}(D_{R_{out}})} + \bar{C}\lVert \mathcal{E}_{\hat{\kappa}(\By)}\uref(\By)\rVert_{L^2(D_{R_{out}})}\\
 \leq &\,C \lVert \hat{\alpha}(\By;\cdot)\rVert_{C_{pw}^{\hexp}(\overline{D_{R_{out}}})}\lVert \uref_i(\By)\rVert_{H^{\beta'+1}(D_{R_{out}})}\\
 & + \bar{C}\lVert \hat{\kappa}(\By;\cdot)\rVert_{C_{pw}^0(\overline{D_{R_{out}}})}\lVert\uref(\By)\rVert_{L^2(D_{R_{out}})}
\end{align*}
with $C$ (the same as in \eqref{eq:multiplier}) and $\bar{C}$ independent of $J$ and $\By$. The norms of the coefficients have a $J$- and $\By$-independent bound thanks to the assumptions of Theorem \ref{thm:contreg}. The norm $\lVert\uref(\By)\rVert_{L^2(D_{R_{out}})}$ can be bounded as
\begin{equation*}
 \lVert\uref(\By)\rVert_{L^2(D_{R_{out}})}\leq \lVert\uref(\By)\rVert_{H^1(D_{R_{out}})}\leq C\left(\lVert u_i\rVert_{H^{\frac{1}{2}}(\partial D_{R_{out}})} + \Big\lVert\dfrac{\partial u_i}{\partial\Bn_{out}}\Big\rVert_{H^{-\frac{1}{2}}(\partial D_{R_{out}})}\right),
\end{equation*}
where $C$ is a $J$- and $\By$-independent constant, thanks to Assumptions \ref{Ak1k2} and the properties of $\Phi$ \cite[Cor. 3.2.6]{LSthesis}. Finally, there exists $C$ independent of $J$ and $\By$ such that $\lVert \uref_i(\By)\rVert_{H^{\beta'+1}(D_{R_{out}})}\leq C \lVert \uref_i(\By)\rVert_{C^{1,\beta'}_{pw}(\overline{D_{R_{out}})}}$ $\leq C\max\left\{1,\lVert D\Phi(\By)\rVert_{C^{\beta'}_{pw}(\overline{D_{R_{out}}})}\right\}\lVert u_i\rVert_{C^{1,\beta'}(\overline{D_{R_{out}}})}$, and $\lVert D\Phi(\By)\rVert_{C^0_{pw}(\overline{D_{R_{out}}})}$ is uniformly bounded with respect to $J$ and $\By$ thanks to the uniform bounds on the radius.\label{appendix:continuity}
\section{Schauder estimates for the transmission problem}
We present here the proof to Theorem \ref{thm:ureg}. We adapt the results of \cite[Ch. 6]{GT} and \cite[Ch. 6]{WYW}, stated for boundary value problems, to the transmission problem \eqref{eq:varform}, with particular emphasis on having constants which are independent of $J\in\bbN$ and $\By\in\pspace_J$. We first address the local regularity at the interface, then the interior regularity, and finally the global regularity estimate. Also, we prove these estimates for $k=2$ in \eqref{eq:holderbound1}, and extend them for any $k\geq 2$ at the end. We denote generically by $n$ the spatial dimension, that in Theorem \ref{thm:ureg} is $n=2$. 

We use the following abbreviations for norms and seminorms:
 \begin{notation} Let $v(\Bx)$ be a function on $\overline{\Omega}\subset\bbR^n$, $n\geq 1$. For $\hexp\in(0,1)$ and $k\in\bbN$, we use the following notation for the norms and seminorms in $C^{k}(\overline{\Omega})$:
\begin{equation*}
  \cont{v}{\overline{\Omega}}=\sup_{\Bx\in\overline{\Omega}}|v(\Bx)|,\quad \scont[k]{v}{\overline{\Omega}}=\sum_{|\Bnu|=k}\cont{\dfrac{\partial^{|\Bnu|}v}{\partial \Bx^{\Bnu}}}{\overline{\Omega}},\quad \cont[k]{v}{\overline{\Omega}}=\sum_{|\Bnu|\leq k}\cont{\dfrac{\partial^{|\Bnu|}v}{\partial \Bx^{\Bnu}}}{\overline{\Omega}},
\end{equation*}
and for the norms and seminorms in $C^{k,\hexp}(\overline{\Omega})$:
\begin{align*}
 \sholder{v}{\overline{\Omega}}=\sup_{\substack{{\Bx_1,\Bx_2\in\overline{\Omega},}\\{\Bx_1\neq\Bx_2}}}\frac{|v(\Bx_1)-v(\Bx_2)|}{|\Bx_1-\Bx_2|^{\hexp}}, \quad& \sholderk{v}{\overline{\Omega}}{k}=\sum_{|\Bnu|=k}\sholder{\dfrac{\partial^{|\Bnu|}v}{\partial \Bx^{\Bnu}}}{\overline{\Omega}},\\
 \holder{v}{\overline{\Omega}}=\cont{v}{\overline{\Omega}}+ \sholder{v}{\overline{\Omega}},\quad& \holderk{v}{\overline{\Omega}}{k}=\sum_{|\Bnu|\leq k}\sholder{\dfrac{\partial^{|\Bnu|}v}{\partial \Bx^{\Bnu}}}{\overline{\Omega}}.
\end{align*}
If $\Omega=\Omega_{in}\cup\Gamma\cup\Omega_{out}$, where $\Gamma$ is an interface separating the two subdomains, then we denote
\begin{align*}
 \sholderpw{v}{\Omega}=\sholder{v}{\Omega_{in}\cup\Gamma}+\sholder{v}{\Omega_{out}\cup\Gamma},\qquad&\holderpw{v}{\Omega}=\holder{v}{\Omega_{in}\cup\Gamma}+\holder{v}{\Omega_{out}\cup\Gamma},\\
  \sholderpw{v}{\overline{\Omega}}=\sholder{v}{\overline{\Omega_{in}}}+\sholder{v}{\overline{\Omega_{out}}},\qquad&\holderpw{v}{\overline{\Omega}}=\holder{v}{\overline{\Omega_{in}}}+\holder{v}{\overline{\Omega_{out}}},
\end{align*}
and analogously for the piecewise-$C^k$ and piecewise-$C^{k,\hexp}$ norms and seminorms.
\end{notation}

\subsection{Local estimates at the interface $\hat{\Gamma}$}
Without loss of generality, we assume $\hat{\Gamma}$ to be the boundary of the upper half-plane, as every $C^{2,\hexp}$ boundary is $C^{2,\hexp}$-diffeomorphic to the upper half-plane (with $J$- and $\By$-independent continuity constants).

The standard technique to prove Schauder estimates for the solution to \eqref{eq:varform} is the method of solidifying coefficients (see \cite[Sect. 6.3.2]{WYW} and \cite[Proof of Thm. 6.2]{GT}).

Fixed $\xreffix\in\hat{\Gamma}$ and a ball $B_R(\xreffix)$ of radius $R$ centered in $\xreffix$, we can write \eqref{eq:varform} restricted to $B_R(\xreffix)$ as

\begin{subequations}\label{eq:modelpbconst}
\begin{align}[left=\empheqlbrace]
& -\refgrad\cdot\left(\aref(\By;\xreffix)\refgrad \uref\right)= \freffix(\By;\xref) \quad \text{in }B_R^{-}(\xreffix)\cup B_R^{+}(\xreffix),\\
& \llbracket \uref \rrbracket_{B_R(\xreffix)\cap\gammaref} = 0, \quad\Big\llbracket \aref(\By;\xreffix)\dfrac{\partialref \uref}{\partialref \nref}\Big\rrbracket_{B_R(\xreffix)\cap\gammaref} = \greffix(\By;\xref),\label{eq:tcconst}
\end{align}
\end{subequations}
with
\begin{align}
 \freffix(\By;\xref)&:=\cref(\By;\xref)\uref+\refgrad\cdot\left((\aref(\By;\xref)-\aref(\By;\xreffix))\refgrad\uref\right),\label{eq:fhat}\\
 \greffix(\By;\xref)&:=\Big\llbracket \left(\aref(\By;\xreffix)-\aref(\By;\xref)\right)\dfrac{\partialref \uref}{\partialref \nref}\Big\rrbracket_{B_R(\xreffix)\cap\gammaref},\label{eq:ghat}
\end{align}
and with $B_R^{+}(\xreffix):=B_R(\xreffix)\cap\hdoutR$ and $B_R^{-}(\xreffix):=B_R(\xreffix)\cap\hdin$. We develop our analysis taking
\begin{equation}\label{eq:Fsimplified}
 \freffix(\By;\xref) = \fref(\By;\xref)+(\aref(\By;\xref)-\aref(\By;\xreffix))\circ\hat{\jac}^2 \uref,
\end{equation}
where $(A\circ B)_{ij}=A_{ij}B_{ij}$, $i,j=1,\ldots,n$, is the Hadamard product for matrices, and $\hat{\jac}^2 \uref$ denotes the Hessian matrix of $\uref$. The term $\fref(\By;\xref)$ is a generic right-hand side, possibly including lower order terms; in our case, $\fref(\By;\xref)=\cref(\By;\xref)\uref+\left(\refgrad\cdot \aref(\By;\xref)\right)\cdot\refgrad\uref$.

Since, by assumption, the constant matrix $\aref(\By;\xreffix)$ is symmetric positive definite, for every $\By\in\pspace_J$ and every $J\in\bbN$ there exists an orthonormal matrix $\omatrix_y$, dependent on $\By$ and $J$, such that
\begin{equation}\label{eq:diaga}
 \omatrix^{\top}_y\aref(\By;\xreffix)\omatrix_y=\Lambda_y,%\quad\text{for every }\By\in\pspace_J \text{ and every }J\in\bbN,
\end{equation}
with $\diagmatrix_y$ a diagonal matrix dependent of $\By\in\pspace_J$, $J\in\bbN$. The entries of $\diagmatrix_y$ have $J$- and $\By$-uniform lower and upper bounds, because, by assumption, $\hat{\alpha}(\By;\cdot)$ has $J$-, $\By$- and $\hat{\Bx}$-uniform lower and upper bounds on its singular values. We denote these bounds by $\evmin$ and $\evmax$, respectively.

Introducing the change of coordinates $\xif=\omatrix^{\top}_y\xref$ for $\xref\in B_R(\xreffix)$, and using the symbols $\xifgrad$ and $\partialxif$ to denote differentiation with respect to $\xif$, \eqref{eq:modelpbconst} becomes:
\begin{subequations}\label{eq:modelpbxi}
\begin{align}[left=\empheqlbrace]
& -\xifgrad\cdot\left(\diagmatrix_y\xifgrad \uxif\right)= \fxif_y \quad \text{in }B_R^{-}(\xiffix)\cup B_R^{+}(\xiffix),\label{eq:pdexi}\\
& \llbracket \uxif \rrbracket_{\gammaxif_y} = 0, \quad\Big\llbracket \diagmatrix_y\dfrac{\partialxif \uxif}{\partialxif \nxif}\Big\rrbracket_{\gammaxif_y} = \gxif_y,\label{eq:tcxi}
\end{align}
\end{subequations}
for every $\By\in\pspace_J$ and every $J\in\bbN$. We have denoted $\fxif_y:=\freffix(\By;\omatrix_y\xif)$ and $\gxif_y:=\greffix(\By;\omatrix_y\xif)$. Since $\omatrix_y$ is orthonormal, $B_R(\xreffix)$ is mapped to another ball with the same radius $R$ and just a different center $\xiffix$. In \eqref{eq:pdexi}, $B_R^{-}(\xiffix)$ is the preimage of $B_R^{-}(\xreffix)$ under $\omatrix_y$, corresponding to a half-ball with radius $R$ and center in $\xiffix=\omatrix_y^{\top}\xreffix$; the same convention applies for $B_R^{+}(\xiffix)$. In \eqref{eq:tcxi}, $\gammaxif$ is a short notation for the preimage of $B_R(\xreffix)\cap\gammaref$ under $\omatrix_y$.

In the following, $\Ckpwxi{k}{B_R(\xiffix)}:=C^{k,\hexp}(B_R^{+}(\xiffix)\cup\gammaxif)\cup C^{k,\hexp}(B_R^{-}(\xiffix)\cup\gammaxif)$, $k\in\bbN$.

In subsection \ref{A2}, we provide the proof to the following lemma:
\begin{lemma}\label{lem:locestxi}
 Let $R>0$ and let $\uxif$ be a solution to \eqref{eq:modelpbxi} in $B_R(\xiffix)$, $\xiffix\in\gammaxif_y$. If $\fxif_y\in\Ckpwxi{0}{B_R^{-}(\xiffix)}$ and $\gxif_y\in C^{1,\hexp}(\gammaxif_y)$, then
 \begin{align}
 \sholderkpwb{\uxif}{R/2}{2}\leq C &\left(\frac{1}{R^{2+\hexp}}\contpwb{\uxif}{R}+\frac{1}{R^{\hexp}}\contpwb{\fxif_y}{R}+\sholderpwb{\fxif_y}{R}\right.\nonumber\\
  &\left.+\frac{1}{R^{1+\hexp}}\cont{\gxif_y}{\gammaxif_y}+\sholderk{\gxif_y}{\gammaxif_y}{1}\right).\label{eq:xiest}
 \end{align}
The constant $C=C(n,\evmax,\evmin,\hexp)$ is \emph{independent} of the center $\xiffix$ of the ball $B_R$, of $J\in\bbN$ and of $\By\in\pspace_J$.
\end{lemma}

Inserting in \eqref{eq:xiest} the expressions obtained from \eqref{eq:Fsimplified} and \eqref{eq:ghat}, and denoting $\axif:=\aref(\By;\omatrix_y\xif)$ and $\fxifs_y:=\fref(\By;\omatrix_y\xif)$,  we obtain:
\begin{lemma}\label{lem:locestxi2}
  Let $0<R\leq 1$ and let $\uxif$ be a solution to \eqref{eq:modelpbxi} in $B_R(\xiffix)$, $\xiffix\in\gammaxif_y$, with $\freffix$ given by \eqref{eq:Fsimplified}. Let the assumptions of Theorem \ref{thm:ureg} hold. Then
  \begin{equation}
  \sholderkpwb{\uxif}{R/2}{2}\leq C\left(\frac{1}{R^{2+\hexp}}\cont{\uxif}{B_R(\xiffix)}+R^{\hexp}\sholderkpwb{\uxif}{R}{2} + \frac{1}{R^{\hexp}}\contpwb{\fxifs_y}{R}+\sholderpwb{\fxifs_y}{R}\right).\label{eq:xiest2}
  \end{equation}
  The constant $C=C\left(n,\evmin,\evmax,\holderkpwb{\axif}{R}{1}\right)$
is \emph{independent} of the center $\xiffix$ of the ball $B_R$, of $J\in\bbN$ and of $\By\in\pspace_J$.
\end{lemma}
\begin{proof}
  In this proof we denote $B_R:=B_R(\xiffix)$ and use the symbol $\tilde{\jac}^2$ for the Hessian with respect to $\xif$.
  
  We have $\fxif_y=\fxifs_y(\By;\xif)+\left(\omatrix_y^{\top}\left(\axif(\By;\xif)-\axif(\By;\xiffix)\right)\omatrix_y\right) : \tilde{\jac}^2_y \uxif$ and $\gxif_y=\Big\llbracket(\axif(\By;\xiffix)-\axif(\By;\xif))\omatrix_y\dfrac{\partialxif \uxif}{\partialxif \nxif}\Big\rrbracket_{\gammaxif_y}$. Using interpolation inequalities (cf. \cite[Cor. 1.2.1]{WYW}), together with the fact that $\omatrix_y$ is orthonormal for every $\By$, and that $0<R\leq 1$, we obtain:
 \begin{align*}
  \sholderpwbr{\omatrix_y^{\top}\left(\axif(\By;\xif)-\axif(\By;\xiffix)\right)\omatrix_y \circ \tilde{\jac}^2_y \uxif}{R}&\leq \scontpwbr[2]{\uxif}{R}\sholderpwbr{\axif}{R}+\sholderkpwbr{\uxif}{R}{2}\sholderpwbr{\axif}{R}R^{\hexp}\\
  & \leq \holderpwbr{\axif}{R}\left(\scontpwbr[2]{\uxif}{R}+R^{\hexp}\sholderkpwbr{\uxif}{R}{2}\right)\\
  &\leq \holderpwbr{\axif}{R}\left(\frac{1}{R^2}\contpwbr{\uxif}{R}+2 R^{\hexp}\sholderkpwbr{\uxif}{R}{2}\right)\\
    &\leq \holderpwbr{\axif}{R}\left(\frac{1}{R^{2+\hexp}}\contpwbr{\uxif}{R}+2 R^{\hexp}\sholderkpwbr{\uxif}{R}{2}\right),
    \end{align*}
    \begin{align*}
    \contpwbr{\omatrix_y^{\top}\left(\axif(\By;\xif)-\axif(\By;\xiffix)\right)\omatrix_y \circ \tilde{\jac}^2_y \uxif}{R}&\leq  R^{\hexp}\scontpwbr[2]{\uxif}{R}\sholderpwbr{\axif}{R}\\
    &\leq R^{\hexp}\sholderpwbr{\axif}{R}\left(\frac{1}{R^{2+\hexp}}\contpwbr{\uxif}{R}+R^{\hexp}\sholderkpwbr{\uxif}{R}{2}\right)
    \end{align*}
    and
    \begin{align*}
  \sholderk{\gxif_y}{\gammaxif_y}{1}&\leq \scontpwbr[1]{\uxif}{R}\sholderkpwbr{\axif}{R}{1}+\sholderkpwbr{\uxif}{R}{1}\scontpwbr[1]{\axif}{R}+\scontpwbr[2]{\uxif}{R}\sholderpwbr{\axif}{R}+R^{\hexp}\sholderkpwbr{\uxif}{R}{2}\sholderpwbr{\axif}{R}\\
  &\leq 2\holderkpwbr{\axif}{R}{1}\left(\left(\frac{1}{R}+\frac{1}{R^{1+\hexp}}+\frac{1}{R^2}\right)\contpwbr{\uxif}{R}+\left(R^{1+\hexp}+R+R^{\hexp}\right)\sholderkpwbr{\uxif}{R}{2}\right)\\
  &\leq 8\holderkpwbr{\axif}{R}{1}\left(\frac{1}{R^{2+\hexp}}\contpwbr{\uxif}{R}+R^{\hexp}\sholderkpwbr{\uxif}{R}{2}\right),\\
  \cont{\gxif_y}{\gammaxif_y}&\leq R^{\hexp}\sholderpwbr{\axif}{R}\scontpwbr[1]{\uxif}{R}\\
  &\leq R^{\hexp}\sholderpwbr{\axif}{R}\left(\frac{1}{R}\cont{\uxif}{B_R}+R^{1+\hexp}\sholderkpwbr{\uxif}{R}{2}\right)\\
  &= R^{1+\hexp}\sholderpwbr{\axif}{R}\left(\frac{1}{R^2}\cont{\uxif}{B_R}+R^{\hexp}\sholderkpwbr{\uxif}{R}{2}\right)\\
  &\leq  R^{1+\hexp}\sholderpwbr{\axif}{R}\left(\frac{1}{R^{2+\hexp}}\cont{\uxif}{B_R}+R^{\hexp}\sholderkpwbr{\uxif}{R}{2}\right).
  \end{align*}
Inserting these estimates in \eqref{eq:xiest}, we obtain \eqref{eq:xiest2}.
\end{proof}

We now return to the variable $\xref$. We notice that $\omatrix_y\left(B_R(\xiffix)\right)=B_R(\xreffix)$, and, thanks to the orthonormality of $\omatrix_y$, the H\"older norms in the $\xref$-space and in the $\xif$-space do coincide, for every $\xreffix\in \gammaref$, every $\By\in\pspace_J$ and every $J\in\bbN$. Let us denote by $C_{\alpha}$ the $J$- and $\By$-uniform upper bound on the $C_{pw}^{1,\hexp}(\overline{D_{R_{out}}})$-norm of  $\hat{\alpha}(\By;\cdot)$. From \eqref{eq:xiest2}, we have the following estimate:
\begin{equation}\label{eq:estloc1}
      \sholderkpw{\uref}{B^{\pm}_{R/2}(\xreffix)}{2}\leq C\left(\frac{1}{R^{2+\hexp}}\cont{\uref}{\overline{\dom}}+R^{\hexp}\sholderkpw{\uref}{\overline{\dom}}{2} + \frac{1}{R^{\hexp}}\contpw{\fref}{\overline{\dom}}+\sholderpw{\fref}{\overline{\dom}}\right),
\end{equation}
with $C=C\left(n,\evmin,\evmax,C_a\right)$ \emph{independent} of $\xreffix$, of $J\in\bbN$ and of $\By\in\pspace_J$.

To take into account the lower order terms in \eqref{eq:fhat}, we proceed as following:
\begin{itemize}
 \item we write, in \eqref{eq:estloc1}, $\fref(\By;\xref)=\cref(\By;\xref)\uref+\left(\refgrad\cdot \aref(\By;\xref)\right)\cdot\refgrad\uref$;
 \item use the interpolation inequalities (cf. \cite[Cor. 1.2.1]{WYW}) to obtain the bounds $\scontpw[1]{\uref}{\overline{\dom}}\leq R^{1+\hexp}\sholderkpw{\uref}{\overline{\dom}}{2} + \frac{1}{R}\cont{\uref}{\overline{\dom}}$ and $\sholderkpw{\uref}{\overline{\dom}}{1}\leq R\sholderkpw{\uref}{\overline{\dom}}{2} + \frac{1}{R^{1+\hexp}}\cont{\uref}{\overline{\dom}}$;
 \item exploit that $0<R\leq 1$, as we have done in the proof of Lemma \ref{lem:locestxi2}.
\end{itemize}

Summarizing the last steps, the local estimate at $\gammaref$ reads:
\begin{theorem}\label{thm:locestgamma}
 Let the assumptions of Theorem \ref{thm:ureg} be fulfilled, and let us denote by $C_{\alpha}$ and $C_{\kappa}$ the $J$- and $\By$-independent upper bounds on the $C_{pw}^{1,\hexp}(\overline{D_{R_{out}}})$-norm of  $\hat{\alpha}(\By;\cdot)$ and on the $C_{pw}^{0,\hexp}(\overline{D_{R_{out}}})$-norm of $\hat{\kappa}^2(\By,\cdot)$, respectively. If $\uref\in C_{\hat{pw}}^{2,\hexp}(\overline{\dom})$ is a solution to \eqref{eq:varform}, then, for every $\xreffix\in\gammaref$:
 \begin{equation}\label{eq:estloc2}
      \sholderkpw{\uref}{B^{\pm}_{R/2}(\xreffix)}{2}\leq C\left(\frac{1}{R^{2+\hexp}}\cont{\uref}{\overline{\dom}}+R^{\hexp}\sholderkpw{\uref}{\overline{\dom}}{2}\right),
\end{equation}
for a radius $0<R< \min\left\{1,\emph{dist}\left(\xreffix,\partial D_{R_{out}}\right)\right\}$ such that $B_R^{+}(\xreffix)\subset\doutR\cup\gammaref$ and $B_R^{-}(\xreffix)\subset\din\cup\gammaref$. 

The constant $C=C\left(n,\evmin,\evmax,C_{\alpha},C_{\kappa}\right)$ in \eqref{eq:estloc2} is independent of $\xreffix$, of $J\in\bbN$ and of $\By\in\pspace_J$.
\end{theorem}

\subsection{Local interior estimates}
Proceeding as for the local estimate at $\gammaref$, it is easy to verify that analogous estimates hold in the interior of $\din$ and $\doutR$:

\begin{theorem}\label{thm:locestint}
  Let the assumptions of Theorem \ref{thm:ureg} be fulfilled. If $\uref\in C_{\hat{pw}}^{2,\hexp}(\overline{D}_{R_{out}})$ is a bounded solution to \eqref{eq:varform}, then, for every $\xref\in\hdin\cup\hdoutR$:
 \begin{equation}\label{eq:estlocint}
      \sholderk{\uref}{B_{R/2}(\xref)}{2}\leq C\left(\frac{1}{R^{2+\hexp}}\cont{\uref}{\overline{\dom}}+R^{\hexp}\sholderkpw{\uref}{\overline{\dom}}{2}\right),
\end{equation}
for a radius $0<R< \min\left\{1,\emph{dist}\left(\xref,\partial D_{R_{out}}\right)\right\}$ if $\xref\in\hdoutR$, $0<R< \min\left\{1,\emph{dist}\left(\xref,\gammaref\right)\right\}$ if $\xref\in\hdin$. 
The constant $C=C\left(n,\evmin,\evmax,C_{\alpha}, C_{\kappa}\right)$ in \eqref{eq:estlocint} is independent of $\xref$, of $J\in\bbN$ and of $\By\in\pspace_J$ (with $C_{\alpha}$ and $C_{\kappa}$ as in Theorem \ref{thm:locestgamma}).
\end{theorem}
\begin{proof}
 We refer to \cite{WYW}, Sections 6.2.2, 6.2.3, 6.2.6 and 6.3.2. We remark that the $J$- and $\By$-uniform ellipticity condition on $\hat{\alpha}$, together with $J$- and $\By$-independence of the norms of $\aref$ and $\cref$, ensure the $J$- and $\By$-independence of the constant $C$ in \eqref{eq:estlocint}.
\end{proof}

\subsection{Global estimates}
The local estimates at $\partial D_{R_{out}}$ are very similar to the local estimates at $\hat{\Gamma}$. Therefore, we do not present them explicitly, and refer to \cite[Sect. 6.7]{GT} for details. What we obtain is that, under the assumptions of Theorems \ref{thm:locestgamma} and \ref{thm:locestint}, for every $\xref_{\partial D}\in\partial\dout$ and $0<R< \min\left\{1,\text{dist}\left(\xref_{\partial D},\gammaref\right)\right\}$:
     \begin{equation}\label{eq:estlocbdobl}
      \sholderk{\uref}{B^{-}_{R/2}(\xref_{\partial D})}{2}\leq\, C_1\left(\frac{1}{R^{2+\hexp}}\cont{\uref}{\overline{\dom}}+R^{\hexp}\sholderkpw{\uref}{\overline{\dom}}{2}\right)+ C_2\holderk{u_i}{\overline{\dom}}{2},
    \end{equation}
    with $B^{-}_{R/2}(\xref_{\partial D}):=B_{R/2}(\xref_{\partial D})\cap \overline{\hdoutR}$ and the constants $C_1$ and $C_2$ independent of $\xref_{\partial D}$, of $J\in\bbN$ and $\By\in\pspace_J$ ($C_2$ is possibly depending on $R$).

For the global estimate, we recall that,owing to the interpolation inequalities \cite[Cor. 1.2.1]{WYW}, in order to bound $\holderkpw{\uref}{\overline{\dom}}{2}$ it is sufficient to bound $\sholderkpw{\uref}{\overline{\dom}}{2}$ and $\cont{\uref}{\overline{\dom}}$. Using a finite covering argument on $\overline{\dom}$ together with Theorems \ref{thm:locestgamma}, \ref{thm:locestint} and equation \eqref{eq:estlocbdobl}, we obtain:
\begin{theorem}\label{thm:globest2}
  Let the assumptions of Theorem \ref{thm:ureg} be fulfilled. If $\uref\in C_{pw}^{2,\hexp}(\overline{D}_{R_{out}})$ is a bounded solution to \eqref{eq:varform}, then
       \begin{equation}\label{eq:genestobl}
      \holderkpw{\uref}{\overline{\dom}}{2}\leq C\left(\cont{\uref}{\overline{\dom}}+\holderk{u_i}{\overline{\dom}}{2}\right),
\end{equation}
with a constant $C=C\left(n,\evmin,\evmax,C_{\alpha},C_{\kappa}\right)$ independent $J\in\bbN$ and of $\By\in\pspace_J$ (with $C_{\alpha}$ and $C_{\kappa}$ the $J$- and $\By$-uniform bounds on the norms of the coefficients as in Theorem \ref{thm:locestgamma}).
  \end{theorem}

  To obtain the estimate on $\holderkpw{\uref}{\overline{\dom}}{k}$ for $k>2$, one proceeds considering the difference quotient for $k=3$ and then, for $k>3$, proceeds by induction. The $J$- and $\By$-independence on the constants is preserved, provided the assumptions of Theorem \ref{thm:ureg} are fulfilled. We refer to \cite[Thm. 6.17]{GT} and \cite[Thm. 6.19]{GT} for details.

  \begin{remark}
   It is clear that the regularity results reported in this section are not restricted to the Helmholtz transmission problem. In particular, they still hold true if the elliptic operator contains a transport term $\Bb(\By;\hat{\Bx})\cdot\refgrad\uref$, where $\lVert\hat{\Bb}(\By;\cdot)\rVert_{C_{pw}^{k-2,\hexp}(\overline{D_{R_{out}}})}$ is bounded independently of $J$ and $\By$. Indeed, the results in \cite{GT} and \cite{WYW} (our guidelines throughtout this section) are stated for an elliptic operator containing a tranport term. A nonzero right-hand side in \eqref{eq:varform} can be treated adding it to $\hat{F}$ in \eqref{eq:fhat} and including it in $\hat{f}$ in \eqref{eq:Fsimplified}. An extension to nonhomogeneous transmission conditions at $\hat{\Gamma}$ is also possible: a jump in the Dirichlet trace can be treated similarly to nonhomogeneous Dirichlet boundary conditions, and a jump in the Neumann trace can be added to $\hat{g}$ in \eqref{eq:ghat}. 
  \end{remark}

\subsection{Proof of Lemma \ref{lem:locestxi}}\label{A2}
We present here the proof to the Schauder estimate of Lemma \ref{lem:locestxi}. Schauder estimates can be proved either using Green's representation formula for the solution to \eqref{eq:modelpbxi}, as done in \cite[Ch. 6]{GT}, or using Campanato norms, as in \cite[Ch. 6]{WYW}. Here we follow the latter approach.

We consider the solution to the Poisson equation \eqref{eq:modelpbxi}. In this section, we denote $B_R:=B_R(\xiffix)$, $B_R^{+}:=B_R^{+}(\xiffix)$ and $B_R^{-}:=B_R^{-}(\xiffix)$. Without loss of generality, we assume that $\gammaxif_y=\left\{\xif\in\bbR^n: \tilde{x}_{y,n} =0\right\}$, where $\tilde{x}_{y,i}$ denotes the $i^{th}$ component of $\xif$, $i=1,\ldots,n$. If $\gammaxif_y$ is a generic hyperplane in $\bbR^n$, the estimates we will obtain still work if we substitute derivatives with respect to the cartesian coordinates by derivatives with respect to the normal and tangential directions with respect to $\gammaxif_y$, see \cite[Rmk. 6.2.8]{WYW}. Also, we can assume the solution $\uxif$ to be sufficiently smooth, see Proposition 6.2.1 in \cite{WYW} (the latter still holds true if we consider the transmission problem \eqref{eq:modelpbxi}). 

Analogously to \cite[Sect. 6.7]{GT}, we first assume that $\gxif_y\equiv 0$, and only at the end return to the general case of nonzero $\gxif_y$. 

\subsubsection{Preliminaries}
 This subsection contains some technical lemmas that will be used in the next subsections.
 
 \begin{lemma}[p. 174 in \cite{WYW}]\label{lem:convex}
  For every $w\in L^2(\Omega)$, $\Omega\subset \bbR^n$, the function 
  $$p(\lambda):=\int_{\Omega}\left(w(\Bz)-\lambda\right)^2\dz,\quad\lambda\in\bbR$$
  is strictly convex, and attains its minimum at $\lambda=w_{\Omega}:=\frac{1}{|\Omega|}\int_{\Omega}w(\Bz)\dz$.
 \end{lemma}
\begin{proof}
 By trivial calculations one sees that $\dfrac{d^2 p}{d\lambda^2}(\lambda)\equiv 2|\Omega|>0$ and $\dfrac{dp}{d\lambda}(\lambda)=0\Leftrightarrow \lambda=w_{\Omega}$.
\end{proof}

\begin{lemma}[Thm. 6.1.1 and Rmk. 6.1.2 in \cite{WYW}]\label{lem:campanato}
 Let $B_R\in\bbR^n$ be a ball with radius $R$ and, for any $0<\lambda<1$, consider the quantity
 \begin{equation}\label{eq:campanatonorm}
  \left| w\right|^{(\lambda)}_{p,\mu;B_R}:=\sup_{\substack{\Bx\in B_R,\\ 0<\rho<\lambda \diam B_R}}\left(\rho^{-\mu}\int_{B_{\rho}(\Bx)\cap B_R}|w(\Bx)-w_{\Bx,\rho}|^p\dz\right)^{\frac{1}{p}},
 \end{equation}
where $w_{\Bx,\rho}:=\frac{1}{|B_{\rho}(\Bx)\cap B_R|}\int_{B_{\rho}(\Bx)\cap B_R}w(\Bz)\dz$.

Then, for $\hexp=\frac{\mu-n}{p}\in(0,1]$, \eqref{eq:campanatonorm} is a seminorm equivalent to the H\"older seminorm $\sholder{w}{B_R}$, that is, there exist positive constants $C_1,C_2$, depending only on $n,R,p,\mu$ and $\lambda$, such that
\begin{equation*}
C_1\sholder{w}{B_R}\leq    \left| w\right|^{(\lambda)}_{p,\mu;B_R} \leq C_2 \sholder{w}{B_R}.
\end{equation*}
\end{lemma}

\begin{lemma}[Iteration lemma, Sect. 6.2.5 in \cite{WYW}]\label{lem:iterlem}
 Assume that $\psi(R)$ is a nonnegative and nondecreasing function on $[0,R_0]$, satisfying
 \begin{equation*}
  \psi(\rho)\leq A \left(\frac{\rho}{R}\right)^{\alpha_1}\psi(R)+BR^{\alpha_2},\quad 0<\rho<R\leq R_0,
 \end{equation*}
where $\alpha_1,\alpha_2$ are constants with $0<\alpha_2<\alpha_1$. Then there exists a constant $C$, depending only on $A,\alpha_1$ and $\alpha_2$, such that
\begin{equation*}
 \psi(\rho)\leq C\left(\frac{\rho}{R}\right)^{\alpha_2}\left(\psi(R)+BR^{\alpha_2}\right),\quad 0<\rho<R\leq R_0.
\end{equation*}
\end{lemma}

In the next lemma, we consider the matrix $\diagmatrix_y$ as defined in \eqref{eq:diaga}.

\begin{lemma}\label{lem:diagequiv}
 For every $\Bw\in\bbR^n$,
 \begin{equation*}
  \evmin \lVert \Bw\rVert^2\leq \lVert \diagmatrix_y^{\frac{1}{2}}\Bw\rVert^2\leq \evmax \lVert \Bw\rVert^2,
 \end{equation*}
where $\evmin,\evmax>0$ are, respectively, the $J$- and $\By$-independent lower and upper bounds for the eigenvalues of $\diagmatrix_y$.
\end{lemma}
\begin{proof}
 The proof is trivial. We just remark that $\diagmatrix_y^{\frac{1}{2}}$ is well defined thanks to the assumption that $\evmin,\evmax>0$.
\end{proof}

\begin{notation}
 Given an open domain $\Omega=\Omega^{+}\cup\gammaxif_y\cup\Omega^{-}\subset\bbR^n$ divided into two parts, $\Omega^{+}$ and $\Omega^{-}$, by $\gammaxif_y$, and given a function $h\in L^2(\Omega^{+})\cup L^2(\Omega^{-})$, we introduce the short notation
 \begin{equation*}
\int_{\Omega^{\pm}}h^{\pm}(\Bz)\dz := \int_{\Omega^{+}}h(\Bz)|_{\Omega^{+}}\dz + \int_{\Omega^{-}}h(\Bz)|_{\Omega^{-}}\dz.  
 \end{equation*}
 Also, $(h^{\pm})^2:=\begin{cases}
                           (h|_{\Omega^{+}})^2,&\text{in }\Omega^{+},\\
                           (h|_{\Omega^{-}})^2,&\text{in }\Omega^{-},
                          \end{cases}$, and, if $h\in H^1(\Omega^{+})\cup H^1(\Omega^{-})$, $(\xifgrad h)^{\pm}:=\begin{cases}
                                                                                   \xifgrad h|_{\Omega^{+}}&\text{in }\Omega^{+},\\
                                                                                    \xifgrad h|_{\Omega^{-}}&\text{in }\Omega^{-}.
                                                                                  \end{cases}$
                                                                                  Analogous notations with the symbol $\pm$ as exponent will follow the same rule.
 
 Furthermore, we use the symbol $\partialxifi{i}:=\dfrac{\partialxif}{\partialxif\tilde{x}_{y,i}}$, $i=1,\ldots,n$, to denote partial differentiation, $\tilde{\jacsl}^2$ for the Hessian in the $\xif$-coordinates, and in general, $\tilde{\jacsl}^j$, $j\in\bbN$, to denote the tensor containing all the partial derivatives of order $j$.
\end{notation}

\subsubsection{Cacciopoli's inequalities}
\begin{theorem}\label{thm:cacc}
 Let $\uxif$ be a solution to \eqref{eq:modelpbxi} in $B_R=B_R(\xiffix)$ with $\gxif_y\equiv 0$. Then, for every $0<\rho<R$ and every $\lambda\in\bbR$, it holds:
 \begin{align}
  \hspace{-0.3cm}\int_{B_{\rho}^{\pm}}\left|(\xifgrad\uxif(\Bz))^{\pm}\right|^2\dz&\leq C_1 \left[\frac{1}{(R-\rho)^2}\int_{B_R}(\uxif(\Bz)-\lambda)^2\dz+(R-\rho)^2\int_{B_R^{\pm}}(\fxif_y^{\pm})^2(\Bz)\dz\right],\hspace{-0.3cm}\label{eq:cacc1}\\
  \hspace{-0.3cm}\int_{B_{\rho}^{\pm}}\left|(\xifgrad\wxif(\Bz))^{\pm}\right|^2\dz&\leq C_2\left[\frac{1}{(R-\rho)^2}\int_{B_R}\wxif^2(\Bz)\dz+\int_{B_R^{\pm}}\left(\fxif_y^{\pm}(\Bz)-\fxif_{y,R}^{\pm}\right)^2\dz\right],\label{eq:cacc2}
 \end{align}
where $\wxif:=\partialxifi{i}\uxif$, $i=1,\ldots,n-1$, $B_{\rho}=B_{\rho}(\xiffix)$ and
\begin{equation*}
\fxif^{+}_{y,R}:=\frac{1}{|B_R^{+}|}\int_{B_R^{+}}\fxif_y(\Bz)\dz,\quad  \fxif^{-}_{y,R}:=\frac{1}{|B_R^{-}|}\int_{B_R^{-}}\fxif_y(\Bz)\dz.
\end{equation*}
The constants $C_1=C_1(n,\evmin,\evmax)$, $C_2=C_2(n,\evmin,\evmax)$ are independent of the center $\xiffix$ of $B_R$ and $B_{\rho}$, and, overall, they are independent of $J\in\bbN$ and $\By\in\pspace_J$.
\end{theorem}
\begin{proof}(On the lines of the proof of Thm. 6.2.2 in \cite{WYW}.) We first proof \eqref{eq:cacc1}. Let $\eta\in C^{\infty}_0(B_R)$ be a cut-off function such that:
\begin{equation}\label{eq:etacond}
 0\leq \eta(\xif)\leq 1,\quad \eta(\xif)=1\text{ in }B_{\rho},\quad \max_{\xif\in B_R}\lVert\xifgrad\eta(\xif)\rVert\leq \frac{C_{\eta}}{(R-\rho)\evmax},
\end{equation}
for some $J$-  and $\By$-independent constant $C_{\eta}>0$ (possibly dependent on $n$). Multiplying \eqref{eq:pdexi} by $\eta^2(\uxif-\lambda)$, integrating by parts on $B_R^{+}$ and $B_R^{-}$ and using \eqref{eq:tcxi}, we obtain:
\begin{equation*}
 \int_{B_R^{\pm}}\eta^2\left|\left(\diagmatrix_y^{\frac{1}{2}}\xifgrad\uxif\right)^{\pm}\right|^2\dz=-2\int_{B_R^{\pm}}\eta\left(\diagmatrix_y\xifgrad\uxif\right)^{\pm}\cdot\left(\xifgrad\eta\right)(\uxif-\lambda)\dz+ \int_{B_R^{\pm}}\eta^2(\uxif-\lambda)\fxif_y^{\pm}\dz.
\end{equation*}
From this, applying Cauchy's inequality on the right-hand side with $\varepsilon=\frac{1}{2}$ for the first term and $\varepsilon=\frac{1}{(R-\rho)^2}$ for the second term:
\begin{align*}
 \int_{B_R^{\pm}}\eta^2\left|\left(\diagmatrix_y^{\frac{1}{2}}\xifgrad\uxif\right)^{\pm}\right|^2\dz\leq\; & \frac{1}{2}\int_{B_R^{\pm}}\eta^2\left|\left(\diagmatrix_y^{\frac{1}{2}}\xifgrad\uxif\right)^{\pm}\right|^2\dz+2\int_{B_R^{\pm}}\left|\left(\diagmatrix_y^{\frac{1}{2}}\xifgrad\eta\right)^{\pm}\right|^2 (\uxif-\lambda)^2\dz\\
 & +\frac{(R-\rho)^2}{2}\int_{B_R^{\pm}}\eta^2 \left(\fxif_y^{\pm}\right)^2\dz + \frac{1}{2(R-\rho)^2}\int_{B_R^{\pm}}\eta^2 (\uxif-\lambda)^2\dz.
\end{align*}
Exploiting the properties \eqref{eq:etacond} of $\eta$ in the equation above and Lemma \ref{lem:diagequiv}, we have:
\begin{equation*}
  \int_{B_{\rho}^{\pm}}\left|\left(\diagmatrix_y^{\frac{1}{2}}\xifgrad\uxif\right)^{\pm}\right|^2\dz\leq \left(C_{\eta}^2+\frac{1}{4}\right)\frac{1}{(R-\rho)^2}\int_{B_R}(\uxif(\Bz)-\lambda)^2\dz + \frac{(R-\rho)^2}{4}\int_{B_R^{\pm}}(\fxif_y^{\pm})^2(\Bz)\dz.
\end{equation*}
The estimate \eqref{eq:cacc1} follows then just using Lemma \ref{lem:diagequiv} to have the lower bound $\evmin\left|\left(\xifgrad\uxif\right)^{\pm}\right|^2\leq\left|\left(\diagmatrix_y^{\frac{1}{2}}\xifgrad\uxif\right)^{\pm}\right|^2$ on the integrand at the left-hand side. It is clear then that $C_1$ in \eqref{eq:cacc1} depends only on $n$, $\evmin$ and $\evmax$ and it does not depend on the center $\xiffix$ of $B_R$.

To prove \eqref{eq:cacc2}, we differentiate \eqref{eq:modelpbxi} by $\tilde{x}_{y,i}$, $i=1,\ldots,n-1$. Then $\wxif=\partialxifi{i}\uxif$, $i=1,\ldots,n-1$, satisfies
\begin{subequations}
\begin{align}[left=\empheqlbrace]
 & -\xifgrad \cdot \left(\diagmatrix_y \xifgrad \wxif\right) = \partialxifi{i}\left(\fxif_y-\fxif^{\pm}_{y,R}\right),\quad \text{in }B_R^{+}\cup B_R^{-},\label{eq:poissonwpde}\\
 & \llbracket \wxif \rrbracket_{\gammaxif_y} =0,\quad\Big\llbracket \diagmatrix_y\dfrac{\partialxif \wxif}{\partialxif \nxif}\Big\rrbracket_{\gammaxif_y} = 0.\label{eq:poissonwtc}
\end{align}\label{eq:poissonw}
 \end{subequations}
 Then, multiplying \eqref{eq:poissonwpde} by $\eta^2\wxif$ and integrating by parts:
 \begin{equation*}
\int_{B_R^{\pm}}\eta^2\left|\left(\diagmatrix_y^{\frac{1}{2}}\xifgrad\wxif\right)^{\pm}\right|^2\dz = -2\int_{B_R^{\pm}}\eta\wxif\left(\diagmatrix_y \xifgrad \wxif\right)^{\pm}\cdot\xifgrad\eta\dz + \int_{B_R^{\pm}}\partialxifi{i}\left(\fxif^{\pm}_y-\fxif^{\pm}_{y,R}\right)\eta^2\wxif\dz.
 \end{equation*}
 If we integrate by parts (on $B_R^{+}$ and $B_R^{-}$ separately) the last term on the right-hand side, then $\llbracket \eta^2\wxif\left(\fxif^{\pm}_y-\fxif^{\pm}_{y,R}\right)\tilde{n}_{y,i}\rrbracket_{\gammaxif_y}=0$ for $i=1,\ldots,n-1$, due to the fact that the tangential components (with respect to $\gammaxif_y$) of $\nxif$ are zero. Thus
  \begin{equation*}
\int_{B_R^{\pm}}\eta^2\left|\left(\diagmatrix_y^{\frac{1}{2}}\xifgrad\wxif\right)^{\pm}\right|^2\dz = -2\int_{B_R^{\pm}}\eta\wxif\left(\diagmatrix_y \xifgrad \wxif\right)^{\pm}\cdot\xifgrad\eta\dz - \int_{B_R^{\pm}}\partialxifi{i}(\eta^2 \wxif)\left(\fxif^{\pm}_y-\fxif^{\pm}_{y,R}\right)\dz.
 \end{equation*}
%  Applying Cauchy's inequality on the left-hand side with $\varepsilon=\frac{1}{4}$ for the first summand, a generic $\varepsilon=\bar{\varepsilon}$ for the second summand and $\varepsilon=\frac{1}{2}$ for the third summand, and then using Lemma \ref{lem:diagequiv}, we can choose $\bar{\varepsilon}=\bar{\varepsilon}(\evmin)$ sufficiently small and exploit the properties \eqref{eq:etacond} of $\eta$ to obtain \eqref{eq:cacc2}. 
Splitting the derivative in the last integral, and using Cauchy's inequality on each term with $\varepsilon=\varepsilon(\evmin)$ sufficiently small, the properties \eqref{eq:etacond} of $\eta$ lead to \eqref{eq:cacc2}. As for the constant in \eqref{eq:cacc1}, it is clear that $C_2=C_2(n,\evmin,\evmax)$ in  \eqref{eq:cacc2} is independent of the center $\xiffix$ of $B_R$ and in general of $J\in\bbN$ and $\By\in\pspace_J$.
\end{proof}
\begin{remark}[Analogous to Rmk. 6.2.5 in \cite{WYW}]\label{rmk:caccn} Since \eqref{eq:pdexi} can be rewritten as $\left(\diagmatrix_y\right)_{nn}\partialxifi{nn}^2\uxif=-\sum_{j=1}^{n-1}\left(\diagmatrix_y\right)_{jj}\partialxifi{jj}^2\uxif-\fxif_y$, we obtain, using \eqref{eq:cacc2}:
\begin{align}
 \int_{B_{\rho}^{\pm}}\left|\left(\tilde{\jacsl}^2 \uxif\right)^{\pm}(\Bz)\right|^2\dz\leq\; &C\left(\sum_{j=1}^{n-1}\int_{B_{\rho}^{\pm}}\left|\left(\xifgrad\left(\partialxifi{j}\uxif\right)\right)^{\pm}\right|^2\dz +\int_{B_{\rho}^{\pm}}\left(\fxif_y^{\pm}\right)^2\dz\right)\nonumber\\
 \leq\; &C'\left[\frac{1}{(R-\rho)^2}\int_{B_R}\left|\left(\xifgrad\uxif\right)^{\pm}\right|^2(\Bz)\dz\right.\nonumber\\
 &\left.+\int_{B_R^{\pm}}\left(\fxif_y^{\pm}(\Bz)-\fxif_{y,R}^{\pm}\right)^2\dz+\int_{B_{R}^{\pm}}\left(\fxif_y^{\pm}\right)^2(\Bz)\dz\right],\label{eq:caccrmk}
\end{align}
where, thanks to Lemma \ref{lem:diagequiv} and Theorem \ref{thm:cacc}, $C$ and $C'$ depend only on $n$, $\evmin$ and $\evmax$.
\end{remark}

If we apply \eqref{eq:caccrmk} with $\rho=\frac{R}{2}$ and $\frac{3}{4}R$, and \eqref{eq:cacc1} with $\rho=\frac{3}{4}R$ and $R$ (and $\lambda=0$), we obtain \cite[Cor. 6.2.4]{WYW}:
\begin{corollary}\label{cor:d2bound}
  Let $\uxif$ be a solution to \eqref{eq:modelpbxi} in $B_R=B_R(\xiffix)$ with $\gxif_y\equiv 0$. Then it holds:
  \begin{equation}
  \int_{B_{R/2}^{\pm}} \left|\left(\tilde{\jacsl}^2 \uxif\right)^{\pm}(\Bz)\right|^2\dz\leq C \left[\frac{1}{R^4}\int_{B_R}\uxif^2 (\Bz)\dz + R^n \contpwbr{\fxif_y}{R}^2+ R^{n+2\hexp}\sholderpwbr{\fxif_y}{R}^2\right],
  \end{equation}
where $C=C(n,\evmin,\evmax)$ is independent of the center $\xiffix$ of $B_R$ and, overall, is independent of $J\in\bbN$ and $\By\in\pspace_J$.
  \end{corollary}
  
  \begin{corollary}\label{cor:hk}
 Let $\uxif$ be a solution to \eqref{eq:modelpbxi} in $B_1=B_1(\xiffix)$. If $\fxif_y\equiv 0$ and $\gxif_y\equiv 0$, then, for any positive integer $k\in\bbN$:
 \begin{equation}
  \lVert\uxif\rVert_{H^k(B_{1/2}^{+})} + \lVert\uxif\rVert_{H^k(B_{1/2}^{-})}\leq C\left(  \lVert\uxif\rVert_{L^2(B_{1/2}^{+})}+  \lVert\uxif\rVert_{L^2(B_{1/2}^{-})}\right),
 \end{equation}
with $C=C(n,\evmin,\evmax,k)$ independent of the center $\xiffix$ of $B_1$ and overall of $J\in\bbN$ and $\By\in\pspace_J$.
\end{corollary}
\begin{proof}(On the lines of the proof of Corollary 6.2.5 in \cite{WYW}.) For $k=1$, the claim follows directly from Cacciopoli's inequality \eqref{eq:cacc1} with $\lambda=0$.  For $k=2$, it follows from \eqref{eq:caccrmk}. For $k>2$, we can proceed analogously to the proof of \eqref{eq:cacc2} and  Remark \ref{rmk:caccn}.
\end{proof}

\begin{corollary}\label{cor:supbound}
  Let $\uxif$ be a solution to \eqref{eq:modelpbxi} in $B_1=B_1(\xiffix)$. If $\fxif_y\equiv 0$ in and $\gxif_y\equiv 0$, then
  \begin{equation}
\sup_{B_{R/2}^{+}\cup B_{R/2}^{-}}\left|\uxif\right|\leq C\left(\frac{1}{R^n}\int_{B_R^{+}\cup B_R^{-}}\uxif^2(\Bz)\dz\right)^{\frac{1}{2}},
  \end{equation}
with $C=C(n,\evmin,\evmax)$ independent of the center $\xiffix$ of $B_R$ and overall of $J\in\bbN$ and $\By\in\pspace_J$.
\end{corollary}
\begin{proof}(Analogous to proof of Corollary 6.2.6 in \cite{WYW}) We first establish the estimate assuming $R=1$. The Sobolev embedding theorem \cite[Thm. 7.26]{GT} applied in $B_{1/2}^{+}$ and in $B_{1/2}^{-}$ implies that, for $k>\frac{n}{2}$,
 \begin{align*}
 \sup_{B_{1/2}^{+}\cup B_{1/2}^{-}}\left|\uxif\right|\leq C \left(  \lVert\uxif\rVert_{H^k(B_{1/2}^{+})} + \lVert\uxif\rVert_{H^k(B_{1/2}^{-})}\right),
 \end{align*}
where the constant $C$ depends on $n$ only. The claim for $R=1$ follows then from Corollary \ref{cor:hk}.% (and using that, for $a,b,\in\bbR$, $a+b\leq \sqrt{2}(a^2+b^2)^{\frac{1}{2}}$).

For a generic radius $R$, the result follows from a scaling argument, defining $v(\Bz):=\uxif(R\Bz)$, $\Bz\in B_1$, and applying the estimate for $R=1$ to the function $v$.
\end{proof}

\subsubsection{Interface estimate for the Laplace equation}
\begin{theorem}\label{thm:laplace}
   Let $\uxif$ be a solution to \eqref{eq:modelpbxi}, with $\fxif_y\equiv 0$ and $\gxif_y\equiv 0$. Then, for every $0<\rho\leq R$ and every $i\in\bbN$:
   \begin{equation}\label{eq:laplacebound}
    \int_{B_{\rho}^{\pm}}\left|\left(\tilde{\jacsl}^i\uxif\right)^{\pm}(\Bz)\right|^2\dz\leq C\left(\frac{\rho}{R}\right)^n \int_{B_R^{\pm}}\left|\left(\tilde{\jacsl}^i\uxif\right)^{\pm}(\Bz)\right|^2\dz,
   \end{equation}
   where $C=C(n,\evmin,\evmax)$ is independent of the center $\xiffix$ of $B_R$ and overall of $J\in\bbN$ and $\By\in\pspace_J$.
\end{theorem}
\begin{proof}(On the lines of the proof of Thm. 6.2.4 in \cite{WYW}.)

\textsl{Case $i=0$.}

If $0<\rho<\frac{R}{2}$, then
\begin{equation*}
 \int_{B_{\rho}}(\uxif^{\pm})^2\dz \leq |B_{\rho}|\sup_{B_{\rho}^{+}\cup B_{\rho}^{-}} \uxif^2 = C_n \rho^n \sup_{B_{\rho}^{+}\cup B_{\rho}^{-}} \uxif^2\leq C\left(\frac{\rho}{R}\right)^n  \int_{B_{\rho}}(\uxif^{\pm})^2\dz,
\end{equation*}
where the last constant is given by the product of $C_n$ with the constant from Corollary \ref{cor:supbound}, and $C_n$ is a constant depending on $n$ only.

If $\frac{R}{2}\leq\rho\leq R$, then, trivially,
\begin{equation*}
  \int_{B_{\rho}}(\uxif^{\pm})^2\dz\leq  \int_{B_R}(\uxif^{\pm})^2\dz\leq 2^n \left(\frac{\rho}{R}\right)^n  \int_{B_R}(\uxif^{\pm})^2\dz.
\end{equation*}

\textsl{Case $i=1$.}

We first consider $0<\rho<\frac{R}{2}$. For $k-1>\frac{n}{2}$, the Sobolev embedding theorem \cite[Thm. 7.26]{GT} ensures that $\lVert \xifgrad\uxif\rVert_{L^{\infty}(B_{R/2}^{+})}\leq C_k \lVert \xifgrad\uxif\rVert_{H^{k-1}(B_{R/2}^{+})}$, for a constant $C_k$ dependent on $n$ only (and analogously in $B_{R/2}^{-}$). Exploiting this fact, we have:
\begin{align*}
 \int_{B_{\rho}^{\pm}}\left|\left(\xifgrad\uxif\right)^{\pm}\right|^2\dz &\leq C_n\rho^n \sup_{B_{R/2}^{+}\cup B_{R/2}^{-}} \left|\xifgrad\uxif\right|^2\\
 &\leq C_n C_k\rho^n\sum_{l=0}^{k-1}R^{2l-n}\int_{B_{R/2}^{\pm}}\left|\left(\tilde{\jac}^l(\xifgrad\uxif)\right)^{\pm}(\Bz)\right|^2\dz\\
 &= C_n C_k\rho^n\sum_{j=1}^{k}R^{2(j-1)-n}\int_{B_{R/2}^{\pm}}\left|\left(\tilde{\jac}^j\uxif\right)^{\pm}(\Bz)\right|^2\dz,
\end{align*}
where $C_n$ is a constant depending only on $n$. The factors $R^{2l-n}$, $l=0,\ldots,k-1$, in the second inequality are due to a scaling argument as in Corollary \ref{cor:supbound}. We note that $k$ depends on $n$ only. Denoting $\tilde{u}_{y,R}:=\frac{1}{|B_R|}\int_{B_R}\uxif(\Bz)\dz$, and observing that $(\uxif-\tilde{u}_{y,R})$ fulfills \eqref{eq:modelpbxi} with $\fxif_y\equiv 0$ and $\gxif_y\equiv 0$, we can apply Corollary \ref{cor:hk} and derive:
\begin{align*}
  \int_{B_{\rho}^{\pm}}\left|\left(\xifgrad\uxif\right)^{\pm}\right|^2\dz &\leq C\rho^n\sum_{j=1}^k R^{2(j-1)-n}R^{-2j}\int_{B_R^{\pm}}\left|\uxif^{\pm}-\tilde{u}_{y,R}\right|^2\dz\\
  &= k C\left(\frac{\rho}{R}\right)^n R^{-2}\int_{B_R^{\pm}}\left|\uxif^{\pm}-\tilde{u}_{y,R}\right|^2\dz\\
  &\leq k C C_p\left(\frac{\rho}{R}\right)^n R^{-2}R^2 \int_{B_R^{\pm}}\left|\left(\xifgrad\uxif\right)^{\pm}\right|^2\dz.
\end{align*}
The constant $C$ in the previous inequalities is the product of $C_nC_k$ by the constant of Corollary \ref{cor:hk}. The constant $C_p$ is instead a scalar factor, independent of $R$, coming from application of the Poincar\'e inequality for balls (that we could apply being $\uxif$ in $H^1(B_R)$).

For $\frac{R}{2}\leq\rho\leq R$, the inequality \eqref{eq:laplacebound} follows trivially taking $C\geq 2^n$.

\textsl{Case $i=2$.}

For $j=1,\ldots,n-1$, $\llbracket \partialxifi{j}\uxif\rrbracket_{\gammaxif_y}=0$ and $\Big\llbracket \diagmatrix_y \dfrac{\partialxif ( \partialxifi{j}\uxif)}{\partialxif\nxif}\Big\rrbracket_{\gammaxif_y}=0$. Thus, the case for $i=1$ implies that
\begin{equation}\label{eq:dnjestimate}
 \int_{B_{\rho}^{\pm}}\left|\left(\xifgrad (\partialxifi{j}\uxif)\right)^{\pm}\right|^2\dz\leq C\left(\frac{\rho}{R}\right)^n  \int_{B_R^{\pm}}\left|\left(\xifgrad (\partialxifi{j}\uxif)\right)^{\pm}\right|^2\dz,
\end{equation}
with $C=C(n,\evmin,\evmax)$. For $j=n$, we can use that $\left(\diagmatrix_y\right)_{nn}\partialxifi{nn}^2\uxif=-\sum_{j=1}^{n-1}\left(\diagmatrix_y\right)_{jj}\partialxifi{jj}^2\uxif$ to obtain:
\begin{align*}
 \int_{B_{\rho}^{\pm}}\left|\left(\partialxifi{nn}^2\uxif\right)^{\pm}\right|^2\dz&\leq C_n \frac{\evmax^2}{\evmin^2}\sum_{j=1}^{n-1} \int_{B_{\rho}^{\pm}}\left|\left(\partialxifi{jj}^2\uxif\right)^{\pm}\right|^2\dz\\
 &\leq C C_n \frac{\evmax^2}{\evmin^2} \left(\frac{\rho}{R}\right)^n\sum_{j=1}^{n-1} \int_{B_R^{\pm}}\left|\left(\xifgrad (\partialxifi{j}\uxif)\right)^{\pm}\right|^2\dz\\
 &\leq C C_n \frac{\evmax^2}{\evmin^2} \left(\frac{\rho}{R}\right)^n \int_{B_R^{\pm}}\left|\left(\tilde{\jac}^2\uxif\right)^{\pm}(\Bz)\right|^2\dz,
\end{align*}
where the constant $C_n$ depends on $n$ only, and the constant $C=C(n,\evmin,\evmax)$ is the constant in \eqref{eq:dnjestimate}. The latter inequality together with \eqref{eq:dnjestimate} imply finally \eqref{eq:laplacebound}.

\textsl{Case $i=3$.}

The result follows similarly as for the case $i=2$: the estimates associated to $\tilde{\jac}^2(\partialxifi{j}\uxif)$, for $j=1,\ldots,n-1$, follow from the case $i=2$; for $\partialxifi{nnn}^3\uxif$, we observe that, by differentiation, \eqref{eq:pdexi} implies that $\left(\diagmatrix_y\right)_{nn}\partialxifi{nnn}^3\uxif=-\sum_{j=1}^{n-1}\left(\diagmatrix_y\right)_{jj}\partialxifi{jjn}^3\uxif$, and we can proceed as we did in the case $i=2$ for $\partialxifi{nn}^2\uxif$.

The case $i>3$ can be proved analogously to the case $i=3$.
\end{proof}

\begin{theorem}\label{thm:laplace2}
 Let $\uxif$ be a solution to \eqref{eq:modelpbxi}, with $\fxif_y\equiv 0$ and $\gxif_y\equiv 0$. Then, for every $0<\rho\leq R$:
 \begin{equation}\label{eq:n2bound}
\int_{B_{\rho}^{\pm}}\left(\uxif^{\pm}(\Bz)-\tilde{u}_{y,\rho}^{\pm}\right)^2\dz\leq C\left(\frac{\rho}{R}\right)^{n+2}\int_{B_R}\uxif^2 (\Bz)\dz,
 \end{equation}
 where $\tilde{u}_{y,\rho}^{+}=\frac{1}{|B_{\rho}^{+}|}\int_{B_{\rho}^{+}}\uxif(\Bz)\dz$ and $\tilde{u}_{y,\rho}^{-}=\frac{1}{|B_{\rho}^{-}|}\int_{B_{\rho}^{-}}\uxif(\Bz)\dz$. The constant $C=C(n,\evmin,\evmax)$ is independent of the center $\xiffix$ of $B_R$ and, overall, of $J\in\bbN$ and $\By\in\pspace_J$.
\end{theorem}
\begin{proof}(On the lines of proof of Thm. 6.2.5 in \cite{WYW}.) If $0<\rho<\frac{R}{2}$:
 \begin{align*}
  \int_{B_{\rho}^{\pm}}\left(\uxif^{\pm}-\tilde{u}_{y,\rho}^{\pm}\right)^2\dz&\leq C_p \rho^2\int_{B_{\rho}^{\pm}}\left|\left(\xifgrad\uxif\right)^{\pm}\right|^2\dz\\
  &\leq  C C_p\rho^2\left(\frac{\rho}{R}\right)^n\int_{B_{R/2}^{\pm}}\left|\left(\xifgrad\uxif\right)^{\pm}\right|^2\dz\\
  &\leq C'\rho^2\left(\frac{\rho}{R}\right)^n\frac{4}{R^2}\int_{B_{R}^{\pm}}\left|\left(\xifgrad\uxif\right)^{\pm}\right|^2\dz.
 \end{align*}
In the first step we have applied the Poincar\'e inequality in $B_{\rho}^{+}$ and $B_{\rho}^{-}$ separately, and denoted by $C_p$ the $\rho$-independent scalar factor in the Poincar\'e constant. For the second inequality, we have used Theorem \ref{thm:laplace}, and the constant $C$ corresponds to the constant in \eqref{eq:laplacebound}. Finally, the last line follows from Cacciopoli's inequality \eqref{eq:cacc1}, and we have denoted by $C'$ the multiplication of $CC_p$ with the constant in \eqref{eq:cacc1}.

If $\frac{R}{2}\leq\rho\leq R$, \eqref{eq:n2bound} follows simply taking $C\geq 2^{n+2}$.
\end{proof}

\subsubsection{Interface estimate for the Poisson equation}

\begin{theorem}\label{thm:bdpoisson1}
  Let $\uxif$ be a solution to \eqref{eq:modelpbxi} in $B_{R_0}$, with $\gxif_y\equiv 0$, and let $\wxif=\partialxifi{i}\uxif$, $i=1,\ldots,n-1$. Then, for any $0<\rho\leq R\leq R_0$:
  \begin{align}
   \frac{1}{\rho^{n+2\hexp}}\int_{B_{\rho}^{\pm}}\sum_{j=1}^n \left|\left(\partialxifi{j}\wxif\right)^{\pm}(\Bz)-\left(\partialxifi{j}\wxif\right)^{\pm}_{\rho}\right|^2\dz\leq \;&\frac{C_1}{R^{n+2\hexp}}\int_{B_R^{\pm}}\sum_{j=1}^{n-1}\left|\left(\partialxifi{j}\wxif\right)^{\pm}(\Bz)\right|^2 \dz\hspace{-1cm}\nonumber\\
   &+ \frac{C_1}{R^{n+2\hexp}}\int_{B_R^{\pm}}\left|\left(\partialxifi{n}\wxif\right)^{\pm}(\Bz)-\left(\partialxifi{n}\wxif\right)^{\pm}_{R}\right|^2\dz\nonumber\\
   &+C_2\sholderpwbr{\fxif_y}{R}^2,\nonumber
  \end{align}
where $\left(\partialxifi{j}\wxif\right)^{+}_{\rho}:=\frac{1}{|B_{\rho}^{+}|}\int_{B_{\rho}^{+}}\partialxifi{j}\wxif\dz$ and $\left(\partialxifi{j}\wxif\right)^{-}_{\rho}:=\frac{1}{|B_{\rho}^{-}|}\int_{B_{\rho}^{-}}\partialxifi{j}\wxif\dz$, and similarly for $\left(\partialxifi{j}\wxif\right)^{\pm}_R$. The constants $C_1=C_1(n,\evmin,\evmax,\hexp)$ and $C_2=C_2(n,\evmin,\evmax)$ are independent of the center $\xiffix$ of $B_R$ and, overall, they are independent of $J\in\bbN$ and $\By\in\pspace_J$.
\end{theorem}
\begin{proof}(On the lines of the proof of Thm. 6.2.9 in \cite{WYW}.)  We decompose $\wxif$ as $\wxif=\wxif'+\wxif''$, where
\begin{equation*}
\hspace{-0.15cm}\left\{ \hspace{-0.35cm}\begin{array}{ll}
 & -\xifgrad \cdot \left(\diagmatrix_y \xifgrad \wxif'\right) = 0,\; \text{in }B_R^{+}\cup B_R^{-},\\
 & \llbracket \wxif' \rrbracket_{\gammaxif_y} =0,\quad\Big\llbracket \diagmatrix_y\dfrac{\partialxif \wxif'}{\partialxif \nxif}\Big\rrbracket_{\gammaxif_y} = 0,\\
 & \wxif'|_{\partial B_R}=\wxif,
 \end{array}\right.\;\left\{ \hspace{-0.35cm}\begin{array}{ll}
 & -\xifgrad \cdot \left(\diagmatrix_y \xifgrad \wxif''\right) = \partialxifi{i}\left(\fxif_y-\fxif_{y,R}^{\pm}\right),\;\text{in }B_R^{+}\cup B_R^{-},\\
 & \llbracket \wxif'' \rrbracket_{\gammaxif_y} =0,\quad\Big\llbracket \diagmatrix_y\dfrac{\partialxif \wxif''}{\partialxif \nxif}\Big\rrbracket_{\gammaxif_y} = 0,\\
 & \wxif''|_{\partial B_R}=0,
 \end{array}\right.
\end{equation*}
and $\fxif_{y,R}^{\pm}=\begin{cases}
                        \fxif_y^{+}=\frac{1}{|B_R^{+}|}\int_{B_R^{+}}\fxif_y(\Bz)\dz\;\text{in }B_R^{+},\\
                        \fxif_y^{-}=\frac{1}{|B_R^{-}|}\int_{B_R^{-}}\fxif_y(\Bz)\dz\;\text{in }B_R^{-}.
                       \end{cases}$  
  
  \medskip
We first consider $j=1,\ldots,n-1$. In this case, $\partialxifi{j}\wxif'$ solves \eqref{eq:modelpbxi} with $\fxif_y\equiv 0$ and $\gxif_y\equiv 0$, and Theorem \ref{thm:laplace2} gives:
\begin{equation}\label{eq:bdfromlapl}
 \int_{B_{\rho}^{\pm}}\left|\left(\partialxifi{j}\wxif'\right)^{\pm}-\left(\partialxifi{j}\wxif'\right)^{\pm}_{\rho}\right|^2\dz\leq C\left(\frac{\rho}{R}\right)^{n+2}\int_{B_R}\left|\left(\partialxif{j}\wxif'\right)\right|^2\dz,
\end{equation}
where $C=C(n,\evmin,\evmax)$ is the constant in \eqref{eq:n2bound}. Then, for $\partialxifi{j}\wxif$, using Lemma \ref{lem:convex} we can write:
\begin{align}
  \int_{B_{\rho}^{\pm}}\left|\left(\partialxifi{j}\wxif\right)^{\pm}-\left(\partialxifi{j}\wxif\right)^{\pm}_{\rho}\right|^2\dz&\leq 2 \int_{B_{\rho}^{\pm}}\left|\left(\partialxifi{j}\wxif'\right)^{\pm}\right|^2\dz +2 \int_{B_{\rho}^{\pm}}\left|\left(\partialxifi{j}\wxif''\right)^{\pm}\right|^2\dz\label{eq:bd2}\\
  &\leq C'\left(\frac{\rho}{R}\right)^{n+2}\int_{B_R}\left|\left(\partialxifi{j}\wxif\right)\right|^2\dz + C''\int_{B_R^{\pm}}\left|\left(\partialxifi{j}\wxif''\right)^{\pm}\right|^2\dz,\nonumber
\end{align}

\vspace{-1.2cm}
\noindent for $C'$, $C''$ independent of $\xiffix$, $J$ and $\By$. The last summand in the above inequality can be bounded as 
\begin{align*}
 \int_{B_R^{\pm}}\left|\left(\partialxifi{j}\wxif''\right)^{\pm}\right|^2\dz&\leq  \int_{B_R^{\pm}}\left|\left(\xifgrad\wxif''\right)^{\pm}\right|^2\dz\\
 &\leq \frac{1}{\evmin}\int_{B_R^{\pm}}\left|\left(\diagmatrix_y^{\frac{1}{2}}\xifgrad \wxif''\right)^{\pm}\right|^2\dz\\
 &=-\frac{1}{\evmin}\int_{B_R^{\pm}}\left(\xifgrad\cdot\left(\diagmatrix_y\xifgrad\wxif''\right)\wxif''\right)^{\pm}\dz\\
 &=\frac{1}{\evmin}\int_{B_R^{\pm}}\partialxifi{i}\left(\fxif_y^{\pm}-\fxif_{y,R}^{\pm}\right)\left(\wxif''\right)^{\pm}\dz\\
 &=-\frac{1}{\evmin}\int_{B_R^{\pm}}\left(\fxif_y^{\pm}-\fxif_{y,R}^{\pm}\right)\left(\partialxifi{i}\wxif''\right)^{\pm}\dz\\
 &\leq\frac{1}{\evmin}\left(\frac{1}{2\varepsilon}\int_{B_R^{\pm}}\left|\fxif_y^{\pm}-\fxif_{y,R}^{\pm}\right|^2\dz+\frac{\varepsilon}{2}\int_{B_R^{\pm}}\left|\left(\xifgrad\wxif''\right)^{\pm}\right|^2 \dz\right),
\end{align*}
where for the third and fifth line we have used integration by parts, and the jump terms on $\gammaxif$ vanished because of the transmission conditions in the first case, and because of null tangential components of $\nxif$ in the second case; the boundary terms vanished because of the Dirichlet boundary conditions. In the last step, we have applied Cauchy's inequality for a generic $\varepsilon>0$. Choosing $\varepsilon=\varepsilon(\evmin)$ sufficiently small, we finally obtain
\begin{equation}\label{eq:bdgradw2}
  \int_{B_R^{\pm}}\left|\left(\partialxifi{j}\wxif''\right)^{\pm}\right|^2\dz\leq  \int_{B_R^{\pm}}\left|\left(\xifgrad\wxif''\right)^{\pm}\right|^2\dz\leq C'''R^{n+2\hexp}\sholderpwbr{\fxif_y}{R}^2,
\end{equation}
for a positive constant $C'''=C'''(\evmin)$. Combining the last estimate with \eqref{eq:bd2}, we infer 
\begin{align}\label{eq:bdtg}
\int_{B_{\rho}^{\pm}}\left|\left(\partialxifi{j}\wxif\right)^{\pm}(\Bz)-\left(\partialxifi{j}\wxif\right)^{\pm}_{\rho}\right|^2\dz\leq\;& C'\left(\frac{\rho}{R}\right)^{n+2}\int_{B_R^{\pm}}\left|\left(\partialxifi{j}\wxif\right)^{\pm}(\Bz)\right|^2 \dz \\
&+(C''C''') R^{n+2\hexp}\sholderpwbr{\fxif_y}{R}^2,
\end{align}
for $j=1,\ldots,n-1$.

We now consider $j=n$ and $0<\rho<\frac{R}{2}$. Applying the Poincar\'e inequality, Lemma \ref{lem:convex}, Theorem \ref{thm:laplace} and equation \eqref{eq:bdgradw2} (which holds for $j=n$, too):
\begin{align*}
  \int_{B_{\rho}^{\pm}}\left|\left(\partialxifi{n}\wxif\right)^{\pm}-\left(\partialxifi{n}\wxif\right)^{\pm}_{\rho}\right|^2\dz&\leq 2 \int_{B_{\rho}^{\pm}}\left|\left(\partialxifi{n}\wxif'\right)^{\pm}-\left(\partialxifi{n}\wxif'\right)^{\pm}_{\rho}\right|^2\dz\\
  &\;\;\;\;+2 \int_{B_{\rho}^{\pm}}\left|\left(\partialxifi{n}\wxif''\right)^{\pm}-\left(\partialxifi{n}\wxif''\right)^{\pm}_{\rho}\right|^2\dz\\
  &\leq C_p \rho^2\int_{B_{\rho}^{\pm}}\left|\left(\xifgrad\left(\partialxifi{n}\wxif'\right)\right)^{\pm}\right|^2\dz+2 \int_{B_{\rho}^{\pm}}\left|\left(\partialxifi{n}\wxif''\right)^{\pm}\right|^2\dz\\
  &\leq C_p C\rho^2\left(\frac{\rho}{R}\right)^n\int_{B_{R/2}^{\pm}}\left|\left(\tilde{\jac}^2 \wxif'\right)^{\pm}\right|^2\dz + C'''R^{n+2\hexp}\sholderpwbr{\fxif_y}{R}^2,
\end{align*}
with $C_p$ a scalar factor, independent of $R$, coming from application of the Poincar\'e inequality, and $C$ is the constant in Theorem \ref{thm:laplace}. We can bound the integral on the right-hand side by:
\begin{align*}
 \int_{B_{R/2}^{\pm}}\left|\left(\tilde{\jac}^2 \wxif'\right)^{\pm}\right|^2\dz&\leq  \int_{B_{R/2}^{\pm}}\left|\left(\partialxifi{nn}^2 \wxif'\right)^{\pm}\right|^2\dz + 2 \sum_{j=1}^{n-1}\int_{B_{R/2}^{\pm}}\left|\left(\xifgrad (\partialxifi{j}\wxif')\right)^{\pm}\right|^2\dz\\
 &\leq C\sum_{j=1}^{n-1}\int_{B_{R/2}^{\pm}}\left|\left(\xifgrad (\partialxifi{j}\wxif')\right)^{\pm}\right|^2\dz\\
 &\leq C C_c \frac{1}{R^2}\sum_{j=1}^{n-1}\int_{B_{R}^{\pm}}\left|\left(\partialxifi{j}\wxif'\right)^{\pm}\right|^2\dz\\
 &\leq C C_c \frac{1}{R^2}\left(\sum_{j=1}^{n-1}\int_{B_{R}^{\pm}}\left|\left(\partialxifi{j}\wxif\right)^{\pm}\right|^2\dz+\sum_{j=1}^{n-1}\int_{B_{R}^{\pm}}\left|\left(\partialxifi{j}\wxif''\right)^{\pm}\right|^2\dz\right)\\
 &\leq C C_c \frac{1}{R^2}\left(\sum_{j=1}^{n-1}\int_{B_{R}^{\pm}}\left|\left(\partialxifi{j}\wxif\right)^{\pm}\right|^2\dz+(n-1)C'''R^{n+2\hexp}\sholderpwbr{\fxif_y}{R}^2\right).
\end{align*}
In the second line we have used the equality $\left(\diagmatrix_y\right)_{nn}\partialxifi{nn}^2\wxif'=-\sum_{j=1}^{n-1}\left(\diagmatrix_y\right)_{jj}\partialxifi{jj}^2\wxif'$ and Lemma \ref{lem:diagequiv}, and thus $C=C(n,\evmin,\evmax)$ is $J$- and $\By$-independent. In the third line we have exploited the Cacciopoli inequality \eqref{eq:cacc1}, and, in the last step, the bound \eqref{eq:bdgradw2}. Summarizing, for $j=n$ and $0<\rho<\frac{R}{2}$ we have
\begin{equation*}
   \int_{B_{\rho}^{\pm}}\left|\left(\partialxifi{n}\wxif\right)^{\pm}-\left(\partialxifi{n}\wxif\right)^{\pm}_{\rho}\right|^2\dz\leq C_1' \left(\frac{\rho}{R}\right)^{n+2}\sum_{j=1}^{n-1}\int_{B_{R}^{\pm}}\left|\left(\partialxifi{j}\wxif\right)^{\pm}\right|^2\dz + C_2' R^{n+2\hexp}\sholderpwbr{\fxif_y}{R}^2,
\end{equation*}
for two constants $C_1'=C_1'(n,\evmin,\evmax)$ and $C_2'=C_2'(n,\evmin,\evmax)$ independent of the center of $B_R$, of $J\in\bbN$ and of $\By\in\pspace_J$.

Combining the latter estimate with the estimate \eqref{eq:bdtg}, we finally obtain:
  \begin{align*}
   \int_{B_{\rho}^{\pm}}\sum_{j=1}^n \left|\left(\partialxifi{j}\wxif\right)^{\pm}(\Bz)-\left(\partialxifi{j}\wxif\right)^{\pm}_{\rho}\right|^2\dz\leq \;& C_1\left(\frac{\rho}{R}\right)^{n+2}\int_{B_R^{\pm}}\sum_{j=1}^{n-1}\left|\left(\partialxifi{j}\wxif\right)^{\pm}(\Bz)\right|^2 \dz\\
   &+ C_1\left(\frac{\rho}{R}\right)^{n+2}\int_{B_R^{\pm}}\left|\left(\partialxifi{n}\wxif\right)^{\pm}(\Bz)-\left(\partialxifi{n}\wxif\right)^{\pm}_{R}\right|^2\dz\\
   &+C_2 R^{n+2\hexp}\sholderpwbr{\fxif_y}{R}^2,
  \end{align*}
for two positive constants $C_1=C_1(n,\evmin,\evmax)$ and $C_2=C_2(n,\evmin,\evmax)$. The claim for $0<\rho<\frac{R}{2}$ follows then by application of Lemma \ref{lem:convex} and the Iteration Lemma \ref{lem:iterlem}.

\smallskip

If instead $\frac{R}{2}\leq \rho\leq R$, the claim holds simply by choosing $C_1\geq 2^{n+2}$ and using Lemma \ref{lem:convex}.
\end{proof}

\begin{theorem}\label{thm:bdpoisson2}
   Let $\uxif$ be a solution to \eqref{eq:modelpbxi} in $B_{R}$, with $\gxif_y\equiv 0$, and let $\wxif=\partialxifi{i}\uxif$, $i=1,\ldots,n$. Then, for any $0<\rho\leq \frac{R}{2}$:
   \begin{equation}
    \int_{B_{\rho}^{\pm}}\left|\left(\xifgrad\wxif\right)^{\pm}(\Bz)-\left(\xifgrad \wxif\right)^{\pm}_{\rho}\right|^2\dz\leq C\rho^{2+n\hexp}M_R,
   \end{equation}
with
\begin{equation}
 M_R=\frac{1}{R^{4+2\hexp}}\cont{\uxif}{B_R}^2 + \frac{1}{R^{2\hexp}}\contpwbr{\fxif_y}{R}^2 +  \sholderpwbr{\fxif_y}{R}^2,
\end{equation}
and $C=C(n,\evmin,\evmax,\hexp)$ independent of the center $\xiffix$ of $B_R$ and $B_{\rho}$, and, overall, of $J\in\bbN$ and $\By\in\pspace_J$. The term $\left(\xifgrad \wxif\right)^{\pm}_{\rho}$ has the same meaning as in Theorem \ref{thm:bdpoisson1}.
\end{theorem}
\begin{proof} (On the lines of the proof of Thm. 6.2.10 in \cite{WYW}.) 
 For $i=1,\ldots,n-1$, we can apply Theorem \ref{thm:bdpoisson1}, which, together with Lemma \ref{lem:convex} and Corollary \ref{cor:d2bound}, brings:
 \begin{align*}
  \int_{B_{\rho}^{\pm}}\left|\left(\xifgrad\wxif\right)^{\pm}(\Bz)-\left(\xifgrad \wxif\right)^{\pm}_{\rho}\right|^2\dz&= \int_{B_{\rho}^{\pm}}\sum_{j=1}^{n}\left|\left(\partialxifi{j}\wxif\right)^{\pm}(\Bz)-\left(\partialxifi{j} \wxif\right)^{\pm}_{\rho}\right|^2\dz\\
  &\leq C'\rho^{n+2\hexp}\left(\frac{1}{R^{n+2\hexp}}\int_{B_{R/2}^{\pm}}\sum_{j=1}^{n-1}\left|\left(\partialxifi{j}\wxif\right)^{\pm}(\Bz)\right|^2 \dz\right.\\
   &\;\;\;\;+ \frac{1}{R^{n+2\hexp}}\int_{B_R^{\pm}}\left|\left(\partialxifi{n}\wxif\right)^{\pm}(\Bz)-\left(\partialxifi{n}\wxif\right)^{\pm}_{R}\right|^2\dz\\
   &\;\;\;\;\left.+\sholderpwbr{\fxif_y}{R}^2\right)\\
   &\leq C'\rho^{n+2\hexp}\left(\frac{1}{R^{n+2\hexp}}\int_{B_{R/2}^{\pm}}\sum_{j=1}^{n}\left|\left(\partialxifi{j}\wxif\right)^{\pm}(\Bz)\right|^2 \dz+\sholderpwbr{\fxif_y}{R}^2\right)\\
   &\leq C''\rho^{n+2\hexp}\left(\frac{1}{R^{n+2\hexp+4}}\int_{B_R}\uxif^2\dz+\frac{1}{R^{2\hexp}}\contpwbr{\fxif_y}{R}^2+\sholderpwbr{\fxif_y}{R}^2\right)\\
   &\leq C''\rho^{n+2\hexp}\left(\frac{1}{R^{2\hexp+4}}\cont{\uxif}{R}^2+\frac{1}{R^{2\hexp}}\contpwbr{\fxif_y}{R}^2+\sholderpwbr{\fxif_y}{R}^2\right),
 \end{align*}
with $C'$ and $C''$ depending only on $n$, $\evmin$, $\evmax$ and $\hexp$.

For $i=n$, we observe that, from \eqref{eq:pdexi}:
\begin{align*}
 \left(\diagmatrix_y\right)_{nn}\partialxifi{nn}^2\uxif+\sum_{j=1}^{n-1}\left(\diagmatrix_y\right)_{jj}\left(\partialxifi{jj}^2\uxif\right)^{\pm}_{\rho}+ \fxif^{\pm}_{y,\rho}=\;&-\sum_{j=1}^{n-1}\left(\diagmatrix_y\right)_{jj}\left(\partialxifi{jj}^2\uxif-\left(\partialxifi{jj}^2\uxif\right)^{\pm}_{\rho}\right)\\
 &-\left(\fxif_y-\fxif^{\pm}_{y,\rho}\right)
\end{align*}
in $B_R^{+}\cup B_R^{-}$, where $\left(\partialxifi{jj}^2\uxif\right)^{\pm}_{\rho}$ denotes $\left(\partialxifi{jj}^2\uxif\right)^{+}_{\rho}=\frac{1}{|B_{\rho}^{+}|}\int_{B_{\rho}^{+}}\partialxifi{jj}^2\uxif\dz$  in $B_{\rho}^{+}$ and $\left(\partialxifi{jj}^2\uxif\right)^{-}_{\rho}=\frac{1}{|B_{\rho}^{-}|}\int_{B_{\rho}^{-}}\partialxifi{jj}^2\uxif\dz$  in $B_{\rho}^{-}$, for $j=1,\ldots,n$. Analogous notation holds for $\fxif^{\pm}_{y,\rho}$. Setting $\lambda^{\pm}:=-\sum_{j=1}^{n-1}\left(\partialxifi{jj}^2\uxif\right)^{\pm}_{\rho}-\fxif^{\pm}_{y,\rho}$, we can apply Lemma \ref{lem:convex} and the result for $j=1,\ldots,n-1$ to conclude the proof:
\begin{align*}
 \int_{B_{\rho}^{\pm}}\left|\left(\partialxifi{nn}^2\wxif\right)^{\pm}(\Bz)-\left(\partialxifi{nn}^2 \wxif\right)^{\pm}_{\rho}\right|^2\dz&\leq  \int_{B_{\rho}^{\pm}}\left|\left(\partialxifi{nn}^2\wxif\right)^{\pm}(\Bz)-\lambda^{\pm}\right|^2\dz\\
 &\leq C'''\rho^{n+2\hexp}M_R + \int_{B_{\rho}^{\pm}}\left|\fxif_y^{\pm}-\fxif^{\pm}_{y,\rho}\right|^2\dz\\
 &\leq C'''\rho^{n+2\hexp}M_R + \rho^{n+2\hexp}\sholderpwbr{\fxif_y}{\rho}^2\\
 &\leq (C'''+1)\rho^{n+2\hexp}M_R,
\end{align*}
with $C'''=C'''(n,\evmin,\evmax,\hexp)$.
\end{proof}

\begin{theorem}\label{thm:hessbound}
 Let $\uxif$ be a solution to \eqref{eq:modelpbxi} in $B_R=B_R(\xiffix)$, with $\gxif\equiv 0$. Then
 \begin{equation}\label{eq:hessbound}
  \sholder{\left(\tilde{\jacsl}^2\uxif\right)^{\pm}}{B_{R/2}^{\pm}(\xiffix)}\leq C\left(\frac{1}{R^{2+\hexp}}\cont{\uxif}{B_R(\xiffix)}+\frac{1}{R^{\hexp}}\cont{\fxif_y^{\pm}}{B_R^{\pm}(\xiffix)}+\sholder{\fxif_y^{\pm}}{B_R^{\pm}(\xiffix)}\right),
 \end{equation}
where $C=C(n,\evmin,\evmax,\hexp)$ is \emph{independent} of the center $\xiffix$ of $B_R$, and overall it is independent of $J\in\bbN$ and $\By\in\pspace_J$.
\end{theorem}
\begin{proof}(On the lines of the proof of Thm. 6.2.11 in \cite{WYW}.) For $\xif\in\gammaxif_y\cap B_{R/2}(\xiffix)$ and $0<\rho\leq\frac{R}{4}$, Theorem \ref{thm:bdpoisson2} implies:
 \begin{equation}\label{eq:d2bdry}
\int_{B_{\rho}^{\pm}(\xif)}\left|\left(\tilde{\jac}^2 \uxif\right)^{\pm} - \left(\tilde{\jac}^2 \uxif\right)^{\pm}_{B_{\rho}(\xif)}\right|^2\dz \leq C'\rho^{n+2\hexp}M_{R/2}\leq C'\rho^{n+2\hexp}M_{R},
 \end{equation}
 with $C'=C'(n,\evmin,\evmax,\hexp)$ (where again $(\cdot)^{\pm}_{B_{\rho}^{\pm}(\xif)}$ denotes the mean on $B_{\rho}^{+}$ and $B_{\rho}^{-}$).
 
 If instead $\xif\in B_{R/2}(\xiffix)$ but $\xif\notin \gammaxif_y$, we denote $\xif':=(\tilde{x}_{y,1},\ldots,\tilde{x}_{y,n-1},0)\in\gammaxif_y$, and distinguish two cases: $0<\tilde{x}_{y,n}<\frac{R}{4}$ and $\frac{R}{4}\leq \tilde{x}_{y,n} <\frac{R}{2}$. In the first case, we consider two subcases: $\tilde{x}_{y,n}\leq\rho\leq\frac{R}{4}$ and $0<\rho<\tilde{x}_{y,n}$.
 
 If $0<\tilde{x}_{y,n}<\frac{R}{4}$ and $\tilde{x}_{y,n}\leq\rho\leq\frac{R}{4}$, then $B_{\rho}(\xif)\cap B_{R/2}(\xiffix)\subset B_{2\rho}(\xif')$, and, using \eqref{eq:d2bdry}, we can write:
 \begin{align}
  &\int_{\left(B_{\rho}(\xif)\cap B_{R/2}(\xiffix)\right)^{\pm}}\left|\left(\tilde{\jac}^2 \uxif\right)^{\pm} - \left(\tilde{\jac}^2 \uxif\right)^{\pm}_{B_{\rho}(\xif)\cap B_{R/2}(\xiffix)}\right|^2\dz \nonumber\\
  &\leq   \int_{B_{2\rho}(\xif')^{\pm}}\left|\left(\tilde{\jac}^2 \uxif\right)^{\pm} - \left(\tilde{\jac}^2 \uxif\right)^{\pm}_{B_{2\rho}(\xif')}\right|^2\dz \leq  C'\rho^{n+2\hexp}M_{R}.\label{eq:intcase1}
 \end{align}

  If $0<\tilde{x}_{y,n}<\frac{R}{4}$ and $0<\rho<\tilde{x}_{y,n}$, then  $B_{\rho}(\xif)$ is in the interior, meaning it does not cross the interface $\gammaxif_y$. Therefore, using the analogue of Theorem \ref{thm:bdpoisson1} for the interior \cite[Thm. 6.2.6]{WYW} (where it can be checked, as for the interface case, that the constants in the bound are independent of the center of the ball, of $J\in\bbN$ and $\By\in\pspace_J$), we have:
  \begin{align*}
  & \int_{B_{\rho}(\xif)}\left|\tilde{\jac}^2 \uxif - \left(\tilde{\jac}^2 \uxif\right)_{B_{\rho}(\xif)}\right|^2\dz \\
  &\leq C''\rho^{n+2\hexp}\left(\frac{1}{\tilde{x}_{y,n}^{n+2\hexp}}\int_{B_{\tilde{x}_{y,n}}(\xif)}\left|\left(\tilde{\jac}^2 \uxif\right) - \left(\tilde{\jac}^2 \uxif\right)_{B_{\tilde{x}_{y,n}}(\xif)}\right|^2\dz + \sholder{\fxif_y}{B_{\tilde{x}_{y,n}}(\xif)}\right)\\
  &\leq C'C'' \rho^{n+2\hexp}M_R,
  \end{align*}
 with $C''=C''(n,\evmin,\evmax,\hexp)$. In the last step we have used \eqref{eq:intcase1} with $\rho=\tilde{x}_{y,n}$ (as $B_{\tilde{x}_{y,n}}(\xif)\subset B_{2\tilde{x}_{y,n}}(\xif')$).
 
 Altogether, if $0<\tilde{x}_{y,n}<\frac{R}{4}$, then, for $0<\rho<\frac{R}{4}$:
 \begin{equation}\label{eq:forcampnorm1}
  \int_{\left(B_{\rho}(\xif)\cap B_{R/2}(\xiffix)\right)^{\pm}}\left|\left(\tilde{\jac}^2 \uxif\right)^{\pm} - \left(\tilde{\jac}^2 \uxif\right)^{\pm}_{B_{\rho}(\xif)\cap B_{R/2}(\xiffix)}\right|^2\dz \leq C_1 \rho^{n+2\hexp}M_R,
 \end{equation}
with $C_1=C_1(n,\evmin,\evmax,\hexp)$.

If instead $\frac{R}{4}\leq \tilde{x}_{y,n} <\frac{R}{2}$ and $0<\rho\leq\frac{R}{4}$, then either $B_{\rho}(\xif)\subset B_{R/4}(\xif)\subset B_{3R/4}^{+}(\xiffix)\subset B_R^{+}(\xiffix)$, or $B_{\rho}(\xif)\subset B_{R/4}(\xif)\subset B_{3R/4}^{-}(\xiffix)\subset B_R^{-}(\xiffix)$. In the first case, the analogous of Theorem \ref{thm:bdpoisson2} for the interior (where again it can be checked that the constants in the bound are independent of the center of the ball, of $J\in\bbN$ and $\By\in\pspace_J$) implies:
\begin{align}
  \int_{B_{\rho}(\xif)\cap B^{+}_{R/2}(\xiffix)}\left|\tilde{\jac}^2 \uxif - \left(\tilde{\jac}^2 \uxif\right)_{B_{\rho}(\xif)\cap B^{+}_{R/2}(\xiffix)}\right|^2\dz &\leq \int_{B_{\rho}(\xif)}\left|\tilde{\jac}^2 \uxif - \left(\tilde{\jac}^2 \uxif\right)_{B_{\rho}(\xif)}\right|^2\dz\nonumber\\
  & \leq C_2 \rho^{n+2\hexp}M_R,\nonumber
\end{align}
with $C_2=C_2(n,\evmin,\evmax,\hexp)$. An analogous estimate holds in the second case. Considering this last bound together with \eqref{eq:forcampnorm1}, and using Lemma \ref{lem:campanato} with $\mu=n+2\hexp$, $p=2$ and $\lambda=\frac{R}{4}$, we finally obtain, for $C_3=C_3(n,\hexp)$:
\begin{align*}
  \sholder{\left(\tilde{\jacsl}^2\uxif\right)^{\pm}}{B_{R/2}^{\pm}(\xiffix)}\leq C_3\left|\left(\tilde{\jacsl}^2\uxif\right)^{\pm}\right|^{\left(\tfrac{1}{4}\right)}_{2,n+2\hexp;B_{R/2}^{\pm}(\xiffix)}\leq C M_R^{\frac{1}{2}},
\end{align*}
from which \eqref{eq:hessbound} follows, with $C=C(n,\evmin,\evmax,\hexp)$. 
\end{proof}

The previous result states the local estimate for the Poisson equation in case of homogeneous transmission conditions. We are now in the position to consider the case $\gxif\neq 0$:
\begin{theorem}\label{thm:withg}
 Let $\uxif$ be a solution to \eqref{eq:modelpbxi} in $B_R=B_R(\xiffix)$. Then
 \begin{align*}
    \sholder{\left(\tilde{\jacsl}^2\uxif\right)^{\pm}}{B_{R/2}^{\pm}(\xiffix)}\leq\; & C\left(\frac{1}{R^{2+\hexp}}\cont{\uxif}{B_R(\xiffix)}+\frac{1}{R^{\hexp}}\cont{\fxif_y^{\pm}}{B_R^{\pm}(\xiffix)}+\sholder{\fxif_y}{B_R^{\pm}(\xiffix)}\right.\\
    & \left. +\frac{1}{R^{1+\hexp}}\cont{\gxif_y}{\gammaxif_y}+\sholder{\gxif_y}{\gammaxif_y}\right),
 \end{align*}
where $C=C(n,\evmin,\evmax,\hexp)$ is \emph{independent} of the center $\xiffix$ of $B_R$, of $J\in\bbN$ and $\By\in\pspace_J$.
\end{theorem}
\begin{proof}
 For this proof we use a similar argument as in \cite[pp. 124--125]{GT}.
 
 Consider a nonnegative function $\eta\in C_0^2(\bbR^{n-1})$, such that $\int_{\bbR^{n-1}}\eta(\Bz')\dz'=1$. Lemma 6.38 in \cite{GT} ensures that $\gxif_y$ can be extended outside $\gammaxif_y$ in such a way that its extension belongs to $C_0^{1,\hexp}(\bbR^{n-1})$. With some abuse of notation, we still denote by $\gxif_y$ this extension.
 
 We define $\tilde{\psi}_1$ and $\tilde{\psi}_2$ as the functions fulfilling the following equalities:
 \begin{align}
  \left(\diagmatrix_y\right)^{+}_{nn}\tilde{\psi}_1(\xif)&=\frac{1}{2}\tilde{x}_{y,n}\int_{\bbR^{n-1}}\gxif_y(\xif'-\tilde{x}_{y,n}\Bz')\eta(\Bz')\dz',\label{eq:psi1}\\
  \left(\diagmatrix_y\right)^{-}_{nn}\tilde{\psi}_2(\xif)&=-\frac{1}{2}\tilde{x}_{y,n}\int_{\bbR^{n-1}}\gxif_y(\xif'-\tilde{x}_{y,n}\Bz')\eta(\Bz')\dz',\label{eq:psi2}
 \end{align}
where $\xif'=(\tilde{x}_{y,1}',\ldots,\tilde{x}_{y,n-1}')$. It can be checked (see \eqref{eq:psic2a}) that $\tilde{\psi}_1,\tilde{\psi}_2\in C^{2,\hexp}(\bbR^n)$, and that:
\begin{align*}
 &\tilde{\psi}_1(\xif',0)=\tilde{\psi}_2(\xif',0)=0,\\
 &   \left(\diagmatrix_y\right)^{+}_{nn}\dfrac{\partialxif}{\partialxif\tilde{x}_{y,n}}\tilde{\psi}_1(\xif',0)-\left(\diagmatrix_y\right)^{-}_{nn}\dfrac{\partialxif}{\partialxif\tilde{x}_{y,n}}\tilde{\psi}_2(\xif',0)=\gxif_y(\xif').
\end{align*}
The solution $\uxif$ to \eqref{eq:modelpbxi} can be decomposed as $\uxif=\vxif+\psixif$, where 
\begin{equation*}
 \psixif|_{B_R^{+}}=\tilde{\psi}_1,\quad
 \psixif|_{B_R^{-}}=\tilde{\psi}_2.
\end{equation*}
Then $\vxif$ fulfills
\begin{subequations}
\begin{align*}[left=\empheqlbrace]
 & -\xifgrad \cdot \left(\diagmatrix_y \xifgrad \vxif\right) = \fxif_y + \xifgrad \cdot \left(\diagmatrix_y \xifgrad \psixif\right),\quad \text{in }B_R^{+}\cup B_R^{-},\\
 & \llbracket \vxif \rrbracket_{\gammaxif_y} =0,\quad\Big\llbracket \diagmatrix_y\dfrac{\partialxif \vxif}{\partialxif \nxif}\Big\rrbracket_{\gammaxif_y} = 0.
\end{align*}
 \end{subequations}
 Applying Theorem \ref{thm:hessbound} to $\vxif$ (and with the help of Lemma \ref{lem:diagequiv}), we infer:
 \begin{align}
   \sholder{\left(\tilde{\jac}^2\vxif\right)^{\pm}}{B_{R/2}^{\pm}(\xiffix)}\leq&\, C'\frac{1}{R^{2+\hexp}}\cont{\vxif}{B_R(\xiffix)}%+\cont{\psixif}{B_R(\xiffix)}\right)\nonumber\\
    + C'\frac{1}{R^{\hexp}}\left(\cont{\fxif_y^{\pm}}{B_R^{\pm}(\xiffix)}+\cont[2]{\psixif}{B_R^{\pm}(\xiffix)}\right)\nonumber\\
   & + C'\left(\sholder{\fxif_y}{B_R^{\pm}(\xiffix)}+\sholderk{\psixif}{B_R^{\pm}(\xiffix)}{2}\right),\label{eq:vxiest}
 \end{align}
 with $C'=C'(n,\evmin,\evmax,\beta)$.
 
% From \eqref{eq:psi1}, we can see that
% \begin{equation*}
%  \cont{\tilde{\psi}_1}{B_R^{+}(\xiffix)}\leq\frac{R}{2\evmin}\cont{\gxif_y}{\gammaxif_y},
% \end{equation*}
% and similarly for $ \cont{\tilde{\psi}_2}{B_R^{-}(\xiffix)}$. Furthermore, since, 

Denoting by $\xifgradtg$ the gradient with respect to the first $n-1$ components of the argument, and by $\partialxifi{i}$ the derivative with respect to the $i^{th}$ component, in $B_R^{+}(\xiffix)$ we have:% along $\gammaxif_y$,
\begin{align*}
 \left(\diagmatrix_y\right)^{+}_{nn}\partialxifi{ij}^2\psixif&=\frac{1}{2}\int_{\bbR^{n-1}}\partialxifi{i}\gxif_y(\xif'-\tilde{x}_{y,n}\Bz')\partialxifi{j}\eta(\Bz')\dz',\quad \text{for }i,j\neq n,\\
 \left(\diagmatrix_y\right)^{+}_{nn}\partialxifi{in}^2\psixif&=-\frac{1}{2}\int_{\bbR^{n-1}}\Bz'\cdot\xifgradtg\gxif_y(\xif'-\tilde{x}_{y,n}\Bz')\partialxifi{i}\eta(\Bz')\dz',\quad \text{for }i\neq n,\\
 \left(\diagmatrix_y\right)^{+}_{nn}\partialxifi{nn}^2\psixif&=\frac{1}{2}\int_{\bbR^{n-1}}\Bz'\cdot\xifgradtg\gxif_y(\xif'-\tilde{x}_{y,n}\Bz')\left[(n-2)\eta(\Bz')+\Bz'\cdot\xifgradtg\eta(\Bz')\right]\dz',
\end{align*}
and thus
\begin{subequations}
 \begin{align}
  \cont{\jac^2\psixif}{B_R^{+}(\xiffix)}&\leq C_{\eta}\scont[1]{\gxif_y}{\gammaxif_y},\\
  \sholderk{\psixif}{B_R^{+}(\xiffix)}{2}&\leq C_{\eta}\sholderk{\gxif_y}{\gammaxif_y}{1}.
 \end{align}\label{eq:psic2a}
\end{subequations}
Analogous results hold for the norms on $B_R^{-}(\xiffix)$. The constant $C_{\eta}$ depends on the norms of $\eta$ on $\gammaxif_y$, and thus, in principle, it could depend on $\By\in\pspace_J$ and $J\in\bbN$. However, if, for every $\xiffix\in\gammaxif_y$ considered, we use, in $B_R(\xiffix)$, the same function $\eta$ translated so that it is centered in $\xiffix$, then $C_{\eta}$ is independent of $J\in\bbN$ and of $\By\in\pspace_J$. Combining \eqref{eq:psic2a} with \eqref{eq:vxiest}, and using the interpolation inequalities (cf. \cite[Cor. 1.2.1]{WYW}) to bound $\cont[1]{\gxif}{\gammaxif_y}$, we gather the desired estimate:
\begin{align*}
\sholder{\left(\tilde{\jacsl}^2\uxif\right)^{\pm}}{B_{R/2}^{\pm}(\xiffix)}\leq\; & C \left(\frac{1}{R^{2+\hexp}}\cont{\uxif}{B_R(\xiffix)}+\frac{1}{R^{\hexp}}\cont{\fxif_y^{\pm}}{B_R^{\pm}(\xiffix)}+\frac{1}{R^{\hexp}}\scont[1]{\gxif_y}{\gammaxif_y}\right)\\
& + C\left(\sholderk{\gxif_y}{\gammaxif_y}{1} + \sholder{\fxif_y^{\pm}}{B_R^{\pm}(\xiffix)}\right)\\
\leq\; & C \left(\frac{1}{R^{2+\hexp}}\cont{\uxif}{B_R(\xiffix)}+\frac{1}{R^{1+\hexp}}\cont{\gxif_y}{\gammaxif_y}+\sholderk{\gxif_y}{\gammaxif_y}{1}\right)\\
&+C\left(\frac{1}{R^{\hexp}}\cont{\fxif_y^{\pm}}{B_R^{\pm}(\xiffix)}+ \sholder{\fxif_y^{\pm}}{B_R^{\pm}(\xiffix)}\right),
\end{align*}
with $C=C(n,\evmin,\evmax,\hexp,\eta)$ \emph{independent} of $\xiffix$, of $J\in\bbN$ and of $\By\in\pspace_J$.
\end{proof}\label{appendix:spacereg}
\end{appendices}

% Numerical experiments MLMC: remark that the exponent in the log term still influences the distribution of the samples, and it is taken into account. (~~The -\varepsilon convergence rate was used for existence result~~)

% \ls{TO DO. Remark that if the highest order coefficient is continuous across the interface, then the solution is in $C^{1,\alpha}$ (proof by triangle inequality) and it would be interesting to test QMC.}

%%% subsect: finite element cvg (state also square integrability on finite-dimensional subspaces). State assumptions on mesh! (quasi-uniform? See Schatz paper)

%%%%%%%%%%%%%%%%%%%%%%%%%%%%%%%%%%%%%%%%%%%%%%%%%%%%%%%%%%%%%%%

\bibliography{references.bib}

\begin{thebibliography}{10}

\bibitem{BNT}
{\sc I.~Babu{\v{s}}ka, F.~Nobile, and R.~Tempone}, {\em A stochastic
  collocation method for elliptic partial differential equations with random
  input data}, SIAM Journal on Numerical Analysis, 45 (2007), pp.~1005--1034.

\bibitem{BTZ}
{\sc I.~Babuska, R.~Tempone, and G.~E. Zouraris}, {\em {Galerkin finite element
  approximations of stochastic elliptic partial differential equations}}, SIAM
  Journal on Numerical Analysis, 42 (2004), pp.~800--825.

\bibitem{BBP}
{\sc C.~Bacuta, J.~H. Bramble, and J.~E. Pasciak}, {\em {New interpolation
  results and applications to finite element methods for elliptic boundary
  value problems}}, Journal of Numerical Mathematics, 9 (2001), pp.~179--198.

\bibitem{BSZ}
{\sc A.~Barth, C.~Schwab, and N.~Zollinger}, {\em {Multi-level Monte Carlo
  finite element method for elliptic PDEs with stochastic coefficients}},
  Numerische Mathematik, 119 (2011), pp.~123--161.

\bibitem{BeNTT}
{\sc J.~Beck, F.~Nobile, L.~Tamellini, and R.~Tempone}, {\em {Convergence of
  quasi-optimal stochastic Galerkin methods for a class of PDEs with random
  coefficients}}, Computers \& Mathematics with Applications, 67 (2014),
  pp.~732--751.

\bibitem{B}
{\sc J.-P. Berenger}, {\em A perfectly matched layer for the absorption of
  electromagnetic waves}, Journal of Computational Physics, 114 (1994),
  pp.~185--200.

\bibitem{BGL}
{\sc A.~Bonito, J.-L. Guermond, and F.~Luddens}, {\em {Regularity of the
  Maxwell equations in heterogeneous media and Lipschitz domains}}, Journal of
  Mathematical Analysis and applications, 408 (2013), pp.~498--512.

\bibitem{BS}
{\sc S.~C. Brenner and R.~Scott}, {\em The mathematical theory of finite
  element methods}, vol.~15, Springer Science \& Business Media, 2008.

\bibitem{Caf}
{\sc R.~E. Caflisch}, {\em {Monte carlo and quasi-monte carlo methods}}, Acta
  numerica, 7 (1998), pp.~1--49.

\bibitem{CC}
{\sc C.~Canuto and T.~Kozubek}, {\em A fictitious domain approach to the
  numerical solution of {PDEs} in stochastic domains}, Numerische mathematik,
  107 (2007), pp.~257--293.

\bibitem{CNT}
{\sc J.~E. Castrillon-Candas, F.~Nobile, and R.~F. Tempone}, {\em Analytic
  regularity and collocation approximation for {PDEs} with random domain
  deformations}, Comput. Math. Appl., 71 (2016), pp.~1173--1197.

\bibitem{CherS}
{\sc A.~Chernov and C.~Schwab}, {\em {First order k-th moment finite element
  analysis of nonlinear operator equations with stochastic data}}, Mathematics
  of Computation, 82 (2013), pp.~1859--1888.

\bibitem{CCS}
{\sc A.~Chkifa, A.~Cohen, and C.~Schwab}, {\em {Breaking the curse of
  dimensionality in sparse polynomial approximation of parametric PDEs}},
  Journal de Math{\'e}matiques Pures et Appliqu{\'e}es, 103 (2015),
  pp.~400--428.

\bibitem{CGST}
{\sc K.~A. Cliffe, M.~B. Giles, R.~Scheichl, and A.~L. Teckentrup}, {\em
  {Multilevel Monte Carlo methods and applications to elliptic PDEs with random
  coefficients}}, Computing and Visualization in Science, 14 (2011), pp.~3--15.

\bibitem{CSZ}
{\sc A.~Cohen, C.~Schwab, and J.~Zech}, {\em {Shape Holomorphy of the
  stationary Navier-Stokes Equations}}, Report 2016-45, Seminar for Applied
  Mathematics, ETH Z\"urich, Switzerland.

\bibitem{CM}
{\sc F.~Collino and P.~Monk}, {\em The perfectly matched layer in curvilinear
  coordinates}, SIAM Journal on Scientific Computing, 19 (1998),
  pp.~2061--2090.

\bibitem{DW}
{\sc M.~D. Dettinger and J.~L. Wilson}, {\em {First order analysis of
  uncertainty in numerical models of groundwater flow part: 1. Mathematical
  development}}, Water Resources Research, 17 (1981), pp.~149--161.

\bibitem{DLGS}
{\sc J.~Dick, Q.~T.~L. Gia, and C.~Schwab}, {\em {Higher Order Quasi--Monte
  Carlo Integration for Holomorphic, Parametric Operator Equations}}, SIAM/ASA
  Journal on Uncertainty Quantification, 4 (2016), pp.~48--79.

\bibitem{DKLNS}
{\sc J.~Dick, F.~Y. Kuo, Q.~T. Le~Gia, D.~Nuyens, and C.~Schwab}, {\em {Higher
  order QMC Petrov--Galerkin discretization for affine parametric operator
  equations with random field inputs}}, SIAM Journal on Numerical Analysis, 52
  (2014), pp.~2676--2702.

\bibitem{DKS}
{\sc J.~Dick, F.~Y. Kuo, and I.~H. Sloan}, {\em {High-dimensional integration:
  the quasi-Monte Carlo way}}, Acta Numerica, 22 (2013), pp.~133--288.

\bibitem{DP}
{\sc J.~Dick and F.~Pillichshammer}, {\em {Digital nets and sequences:
  Discrepancy Theory and Quasi--Monte Carlo Integration}}, Cambridge University
  Press, 2010.

\bibitem{DKST}
{\sc T.~J. Dodwell, C.~Ketelsen, R.~Scheichl, and A.~L. Teckentrup}, {\em {A
  hierarchical multilevel Markov chain Monte Carlo algorithm with applications
  to uncertainty quantification in subsurface flow}}, SIAM/ASA Journal on
  Uncertainty Quantification, 3 (2015), pp.~1075--1108.

\bibitem{GP}
{\sc R.~Gantner and M.~Peters}, {\em {Higher Order Quasi-Monte Carlo for
  Bayesian Shape Inversion}}, Report 2015-31, Seminar for Applied Mathematics,
  ETH Z\"urich, Switzerland.

\bibitem{Gantner}
{\sc R.~N. Gantner}, {\em {A Generic C++ Library for Multilevel Quasi-Monte
  Carlo}}, in Proceedings of the Platform for Advanced Scientific Computing
  Conference, PASC '16, New York, NY, USA, 2016, ACM, pp.~11:1--11:12.

\bibitem{GT}
{\sc D.~Gilbarg and N.~S. Trudinger}, {\em Elliptic partial differential
  equations of second order}, Springer, 2015.

\bibitem{Giles08}
{\sc M.~B. Giles}, {\em {Multilevel monte carlo path simulation}}, Operations
  Research, 56 (2008), pp.~607--617.

\bibitem{Giles15}
\leavevmode\vrule height 2pt depth -1.6pt width 23pt, {\em {Multilevel Monte
  Carlo methods}}, Acta Numerica, 24 (2015), p.~259.

\bibitem{Glas}
{\sc P.~Glasserman}, {\em {Monte Carlo methods in financial engineering}},
  vol.~53, Springer Science \& Business Media, 2013.

\bibitem{GLRS}
{\sc J.~Guzm{\'a}n, D.~Leykekhman, J.~Rossmann, and A.~H. Schatz}, {\em
  {H{\"o}lder estimates for Green's functions on convex polyhedral domains and
  their applications to finite element methods}}, Numerische Mathematik, 112
  (2009), pp.~221--243.

\bibitem{GSS}
{\sc J.~Guzm{\'a}n, M.~S{\'a}nchez, and M.~Sarkis}, {\em {On the accuracy of
  finite element approximations to a class of interface problems}}, Mathematics
  of Computation, 85 (2016), pp.~2071--2098.

\bibitem{HL}
{\sc H.~Harbrecht and J.~Li}, {\em First order second moment analysis for
  stochastic interface problems based on low-rank approximation}, ESAIM:
  Mathematical Modelling and Numerical Analysis, 47 (2013), pp.~1533--1552.

\bibitem{HPS}
{\sc H.~Harbrecht, M.~Peters, and M.~Siebenmorgen}, {\em Analysis of the domain
  mapping method for elliptic diffusion problems on random domains}, Numerische
  Mathematik,  (2016), pp.~1--34.

\bibitem{HaSS}
{\sc H.~Harbrecht, R.~Schneider, and C.~Schwab}, {\em Sparse second moment
  analysis for elliptic problems in stochastic domains}, Numerische Mathematik,
  109 (2008), pp.~385--414.

\bibitem{Har}
{\sc Y.~Harness}, {\em {Low-Dimensional Spatial Embedding Method for Shape
  Uncertainty Quantification in Acoustic Scattering}}, arXiv preprint
  arXiv:1704.07727,  (2017).

\bibitem{He98}
{\sc S.~Heinrich}, {\em {Monte Carlo complexity of global solution of integral
  equations}}, Journal of Complexity, 14 (1998), pp.~151--175.

\bibitem{HeS}
{\sc S.~Heinrich and E.~Sindambiwe}, {\em {Monte Carlo complexity of parametric
  integration}}, Journal of Complexity, 15 (1999), pp.~317--341.

\bibitem{HSSS}
{\sc R.~Hiptmair, L.~Scarabosio, C.~Schillings, and C.~Schwab}, {\em Large
  deformation shape uncertainty quantification in acoustic scattering}, Report
  2015-31, Seminar for Applied Mathematics, ETH Z\"urich, Switzerland.
  \url{http://www.sam.math.ethz.ch/sam_reports/reports_final/reports2015/2015-31_rev1.pdf}.

\bibitem{HKMS}
{\sc N.~Hyvo\"onen, V.~Kaarnioja, L.~Mustonen, and S.~Staboulis}, {\em
  {Polynomial collocation for handling an inaccurately known measurement
  configuration in electrical impedance tomography}}, SIAM Journal on Applied
  Mathematics, 77 (2017), pp.~202--223.

\bibitem{JSSZ}
{\sc C.~Jerez-Sanchez, C.~Schwab, and J.~Zech}, {\em {Electromagnetic Wave
  Scattering by Random Surfaces: Shape Holomorphy}}, Report 2016-49, Seminar
  for Applied Mathematics, ETH Z\"urich, Switzerland.

\bibitem{Joc}
{\sc F.~Jochmann}, {\em {An Hs-regularity result for the gradient of solutions
  to elliptic equations with mixed boundary conditions}}, Journal of
  mathematical analysis and applications, 238 (1999), pp.~429--450.

\bibitem{Kras}
{\sc J.~P. Krasovski}, {\em {Isolation of singularities of the Green's
  function}}, Mathematics of the USSR-Izvestiya, 1 (1967), p.~935.

\bibitem{LMWS}
{\sc J.~Li, J.~M. Melenk, B.~Wohlmuth, and J.~Zou}, {\em Optimal a priori
  estimates for higher order finite elements for elliptic interface problems},
  Applied numerical mathematics, 60 (2010), pp.~19--37.

\bibitem{Mir}
{\sc C.~Miranda}, {\em Partial differential equations of elliptic type},
  vol.~2, Springer Science \& Business Media, 2012.

\bibitem{Ned}
{\sc J.-C. N{\'e}d{\'e}lec}, {\em Acoustic and electromagnetic equations:
  integral representations for harmonic problems}, vol.~144, Springer Science
  \& Business Media, 2001.

\bibitem{NTW}
{\sc F.~Nobile, R.~Tempone, and C.~G. Webster}, {\em {A sparse grid stochastic
  collocation method for partial differential equations with random input
  data}}, SIAM Journal on Numerical Analysis, 46 (2008), pp.~2309--2345.

\bibitem{NCSM}
{\sc A.~Nouy, A.~Clement, F.~Schoefs, and N.~Mo{\"e}s}, {\em An extended
  stochastic finite element method for solving stochastic partial differential
  equations on random domains}, Computer Methods in Applied Mechanics and
  Engineering, 197 (2008), pp.~4663--4682.

\bibitem{NSM}
{\sc A.~Nouy, F.~Schoefs, and N.~Mo{\"e}s}, {\em {X-SFEM}, a computational
  technique based on {X-FEM} to deal with random shapes}, European Journal of
  Computational Mechanics/Revue Europ{\'e}enne de M{\'e}canique Num{\'e}rique,
  16 (2007), pp.~277--293.

\bibitem{Sav}
{\sc G.~Savar{\'e}}, {\em {Regularity results for elliptic equations in
  Lipschitz domains}}, Journal of Functional Analysis, 152 (1998),
  pp.~176--201.

\bibitem{LSthesis}
{\sc L.~Scarabosio}, {\em Shape uncertainty quantification for scattering
  transmission problems}, PhD thesis, ETH Z{\"u}rich, 2016.
\newblock Diss. No. 23574.
  \url{{http://e-collection.library.ethz.ch/eserv/eth:49652/eth-49652-02.pdf}}.

\bibitem{Schatz98}
{\sc A.~Schatz}, {\em {Pointwise error estimates and asymptotic error expansion
  inequalities for the finite element method on irregular grids: Part I. Global
  estimates}}, Mathematics of Computation of the American Mathematical Society,
  67 (1998), pp.~877--899.

\bibitem{SS}
{\sc C.~Schillings and C.~Schwab}, {\em Sparse, adaptive {Smolyak} quadratures
  for {Bayesian} inverse problems}, Inverse Problems, 29 (2013), p.~065011.

\bibitem{SG}
{\sc C.~Schwab and C.~J. Gittelson}, {\em Sparse tensor discretizations of
  high-dimensional parametric and stochastic {PDEs}}, Acta Numerica, 20 (2011),
  pp.~291--467.

\bibitem{SZ}
{\sc J.~Sokolowski and J.-P. Zolesio}, {\em {Introduction to Shape
  Optimization}}, Springer, 1992.

\bibitem{TX}
{\sc D.~M. Tartakovsky and D.~Xiu}, {\em Stochastic analysis of transport in
  tubes with rough walls}, Journal of Computational Physics, 217 (2006),
  pp.~248--259.

\bibitem{WYW}
{\sc Z.~Wu, J.~Yin, and C.~Wang}, {\em {Elliptic \& parabolic equations}},
  World Scientific, 2006.

\bibitem{XH}
{\sc D.~Xiu and J.~S. Hesthaven}, {\em High-order collocation methods for
  differential equations with random inputs}, SIAM Journal on Scientific
  Computing, 27 (2005), pp.~1118--1139.

\bibitem{XK}
{\sc D.~Xiu and G.~E. Karniadakis}, {\em {The Wiener--Askey polynomial chaos
  for stochastic differential equations}}, SIAM journal on scientific
  computing, 24 (2002), pp.~619--644.

\bibitem{XT}
{\sc D.~Xiu and D.~M. Tartakovsky}, {\em Numerical methods for differential
  equations in random domains}, SIAM Journal on Scientific Computing, 28
  (2006), pp.~1167--1185.

\bibitem{ZWGB}
{\sc G.~Zhang, C.~G. Webster, M.~Gunzburger, and J.~Burkardt}, {\em
  {Hyperspherical Sparse Approximation Techniques for High-Dimensional
  Discontinuity Detection}}, SIAM Review, 58 (2016), pp.~517--551.

\end{thebibliography}
\bibliographystyle{siam}

\end{document}